\documentclass[reqno, a4paper]{amsart}
\usepackage{amsmath, amssymb, enumerate, bbm, xcolor}
\usepackage{textcmds}
\usepackage{datetime}
\usepackage[colorlinks=true,linkcolor=blue, citecolor=red]{hyperref}
\usepackage[nameinlink,capitalize]{cleveref}
\usepackage{dsfont}

\crefname{equation}{}{} 
\numberwithin{equation}{section}

\allowdisplaybreaks

\newtheorem{theo}{Theorem}[section]
\newtheorem{coro}[theo]{Corollary}
\newtheorem{lemm}[theo]{Lemma}
\newtheorem{prop}[theo]{Proposition}

\theoremstyle{definition}
\newtheorem{defi}[theo]{Definition}
\newtheorem{exam}[theo]{Example}
\newtheorem{rema}[theo]{Remark}

\newtheorem{assumption}[theo]{Assumption}

\newcommand{\B}{\mathbb B}

\newcommand{\D}{\mathbb D}
\newcommand{\E}{\mathbb E}
\newcommand{\F}{\mathbb F}
\newcommand{\bN}{\mathbb N}
\renewcommand{\P}{\mathbb P}
\newcommand{\oP}{\overline{\mathbb P}}

\newcommand{\R}{\mathbb R}

\newcommand{\1}{\mathbbm{1}}

\newcommand{\od}{\mathrm{d}}

\newcommand{\cB}{\mathcal B}
\newcommand{\cC}{\mathcal C}
\newcommand{\cD}{\mathcal D}
\newcommand{\cF}{\mathcal F}

\newcommand{\cK}{\mathcal K}
\newcommand{\cL}{\mathcal L}
\newcommand{\cN}{\mathcal N}
\newcommand{\cP}{\mathcal P}
\newcommand{\cR} {{\mathcal R}}

\newcommand{\Oz}{\Omega}

\newcommand{\tz}{\theta}

\newcommand{\sz}{\sigma}

\newcommand{\sign}{{\rm sign}}
\DeclareMathOperator*{\esssup}{ess\,sup}

\newcommand{\sptext}[3]{\hspace{#1 em}\mbox{#2}\hspace{#3 em}}
\newcommand{\fz}{\infty}
\newcommand{\chf}{\mathds{1}}

\newcommand{\BMO}{\mathrm{BMO}}
\newcommand{\hs}{\mathrm{HS}}
\newcommand{\iue}{(I_u^\varepsilon)}
\newcommand{\cadap}{{\rm c,adap}}
\newcommand{\br}{{\bf r}}

\begin{document}

\title[Coupling of SDEs]{Regularity of stochastic differential equations on the Wiener space by coupling}

\author{Stefan Geiss}
\address{Department of Mathematics and Statistics, P.O.Box 35, FI-40014 University of Jyv\"asky\-l\"a, Finland}
\email{stefan.geiss@jyu.fi}
\author{Xilin Zhou}
\address{Department of Mathematics and Statistics, P.O.Box 35, FI-40014 University of Jyv\"asky\-l\"a, Finland}
\email{xilin.j.zhou@jyu.fi}



\begin{abstract}
Using the coupling method introduced in \cite{Geiss:Ylinen:21}, we investigate regularity properties
of stochastic differential equations, where we consider the Lipschitz case in $\R^d$ and
allow for H\"older continuity of the diffusion coefficient of scalar valued stochastic differential equations.
Two cases of the coupling method are of special interest: The uniform coupling to treat 
the Malliavin Sobolev space $\D_{1,2}$ and real interpolation spaces, and secondly a cut-off coupling to treat the $L_p$-variation
of backward stochastic differential equations where the forward process is the investigated stochastic 
differential equation. 
\end{abstract}

\maketitle
\tableofcontents


\section{Introduction}

We consider the regularity of $d$-dimensional stochastic differential equations (SDEs) of type
\begin{equation}\label{eqn:intro:SDE_random_coefficients}
        X_s^{t,\xi} 
        = \xi + \int_t^s b(u, X^{t,\xi})\,\od u
        + \int_t^s \sigma(u, X^{t,\xi})\,\od W_u
\end{equation}
with random coefficients and driven by an $N$-dimensional Brownian motion,
where $d,N\ge 1$ are fixed. The investigation is based on the coupling method introduced
in \cite{Geiss:Ylinen:21}. Although the title of  \cite{Geiss:Ylinen:21} refers to {\it decoupling}
\cite{Geiss:Ylinen:21} actually proposes a {\it coupling by a partial decoupling}, so that we use  
the description {\it coupling} in the following.
Originally this coupling method was designed to investigate the $L_p$-regularity in time of solutions
to backward stochastic differential equations (BSDEs). For this the starting point was \cite{Geiss:Geiss:Gobet:2012}
on the Wiener space which was extended in \cite{CGeiss:Steinicke:2016} to BSDEs driven by L\'evy processes. The
full framework was developed in \cite{Geiss:Ylinen:21} 
including connections to Malliavin calculus, in particular developing further
descriptions from \cite{Hirsch:1999,Geiss:Toivola:2015} of classical Malliavin Besov spaces by isotropic couplings.
The advantage of the approach \cite{Geiss:Ylinen:21} is that the proposed method is very robust and only needs a minimal set of 
More precisely, the method allows for 
\medskip

\begin{enumerate}[--]

\item a characterization for the Malliavin Sobo\-lev space
      $\D_{1,2}$, see \cref{sec:characterization_D12};

\item a characterization whether a random variable belongs to the Malliavin Sobo\-lev space
      $\D_{1,2}$, where additionally the Malliavin derivative satisfies
      \[ \esssup_{s\in [0,T]} \| D_s \xi \|_{L_2}<\infty, \]
      see \cite[Theorem 4.22]{Geiss:Ylinen:21};

\item a characterization whether a random variable belongs to the Malliavin Besov space
      $(L_p,\D_{1,p})_{\theta,q}$ obtained by real interpolation, 
      see \cite[Theorem 4.16]{Geiss:Ylinen:21};

\item an investigation for the $L_p$-path-regularity of backward stochastic differential equations,
      see \cite[Section 6.5]{Geiss:Ylinen:21}.
\end{enumerate}
\medskip

The plan of the article is as follows:
After some preliminaries in \cref{sec:preliminaries} we introduce in
\underline{\cref{sec:transference}} the transference method, in \underline{\cref{sec:coupling}} the coupling method,
and in \underline{\cref{sec:Besov}} related Besov spaces.
\bigskip

\underline{\cref{sec:Lipschitz}} deals with SDEs in the Lipschitz setting.
The Malliavin differentiability of SDEs in the Markovian setting under Lipschitz type assumptions was studied for example in
\cite{Nualart:06}, \cite{Imkeller:Reis:Salkeld:2019}, and \cite{dosReis:Wilde:2024},  the Sobolev differentiability with respect
to the starting value in \cite{Xie:Zhang:2016}.
Recently, a path-dependent (i.e. non Markovian) setting was considered in \cite{Lee_etal:2023} to obtain
Malliavin differentiability under continuity assumptions on the Fr\'echet derivative of the drift and diffusion coefficient.
Using the coupling method general regularity results for non-random Markovian Lipschitz coefficients
and for $p\ge 2$ were obtained in \cite[Theorem 3]{Geiss:Geiss:Gobet:2012} and \cite[Section 4.7.1]{Geiss:Ylinen:21}.
\medskip

In \cref{sec:Lipschitz} we extent the above results by the coupling method in the following directions:
\begin{enumerate}[--]
\item We allow general predictable random coefficients and random initial conditions, 
      where we also replace the uniform Lipschitz condition on the drift term by a bounded mean oscillation condition. The 
      latter is possible due to Fefferman's inequality and has its origin in the treatment of the generator term
      in backward stochastic differential equations (see \cite[Section 5.4]{Geiss:Ylinen:21}).

\item The coefficients of the SDE are fully path-dependent.
\end{enumerate}

In this framework the general coupling result is \cref{statement:Lipschitz_case_path}
(where we consider $L_p$-moments for all $p\in (0,\infty)$ as in applications (see \cref{sec:application:BSDE})
 small values of $p$ occur due to H\"older assumptions on the data of the BSDE).
In \cref{statement:Malliavin_differentiability} and \cref{statement:Malliavin_differentiability_II}
we consider the Malliavin differentiability and in 
\cref{statement:real_interpolation_SDE_Lipschitz} fractional smoothness in terms of real interpolation spaces.
To do so, we introduce a fractional potential $(U,V)$ for the coefficients $(b,\sigma)$ of the stochastic differential equation, 
which fits the setting of controlled diffusions, see \cref{statement:controlled_diffusions}.
\bigskip

In \underline{\cref{sec:counterexample}} we show that, for H\"older continuous diffusion coefficients, we do not have
Malliavin differentiability of the SDE in general. Our example is based on Cieselski's characterization of H\"older
functions in terms of expansions into Schauder functions.
For example, this is in contrast to the special Heston volatility model, where 
Malliavin differentiability was shown in
\cite{Alos:Ewald:2008}.
\bigskip

In \underline{\cref{sec:Hoelder}} we study the case when the diffusion coefficient is H\"older continuous in dimension one.
Strong approximations of these SDEs are of particular interest, for recent results and references see \cite{Do:Ngo:Pho:2024} and \cite{MG:LY:2024}.
Existence, uniqueness, and the non confluence property have been studied, for example, in
\cite{Fang:Zhang:2005},
\cite{Lan:Wu:2014},
\cite{FYWang:XZhang:2015},
\cite{ZWang:XZhang:2020}, and
\cite{Ren:Zhang:2024}.
Moreover, results about the dependence on the initial value of the SDEs are contained in
\cite{Ren:XZhang:2003},
\cite{XZhang:2005}, and
\cite{Ren:XZhang:2006}.
\smallskip

The aim of \cref{sec:Hoelder} is to investigate whether a diffusion with a H\"older continuous  diffusion coefficient
belongs to the Besov space $\B_p^{\Phi_\alpha}$ introduced in \cref{sec:Bpalpha},
see \cref{statement:holder:main} which follows from the more general \cref{statement:holder:main_general}.
This space plays a key-role
in \cref{sec:application:BSDE} when we discuss the applications to BSDEs. Our results will depend jointly
on the integrability $p$ and the H\"older continuity $\theta\in (1/2,1)$ of the diffusion coefficient. Here we do not assume
that the diffusion coefficient is non-degenerated and will exploit 
results from  Gy\"ongy and R\'asonyi  \cite{Gyongy:Rasonyi:11} about the Euler approximation.
\bigskip 

In \underline{\cref{sec:Zvonkin}} we explain that Zvonkin type transforms and the Lamperti transform go very well
together with the coupling method investigated in this article.
\bigskip

In \underline{\cref{sec:application:BSDE}} we apply the results to backward stochastic
differential equations.  Recent results about Malliavin differentiability of BSDEs -
without using the coupling method - under  classical structural assumptions and with a Markovian forward process can be found in \cite{Imkeller_etal:2024}.
Under minimal structural assumptions and for fully path-dependent forward diffusions 
our approach allows not only to treat the classical  Malliavin differentiability, but also the full scale
of Malliavin Besov spaces  $\B_{p,q}^\theta$.


\section{General setting and preliminaries}
\label{sec:preliminaries}


\subsection{General notation}\ \medskip

Given $A,B\ge 0$ and $c\ge 1$, the notation $A \sim_c B$ stands for
$\frac{1}{c} A \le B \le c A$.
For $x\in \R$ we use $\sign_0(x):= -\1_{(-\infty,0)}(x)+ \1_{(0,\infty)}(x)$. We denote by $|x|$ the euclidean norm for $x\in \R^d$,
and for a matrix $A=(a_{i,k})_{i=1,k=1}^{d,N} \in \R^{d\times N}$ the Hilbert-Schmidt norm
by $|A|_{\hs}:= \sqrt{\sum_{i=1,k=1}^{d,N} |a_{i,k}|^2}$.
Given a function $f:\R^d \to \R$ and $\theta\in (0,1]$ we use the H\"older semi-norms
\[ |f|_\theta:= \sup_{x\not = y} \frac{|f(x)-f(y)|}{|x-y|^\theta}.\]
The Lebesgue measure on $[0,T]$ will be denoted by $\lambda$.

\subsection{Stochastic basis}\ \medskip

In the sequel we assume a stochastic basis satisfying the usual conditions as follows:
For a fixed time horizon $T>0$ we let
\[     W = (W_s)_{s \in [0, T]}
     = \left ( (W_{s,1},\ldots,W_{s,N})^\top \right )_{s \in [0, T]} \]
be an $N$-dimensional standard Brownian motion defined
on a complete probability space $(\Omega,\cF,\P)$, where $W_0\equiv 0$ and all paths are continuous.
We denote by  $(\cF_s)_{s\in [0,T]}$ the augmented natural filtration of $W$ and assume
$\cF=\cF_T$.
\smallskip

\subsection{The Malliavin Sobolev space  $\D_{1,2}$}\ \medskip

Regarding the Malliavin Sobolev spaces $\D_{1,2}\subseteq L_2(\Omega,\cF,\P)$ the reader is referred
to \cite[cf. also Example 1.1.2]{Nualart:06}. For $\xi\in \D_{1,2}$ the Malliavin derivative is a map
\[ D: \D_{1,2} \to L_2(\Omega \times [0,T];\R^N) \]
and $\D_{1,2}$ is a Hilbert space under
$\| \xi\|_{\D_{1,2}}^2:=  \| \xi\|_{L_2}^2 + \| D\xi \|_{L_2(\Omega \times [0,T];\R^N)}^2$.
We use that, for $\xi\in \D_{1,2}$ and $s\in [0,T]$,
\begin{equation}\label{eqn:expected_value_D12}
   \|D (\E[\xi|\cF_s])\|_{L_2(\Omega \times [0,T];\R^N)} \le  \|  D\xi\|_{L_2(\Omega \times [0,T];\R^N)}
\end{equation}
and that for an $\cF_s$-measurable $\xi\in \D_{1,2}$ one has
\begin{equation}\label{eqn:D12_Fs}
\| D\xi\|_{L_2(\Omega\times [0,T];\R^N)} = \| D\xi\|_{L_2(\Omega\times [0,s];\R^N)}.
\end{equation}
For $p\in [2,\infty)$ the subspaces $\D_{1,p} \subseteq \D_{1,2}$ are defined by
\[ \| \xi\|_{\D_{1,p}} := \left ( \| \xi\|_{L_p}^p + \| D\xi \|_{L_p(\Omega \times [0,T];\R^N)}^p \right )^\frac{1}{p}.\]
For example, for $N=1$, if $h:\R^K\to \R$ is a smooth function with compact support, and if $0=t_0<t_1<\cdots < t_K=T$,
then a representative of $D \xi$ with $\xi=h(W_{t_1}-W_{t_0},\ldots, W_{t_K}-W_{t_{K-1}})$ is
\[ \sum_{k=1}^K \frac{\partial h}{\partial x_k}(W_{t_1}-W_{t_0},\ldots, W_{t_K}-W_{t_{K-1}}) \chf_{(t_{k-1},t_k]}
    \in  \cL_2(\Omega \times [0,T]).\]

\subsection{The $\sigma$-algebra of predictable events}\ \medskip

In the article we shall fix $t\in [0,T)$.
We let $\cP_{t,T}$ be the predictable $\sigma$-algebra generated by the continuous $(\cF_s)_{s\in[t,T]}$-adapted processes,
see \cite[page 47]{Revuz:Yor:1999}, where for $t=0$ we simply write  $\cP_T$. We shall use the following fact that can be proved
along the equivalences for predictable processes formulated on \cite[page 47]{Revuz:Yor:1999}:
\bigskip

\begin{lemm}
\label{statement:restriction_operator_predictable_processes}
Let $t\in [0,T)$, $L:[t,T]\times \Omega\to \R$ with $L_t\equiv 0$, and
$\widetilde{L}: [0,T]\times \Omega\to \R$ with $\widetilde{L}_s:= L_{s\vee t}$.
Then the following assertions are equivalent:
\begin{enumerate}[{\rm (1)}]
\item $L$ is $\cP_{t,T}$-measurable.
\item $\widetilde{L}$ is $\cP_T$-measurable.
\end{enumerate}
\end{lemm}
\bigskip

We also need \cref{statement:composition_proedictable_new} below. For its proof
we start with the following construction: Assume $t\in [0,T)$ and 
\begin{equation}\label{eqn:h}
h:[t,T]\times \Omega \times C([t,T];\R^L) \times \R^M \to \R,
\end{equation}
$L,M\in \bN$, such that
\begin{enumerate}[(h1)]
\item \label{h1} $h(t,\cdot) \equiv 0$,
\item \label{h2} $h(s,\omega,x,z)=h(s,\omega,x(\cdot \wedge s),z)$ on $[t,T]\times \Omega \times C([t,T];\R^L)\times \R^M$,
\item \label{h3} $(s,\omega)\mapsto h(s,\omega,x,z)$ is $\cP_{t,T}$-measurable for all $(x,z)\in C([t,T];\R^L)\times \R^M$,
\item \label{h4} $(x,z)\mapsto h(s,\omega,x,z)$ is continuous for all $(s,\omega)\in [t,T]\times \Omega$.
\end{enumerate}
\medskip

For $K\in \bN$ we let $t_k^K:= t + (T-t)\frac{k}{2^K}$, $k=0,\ldots,2^K$,
\[ P^K:C([t,T];\R^L)\to \left ( \ell_\infty^{2^K}\right )^L  \sptext{1}{with}{1}
   P^K x := (x(t_0^K),\ldots,x(t_{2^K-1}^K)), \]
where each entry itself is a vector of dimension $L$, and
\[ S^K:\left (\ell_\infty^{2^K}\right )^L \to C([t,T];\R^L) \]
such that $(a_k)_{k=1}^{2^K}$ is mapped to a vector of $L$ linear splines over $[t,T]$ based on
\[ (t_0^K,a_1),(t_1^K,a_1),\ldots,  (t_{2^K}^K,a_{2^K}). \]
Finally, we define
\[ h^K:[t,T]\times \Omega \times \left (\ell_\infty^{2^K}\right )^L \times \R^M \to \R
   \sptext{1}{by}{1}
   h^K(s,\omega,a,z) := h(s,\omega,S^K a,z).\]
The construction has the following properties which are standard to check:

\pagebreak

\begin{lemm}\label{statement:properties:DNSN}\
\begin{enumerate}[{\rm (1)}]
\item $P^K:C([t,T];\R^L)\to \left (\ell_\infty^{2^K}\right )^L$ and $S^K:\left (\ell_\infty^{2^K}\right )^L\to C([t,T];\R^L)$ are linear operators with
      norm one under the identification $\R^L=\ell_\infty^L$ and 
      $\left (\ell_\infty^{2^K}\right )^L=\ell_\infty^{L 2^K}$.
\item $h^K(t,\cdot) \equiv 0$.
\item $h^K(s,\omega,P^K x,z)= h^K(s,\omega,P^K (x(\cdot \wedge s))  ,z)$
      on $[t,T]\times \Omega\times C([t,T];\R^L)\times \R^M$.
\item $(s,\omega)\mapsto h^K(s,\omega,a,z)$ is $\cP_{t,T}$-measurable
      for all $(x,z)\in C([t,T];\R^L)\times \R^M$.
\item $(a,z)\mapsto h^K(s,\omega,a,z)$ is continuous
      for all $(s,\omega)\in [t,T]\times \Omega$.
\item $\lim_{K\to \infty} h^K(s,\omega,P^K x,z)= h(s,\omega,x,z)$
      on $[t,T]\times \Omega\times C([t,T];\R^L)\times \R^M$.
\end{enumerate}
\end{lemm}

For a continuous function $x:[t,T]\to \R^L$ we obtain a vector
$(x^{K,0},\ldots,x^{K,2^K-1})$
of continuous functions by
\[ x^{K,k}:[t,T] \to \R^L \sptext{1}{by}{1}
   x^{K,k}_s:= x_{s\wedge t_k^K}, \]
for which \cref{statement:properties:DNSN} implies
\[ \lim_{K\to \infty} h^K(s,\omega,(x^{K,k}_s)_{k=0}^{2^K-1},z)
   = h(s,\omega,x,z),\]
where $(x^{K,k}_s)_{k=0}^{2^K-1}$ is interpreted as an element of
$(\ell_\infty^{2^K})^L$.
\smallskip

\begin{prop}
\label{statement:composition_proedictable_new}
Let $h$
satisfy \eqref{eqn:h}, {\rm (h\ref{h1}),\ldots,(h\ref{h4})},  
$A:[t,T]\times \Omega \to \R^L$ be continuous and $(\cF_s)_{s\in [t,T]}$-adapted,
and $B:[t,T]\times \Omega \to \R^M$ be $\cP_{t,T}$-measurable. Then
\[ [t,T]\times \Omega\ni (s,\omega) \mapsto h(s,\omega,A(\omega),B_s(\omega))
   \sptext{1}{is $\cP_{t,T}$-measurable}{0}. \]
\end{prop}

\begin{proof}
We know that 
\begin{multline*}
    \lim_{K\to \infty} h^K(s,\omega,(A_{s\wedge t_k^K}(\omega))_{k=0}^{2^K-1},B_s(\omega))
   = \lim_{K\to \infty} h^K(s,\omega,P^K(A(\cdot\wedge s,\omega)),B_s(\omega)) \\
   =  h(s,\omega,A(\omega),B_s(\omega)) 
\end{multline*}
and that 
$[t,T]\times \Omega\ni (s,\omega) \mapsto  h^K(s,\omega,P^K(A(\cdot\wedge s,\omega)),B_s(\omega))$ is $\cP_{t,T}$-measurable.
For the latter we redefine $A_{t\wedge t_k^K}$ to be zero without changing 
$h^K(s,\omega,P^K(A(\cdot\wedge s,\omega)),B_s(\omega))$ on $[t,T]$ and use that 
$(A_{s\wedge t_k^K} \1_{(t,T]}(s))_{s\in [0,T]}$ is $\cP_T$-measurable.
\end{proof}

\subsection{Fefferman's inequality}\ \medskip

To relax standard Lipschitz assumptions on the drift term of the SDEs we consider we use a BMO-condition 
instead and exploit Fefferman's inequality (\cref{statement:Fefferman_inequality}) to handle
this setting:
\medskip

\begin{defi}[{\cite[Definition 5.7]{Geiss:Ylinen:21}}]
\label{definition:BMO(S_2)}
For an $(\cF_s)_{s\in [t,T]}$-progressively measurable $\R$-valued process $C=(C_s)_{s\in [t,T]}$ with
$\int_t^T \E |C_s|^2 \od s < \infty$ we let $C\in  \BMO(S_2)$ provided that
\[ \|C\|_{ \BMO(S_2)} := \sup_{s\in [t,T]} \left \| \E \left ( \int_s^T |C_u|^2 \od u\Big |\cF_s \right ) \right \|_{L_\infty}^\frac{1}{2}
   <\infty.\]
\end{defi}
\medskip

Here, the $S_2$ in the notation stands for {\it square function.}

\begin{lemm}[Fefferman's inequality, {\cite[Corollary 5.19]{Geiss:Ylinen:21}}]
\label{statement:Fefferman_inequality}
Let $A=(A_s)_{s\in [t,T]}$ and $C=(C_s)_{s\in [t,T]}$ be progressively
measurable $\R$-valued processes with
$\E \int_t^T |C_s|^2 \od s < \infty$ and let $p\in [1,\infty)$. Then one has that
\[ \left \| \int_t^T |A_s C_s| \od s \right \|_{L_p}
   \le \sqrt{2p} \left \| \sqrt{\int_t^T |A_s|^2 \od s} \right \|_{L_p} \|C\|_{ \BMO(S_2)}.\]
\end{lemm}

\subsection{Other inequalities}\ \medskip

We will exploit the Burkholder--Davis--Gundy inequality:
Given $\cP_{t,T}$-measurable processes $L^{(i)}=(L_s^{(i)})_{s\in [t,T]}$, $L_s^{(i)}:\Omega\to \R^N$, $i=1,\ldots,d$, with
$\P(\int_{[t,T]} |L_s^{(i)}|^2 \od s <\infty)=1$, we have
\begin{equation}\label{eqn:BDG}
   \left \| \sup_{s\in [t,T]} \sqrt{\sum_{i=1}^d \left | \int_t^T \langle L_s^{(i)}, \od W_s \rangle \right |^2} \right \|_{L_p}
   \sim_{\beta_p}
       \left \| \sqrt{ \int_t^T \sum_{i=1}^d |L_s^{(i)}|^2  \od s} \right \|_{L_p}
\end{equation}
for $p\in (0,\infty)$ and $\beta_p \ge 1$ not depending on $(d,N)$,
see \cite[Theorem 1.1]{Marinelli:Roeckner:2016}.
We also use Lenglart's inequality
(\cite[Corollaire II]{Lenglart:1977},
 \cite[Exercise IV.4.30, Question IV.1]{Revuz:Yor:1999},
 \cite[Theorem 2.1]{SaGeiss:Scheutzow:2021}):
\bigskip

\begin{lemm}[Lenglart's inequality]
\label{statement:Lenglart}
For a fixed $t\in [0,T)$
let $A=(A_s)_{s\in [t,T]}$ and $C=(C_s)_{s\in [t,T]}$ be non-negative adapted right-continuous
processes, where $C$ is non-decreasing and $\cP_{t,T}$-measurable. Assume
that
\[ \E[A_\tau| \cF_t] \le
   \E[C_\tau| \cF_t] \mbox{ a.s.} \]
for any stopping time $\tau:\Omega\to [t,T|$. Then for all $q\in (0,1)$ and any stopping time
$\sigma:\Omega\to [t,T]$
one has
\begin{equation}\label{eqn:statement:Lenglart}
    E\left [\sup_{s\in [t,\sigma]}|A_s|^q | \cF_t\right ]
   \le \frac{q^{-q}}{1-q} E[|C_\sigma|^q | \cF_t] \mbox{ a.s.}
\end{equation}
\end{lemm}
In the literature \eqref{eqn:statement:Lenglart} is usually formulated for $\sigma=T$. Using the assumption
for $A=(A_{s\wedge \sigma})_{s\in [t,T]}$ and $C=(C_{s\wedge \sigma})_{s\in [t,T]}$
instead for $A=(A_s)_{s\in [t,T]}$ and $C=(C_s)_{s\in [t,T]}$ yields to the above version.
\bigskip

Finally, the following lemma is standard in the literature
(cf. \cite[Theorem IV.2.4]{Ikeda:Watanabe:2nd}):

\begin{lemm}\label{statement:Xt_vs_starting_value_path}
Let $p\in (0,\infty)$ and
$b:[t,T]\times \Omega\to \R^d$ and  $\sigma:[t,T]\times \Omega\to \R^{d\times N}$
with $b(t,\cdot)\equiv 0$ and $\sigma(t,\cdot)\equiv 0$ be $\cP_{t,T}$-measurable such that
\[ \P\left (\int_t^T |b_u| \od u + \int_t^T |\sigma_u|^2 \od u < \infty \right ) =1.\]
For $\xi\in L_p(\Omega,\cF_t,\P;\R^d)$ assume
\[ X_s = \xi + \int_t^s b_u \od u + \int_t^s \sigma_u dW_u
   \sptext{1}{for}{1} s \in [t,T] \mbox{ a.s.}\]
with
\[ |b_s| \le K_b \left [1+\sup_{u\in [t,s]} |X_u| \right ]
   \sptext{1}{and}{1}
  |\sigma_s|_\hs \le K_\sigma \left [1+\sup_{u\in [t,s]} |X_u| \right ]
  \quad \lambda\otimes\P \mbox{ a.e.} \]
Then there is a constant $c_{\eqref{eqn:statement:Xt_vs_starting_value_path},p}=c_{\eqref{eqn:statement:Xt_vs_starting_value_path},p}(T,K_b,K_\sigma,p)>0$ such that
\begin{equation}\label{eqn:statement:Xt_vs_starting_value_path}
  \E \left [ [ 1+ \sup_{u\in [t,T]} |X_u| ]^p | \cF_t \right ]
  \le c^p_{\eqref{eqn:statement:Xt_vs_starting_value_path},p} \E [ [1+|\xi|]^p|\cF_t ].
\end{equation}
\end{lemm}

\begin{proof}
First we assume $p\in [2,\infty)$.
Let $B\in \cF_t$ be of positive measure and choose $n\in \bN$ such that
$\P(B\cap \{ |\xi| \le n\})>0$. Set $A_n:= B\cap \{ |\xi| \le n\}$ and let $\P_{A_n}$ be the normalized restriction
of $\P$ to $A_n$. Define the stopping time (with respect to the stochastic basis restricted to $A_n$)  $\tau_k:A_n\to [0,T]$, $k\ge n$, by
\[ \tau_k(\omega) := \inf \{ s\in [t,T] : |X_s(\omega)| \ge k \} \wedge T.\]
By the Burkholder-Davis-Gundy inequalities we get for $s\in [t,T]$
that
\begin{align*}
&    \left \| \sup_{u\in [t,s\wedge \tau_k]} |X_u| \right \|_{L_p(A_n,\P_{A_n})}\\
&\le \|\xi\|_{L_p(A_n,\P_{A_n})} + K_b \int_t^s \|1_{\{u\le \tau_k\}} [1+|X_u]|\|_{L_p(A_n,\P_{A_n})} \od u \\
& \hspace{10em}               + K_\sigma \beta_p \left \| \left ( \int_t^s 1_{\{u\le \tau_k\}} [1+|X_u|]^2 \od u \right )^\frac{1}{2} \right \|_{L_p(A_n,\P_{A_n})} \\
&\le \|\xi\|_{L_p(A_n,\P_{A_n})} + K_b \int_t^s \|1_{\{u\le \tau_k\}} [1+|X_u|]\|_{L_p(A_n,\P_{A_n})} \od u \\
&   \hspace{10em}                 + K_\sigma \beta_p \left ( \int_t^s \|1_{\{u\le \tau_k\}} [ 1+|X_u|]\|_{L_p(A_n,\P_{A_n})}^2 \od u \right )^\frac{1}{2} \\
&\le \|\xi\|_{L_p(A_n,\P_{A_n})} + [K_b \sqrt{T} + K_\sigma \beta_p] \left ( \int_t^s \| 1_{\{u\le \tau_k\}} [1+|X_u|]\|_{L_p(A_n,\P_{A_n})}^2 \od u \right )^\frac{1}{2}.
\end{align*}
This implies
\begin{align*}
&    \hspace{-2em}\left \| \sup_{u\in [t,s]} [1+|X_{u\wedge \tau_k}|] \right \|_{L_p(A_n,\P_{A_n})} \\
&\le 2 \|[1+|\xi|]\|_{L_p(A_n,\P_{A_n})} \\ 
&    \hspace{1em} + [K_b \sqrt{T} +  K_\sigma \beta_p]
     \left ( \int_t^s \left \| \sup_{u\in [t,v] [1+|X_{u\wedge \tau_k}|]}\right \|_{L_p(A_n,\P_{A_n})}^2 \od v
     \right )^\frac{1}{2}.
\end{align*}
By assumption we have $|\xi| \le n$ on $A_n$, and by the definition of $\tau_k$ we have
$|X_{u\wedge \tau_k}|\le k$. So squaring both sides and
using $|x+y|^2 \le 2 [x^2+y^2]$ we may use Gronwall's inequality to derive
\[  \left \| \sup_{u\in [t,T]} [1+|X_{u\wedge\tau_k}|] \right \|_{L_p(A_n,\P_{A_n})}
    \le c \|1+|\xi|\|_{L_p(A_n,\P_{A_n})} \]
for
$c:= 2 \sqrt{2} e^{[K_b \sqrt{T} +  K_\sigma \beta_p]^2 T}$.
By $k\to \infty$ this gives
\[  \left \| \sup_{u\in [t,T]} [1+|X_u|] \right \|_{L_p(A_n,\P_{A_n})}
     \le c \|1+|\xi|\|_{L_p(A_n,\P_{A_n})}.\]
And, finally by $n\to \infty$ we conclude
\begin{equation}\label{eqn:proof:statement:Xt_vs_starting_value}
\left \| \sup_{u\in [t,T]}[1+|X_u|] \right \|_{L_p(B,\P_B)}
     \le c \|1+|\xi|\|_{L_p(B,\P_B)}.
\end{equation}
Finally, let
$A_s:=  [1+|X_s|]^p$ and $C_s:= [1+|\xi|]^p$.
Now \cref{statement:Lenglart} implies \eqref{eqn:proof:statement:Xt_vs_starting_value} for
$r\in (0,p)$ as well with an additional multiplicative constant arising from  \cref{statement:Lenglart}.
\end{proof}


\section{Transference of the Wiener space}
\label{sec:transference}

In this section we explain how we can transfer random variables and stochastic processes
from one Wiener space to another one. This is the basis of the coupling procedure recalled in \cref{sec:coupling}.
We follow \cite[Chapter 3]{Geiss:Ylinen:21}
and assume two stochastic bases $(\Omega^i,\cF^i,\P^i,(W_s^i)_{s\in [0,T]})$, $i=0,1$,
where $(\Omega^i,\cF^i,\P^i)$ is complete, $W^i=(W_s^i)_{s\in [0,T]}$  are $N$-dimensional Brownian motions with
$W_0^i\equiv 0$,
\[ W_s^i=(W_{s,1}^i,\ldots,W_{s,N}^i)^T,\]
and where all trajectories are assumed to be continuous.
We let $\cF_s^i := \sigma(W_u^i: u\in [0,s]) \vee \cN^i$ with
$\cN^i$ being the null-sets in $\cF^i$, and assume that
$\cF^i=\cF_T^i$. Furthermore, assume that $(R,\cR,\rho)$ is another probability space.
According to \cite[Proposition 2.5(2)]{Geiss:Ylinen:21}
one obtains a linear and isometric bijection
\[ \cC:L_0(R\times\Omega^0,\cR\otimes\cF^0,\rho\otimes\P^0 )\to L_0(R\times\Omega^1,\cR\otimes\cF^1,\rho\otimes\P^1), \]
where we take $d(f,g):= \E \frac{|f-g|}{1+|f-g|}$ as metric and proceed as follows:
\bigskip
\begin{enumerate}[{--}]
\item We choose a representative $f\in [f]\in L_0(R\times\Omega^0,\cR\otimes\cF^0,\rho\otimes\P^0)$ that is 
      $\cR\otimes \sigma(W_s^0:s\in [0,T])$-measurable,
      which exists according to \cite[Lemma 2.1]{Geiss:Ylinen:21}).

\item We consider the maps $J^i:R\times \Omega^i \to R \times \R^\bN$ with
      \[ J^i(r,\omega^i) := (r,(\xi_k^i(\omega^i))_{k\in \bN}) \]
      and $(\xi_k^i)_{k\in \bN}$ being an enumeration of $(\int_0^T h_\ell(s) \od W_{s,j}^i)_{\ell=0,j=1}^{\infty,N}$, where
      $(h_\ell)_{\ell=0}^\infty$ are the Haar functions forming an orthonormal basis in $L_2([0,T])$
      with $h_0\equiv T^{-1/2}$, $h_1=  T^{-1/2}(\1_{(0,\frac{T}{2}]} - \1_{(\frac{T}{2},T]})$, etc.

\item The map $\widehat{f}:R\times \R^\bN\to \R$ is obtained from the factorization
      $f=\widehat{f}\circ J^0$ according to \cite[Lemma 2.2]{Geiss:Ylinen:21}.
\item We define $\cC([f])=\cC(f) := [\widehat{f}\circ J^1] \in L_0(R\times\Omega^1,\cR\otimes\cF^1,\rho\otimes\P^1)$.
\end{enumerate}
\bigskip
By choosing $(R,\cR,\rho)$ to be $(\{0\},2^{\{0\}},\delta_0)$ and
$([0,T],\cB([0,T]),\lambda/T)$, respectively, we use the construction 
to define the two transference operators
\begin{align*}
\cC_0 & :L_0(\Omega^0,\cF^0,\P^0 )\to L_0(\Omega^1,\cF^1,\P^1),\\
\cC_T & :L_0\left ([0,T]\times\Omega^0,\cB([0,T])\otimes\cF^0,\frac{\lambda}{T}\otimes\P^0 \right ) \\
      & \hspace*{13em} \to
         L_0\left ([0,T]\times\Omega^1,\cB([0,T])\otimes\cF^1,\frac{\lambda}{T}\otimes\P^1 \right ).
\end{align*}
To shorten the notation we shall write
\[ \Omega_T := [0,T]\times \Omega
   \sptext{1}{and}{1}
   \Omega_T^i := [0,T]\times \Omega^i
   \sptext{1}{for}{1} i=0,1. \]

To formulate properties of $\cC_0$ and $\cC_T$ we let
$A\in \cL_0(\Omega_T,\cP_T;C(\R^M))$ if $A(\cdot,x):\Omega_T \to \R$ is predictable
for all $x\in \R^M$ and $\R^M\ni x\mapsto A(s,\omega,x)\in \R$ is continuous for all $(s,\omega)\in \Omega_T$.
We use the following properties of the transference operators:
\bigskip

\begin{enumerate}[{\rm (P1)}]
\item \label{condition:P1}
      If $A^0:\Omega_T^0\to \R$ is continuous and adapted, then there is a continuous and adapted $A^1:\Omega^1_T\to \R$ with
      $A^1_s\in \cC_0(A_s^0)$ for all $s\in [0,T]$, see \cite[Lemma 3.1, Proposition 2.12(1)]{Geiss:Ylinen:21}.
      We use the notation $A^1\in C_T^{\rm c,adap}(A^0)$.
      \bigskip

\item \label{condition:P2}
      If $A^0\in \cL_0(\Omega^0_T,\cP_T^0)$, then there is an $A^1\in \cL_0(\Omega^1_T,\cP_T^1)$ such that
      $A^1\in \cC_T(A^0)$, see \cite[Lemma 3.1, Proposition 2.12(2)]{Geiss:Ylinen:21}.
      We  use the notation $A^1\in C_T^{\cP}(A^0)$.
     \smallskip

\item \label{condition:P3}
      If $A^0\in \cL_0(\Omega^0_T,\cP_T^0;C(\R^M))$, then there is an $A^1\in \cL_0(\Omega^1_T,\cP_T^1;C(\R^M))$ such that
      $A^1(\cdot,x)\in \cC_T(A^0(\cdot,x))$ for all $x\in \R^M$, see \cite[Lemma 3.1, Proposition 2.12(2), Proposition 2.13(2)]{Geiss:Ylinen:21}.
      We also use the notation $A^1\in C_T^{\cP}(A^0)$.

\end{enumerate}
\bigskip

In the notation $A^1\in C_T^{\cP}(A^0)$ the symbol $\cP$ means that we start with a predictable process $A^0$ and can arrange 
$A^1$ to be predictable as well. Next we want to restrict the above operators to the interval
$[t,T]$ without repeating the construction from \cite{Geiss:Ylinen:21}.
We start with $C_T^{\rm c,adap}$ by the following observation:

\begin{lemm}
\label{statement:restriction_operator_continuous_processes}
For a continuous and adapted process $(X_s^0)_{s\in [t,T]}$ we have:
\begin{enumerate}[{\rm (1)}]
\item \label{item:1:statement:restriction_operator_continuous_processes}
      The process $(Y_t^0)_{t\in [0,T]}:=(X_{s\vee t}^0-X_t^0)_{s\in [0,T]})$ is continuous and adapted with
      $Y_s\equiv 0$ for $s\in [0,t]$.
\item For $Y^1\in C_{[0,T]}^{\rm c,adap}(Y^0)$ with $Y^0$ taken from \eqref{item:1:statement:restriction_operator_continuous_processes}
      and $A\in C_0(X_t^0)$ define
      \[ X_s^1 := A + Y_s^1 \sptext{1}{for}{1} s\in [t,T].\]
\end{enumerate}
Then $(X_s^1)_{s\in [t,T]}$ is continuous and adapted such that
$X_s^1\in C_0(X_s^0)$ for all $s\in [t,T]$.
\end{lemm}
\begin{proof}
The process $(X_s^1)_{s\in [t,T]}$ is continuous and adapted by construction.
Furthermore, for $s\in [t,T]$  we have that
\[ C_0(X_s^0) = C_0(X_t^0) + C_0(X_s^0-X_t^0)
              = C_0(X_t^0) + C_0(Y_s^0)
              = [A] + [Y_s^1]
              = [X_s^1]. \]
so that we have  $X_s^1 = A + Y_s^1 \in  C_0(X_s^0)$ for the representatives.
\end{proof}

The above lemma leads to the following definition:

\begin{defi}
\label{definition:restriction_operator_continuous_processes}
Given a path-wise continuous and adapted process $(X_s^0)_{s\in [t,T]}$ we define
$(X_s^1)_{s\in [t,T]} \in C_{[t,T]}^\cadap((X_s^0)_{s\in [t,T]})$
provided that
\begin{enumerate}[{\rm (1)}]
\item $(X_s^1)_{s\in [t,T]}$ is path-wise continuous and adapted,
\item $X_s^1\in C_0(X_s^0)$ for all $s\in [t,T]$.
\end{enumerate}
\end{defi}
\medskip

By construction we also have the following properties:

\begin{lemm}\label{statement:properties_CCtTadapted}
Assume $X^0=(X_s^0)_{s\in [t,T]}$ and $X^1=(X_s^1)_{s\in [t,T]}$ from \cref{definition:restriction_operator_continuous_processes}.
\begin{enumerate}[{\rm (1)}]
\item \label{item:1:statement:properties_CCtTadapted}
      Up to indistinguishability the process $X^1$ is unique.
\item \label{item:2:statement:properties_CCtTadapted}
      The processes  $X^0$ and $X^1$ have the same finite-dimensional distributions.
\end{enumerate}
\end{lemm}

\begin{proof}
\eqref{item:1:statement:properties_CCtTadapted} follows from the continuity of the processes and the
fact that $X_s^1$ is almost surely unique.
\eqref{item:2:statement:properties_CCtTadapted} is a consequence of \cite[Proposition 2.5(3)]{Geiss:Ylinen:21}.
\end{proof}
\bigskip

\cref{statement:restriction_operator_predictable_processes} leads to the following counterpart to \cref{definition:restriction_operator_continuous_processes}:
\medskip

\begin{defi}
\label{definition:restriction_operator_predictable_processes}
For $t\in [0,T)$, $A^0 \in \cL_0([t,T]\times \Omega^0,\cP_{t,T}^0)$ or
$A^0 \in \cL_0([t,T]\times \Omega^0,\cP_{t,T}^0;C(\R^M))$ such that in both cases $A^0_t\equiv 0$, we let
$A^1\in \cC_{[t,T]}^\cP(A^0)$ if
\[ A^1: = \widetilde{A}^1|_{[t,T]}
   \sptext{1}{for some}{1}
   \widetilde{A}^1 \in \cC_T^\cP \left ( (A_{s\vee t}^0)_{s\in [0,T]}\right)
   \sptext{.5}{with}{.5} \widetilde{A}^1_s\equiv 0 \sptext{.5}{for}{.5} s\in [0,t]. \]
\end{defi}
\medskip

\begin{rema}
\label{rema:definition:restriction_operator_predictable_processes}
To justify \cref{definition:restriction_operator_predictable_processes} we let
\[ A^0 \in \cL_0([t,T]\times \Omega^0,\cP_{t,T}^0;C(\R^M))
   \sptext{1}{and}{1}
   B^1 \in \cC_T^\cP \left ( (A_{s\vee t}^0)_{s\in [0,T]}\right ) \]
with $A^0_t\equiv 0$.
By definition,
\[ B^1(\cdot,x) \in \cC_T \left ( (A_{s\vee t}^0(\cdot,x))_{s\in [0,T]}\right )
   \sptext{1}{for all}{1} x\in \R^M.\]
As $(1_{(t,T]})_{s\in [0,T]}\in  \cC_T (1_{(t,T]})$, which is expected and follows, for example, from \cite[Proposition 2.5(7)]{Geiss:Ylinen:21}, we get
by \cite[Proposition 2.5(4)]{Geiss:Ylinen:21} that
\[ B^1(\cdot,x) 1_{(t,T]} \in \cC_T \left ( (1_{(t,T]}(s) A_{s\vee t}^0(\cdot,x))_{s\in [0,T]} \right )
   =  \cC_T \left ( ( A_{s\vee t}^0)(\cdot,x))_{s\in [0,T]}\right ). \]
Because $B^1(\cdot,x) 1_{(t,T]}$ is predictable as process on $[0,T]$ we may choose
$ \widetilde{A}^1:= B^1 1_{(t,T]}$.
The case $A^0 \in \cL_0([t,T]\times \Omega^0,\cP_{t,T}^0)$ can be treated in the same way.
\end{rema}

Moreover, we shall use the following lemmas:

\begin{lemm}
For $(X_s^0)_{s\in [t,T]}$ and $(X_s^1)_{s\in [t,T]}$ from \cref{definition:restriction_operator_continuous_processes}
one has
\[ \left (X_s^1 1_{(t,T]}(s) \right )_{s\in [t,T]} \in \cC_{[t,T]}^\cP \left (\left (X_s^0 1_{(t,T]}(s)\right )_{s\in [t,T]}\right ).\]

\end{lemm}

\begin{proof}
Using \cite[Proposition 2.5(7)]{Geiss:Ylinen:21} we get that
\[ \left (X_s^1 1_{(t,T]}(s) \right )_{s\in [0,T]} \in \cC_T \left (\left (X_s^0 1_{(t,T]}(s)\right )_{s\in [0,T]}\right ). \]
As both processes are predictable as processes on $[0,T]$ and vanish on $[0,t]$, the assertion follows.
\end{proof}
\medskip

\begin{lemm}
\label{statement:CPtT_is_isometry}
For $A^0\in \cL_0([t,T]\times \Omega^0,\cP_{t,T}^0)$ with $A_t^0 \equiv 0$ and
$A^1 \in C_{[t,T]}^\cP(A^0)$ we have that the distributions of $A^i:[t,T]\times\Omega^i\to \R$, $i=0,1$, coincide.
\end{lemm}

\begin{proof}
This property is satisfied for $C_T^\cP$ as $C_T$ is an isometry. By \cref{definition:restriction_operator_predictable_processes}
this follows for the restriction $C_{[t,T]}^\cP$ of $C_T^\cP$.
\end{proof}
\medskip

The following Propositions \ref{statement:transference_composition_new} and \ref{statement:transference_sde} are
the basis for the transference of SDEs and are an adaptation of \cite[Theorem 3.3]{Geiss:Ylinen:21}.
They cover the case of random coefficients when we insert another predictable process in non path-dependent form.
Here and later in our framework, we may always assume w.l.o.g. realizations of stochastic integrals where all trajectories are continuous.
Moreover, we note that
$\int_t^s L_s  \od W_s$ means $\int_{(t,s]} L_s  \od W_s$, i.e. the random variable $L_t$ is not used for this stochastic integral.
\bigskip

\begin{prop}
\label{statement:transference_composition_new}
Assume that $h^i:[t,T]\times \Omega^i \times C([t,T];\R^L)\times \R^M\to \R$, $i=0,1$,
satisfy {\rm (h\ref{h1}),\ldots,(h\ref{h4})} with respect to
$(\Omega^i,\cF^i,\P^i,(W_s^i,\cF_s^i)_{s\in [t,T]})$ and that
\begin{enumerate}[{\rm (1)}]
\item \label{item:1:statement:transference_composition_new}
      $h^1(\cdot,\cdot,x,z)\in \cC_{[t,T]}^\cP(h^0(\cdot,\cdot,x,z))$ for all
      $(x,z)\in C([t,T];\R^L)\times \R^M$,
\item \label{item:2:statement:transference_composition_new}
      $A^1\in \cC_{[t,T]}^\cadap(A^0)$,
\item \label{item:3:statement:transference_composition_new}
      $B^1\in \cC_{[t,T]}^\cP(B^0)$.
\end{enumerate}
Then one has
$(h^1(s,\cdot,A^1,B_s^1))_{s\in [t,T]} \in \cC_{[t,T]}^\cP((h^0(s,\cdot,A^0,B_s^0))_{s\in [t,T]} )$.
\end{prop}

\begin{proof}
We have that 
\begin{align*}
(A^1_{s\wedge t_k^K} \1_{[t,T]}(s))_{s\in [0,T]} & \in \cC_T \left ( (A^0_{s\wedge t_k^K} \1_{[t,T]}(s))_{s\in [0,T]} \right )
 \sptext{1}{for}{1} k=0,\ldots,2^K-1,\\
(B^1_{s\vee t})_{s\in [0,T]} & \in \cC_T \left ( (B^0_{s\vee t})_{s\in [0,T]} \right ).
\end{align*}
In the notation of \cref{statement:properties:DNSN} and with 
\cite[Proposition 2.13(1)]{Geiss:Ylinen:21} this implies that 
\begin{multline*}
   \left ( (h^1)^K(s\vee t,\cdot,(A^1_{s\wedge t_k^K} \1_{[t,T]}(s))_{k=0}^{2^K-1},B_{s\vee t}^1)\right )_{s\in [0,T]} \\
   \in
   \cC_T \left (  \left ( (h^0)^K(s\vee t,\cdot,(A^0_{s\wedge t_k^K} \1_{[t,T]}(s))_{k=0}^{2^K-1},B_{s\vee t}^0)\right )_{s\in [0,T]} \right ) .
\end{multline*}
By $K\to \infty$ and \cite[Proposition 2.5(2)]{Geiss:Ylinen:21} we obtain
\[ \left ( h^1(s\vee t,\cdot,A^1,B_{s\vee t}^1)\right )_{s\in [0,T]} \in
   \cC_T \left (  \left ( h^0(s\vee t,\cdot,A^0,B_{s\vee t}^0)\right )_{s\in [0,T]} \right ) .\]
Finally, \cref{statement:composition_proedictable_new} implies that
$\left ( h^i(s,\cdot,A^i,B_{s\vee t}^i)\right )_{s\in [t,T]}$
is $\cP_{t,T}^i$-measurable for $i=0,1$, so that the assertion follows.
\end{proof}
\bigskip

\begin{rema}
\label{statement:conditons_transference_composition_new}
In \eqref{item:1:statement:transference_composition_new} of  \cref{statement:transference_composition_new}
we {\it assume} the existence of $h^1$, whereas in items
\eqref{item:2:statement:transference_composition_new} and \eqref{item:3:statement:transference_composition_new}
we know the existence of $A^1$ and $B^1$. We give two situations where such an $h^1$ in \eqref{item:1:statement:transference_composition_new} exists:

\begin{enumerate}
\item \label{item:1:statement:conditons_transference_composition_new}
      $h^0$ does not depend on $x$, i.e. is not path-dependent.
\item \label{item:2:statement:conditons_transference_composition_new}
      $h^0$ can be factorized as follows:
      \[ h^0(s,\omega^0,x,z) = h^0_1(s,\omega^0,h^0_2(s,x,z),z) \]
      where
      $h^0_1:[t,T]\times \Omega^1 \times \R^K \times \R^M\to \R$ and
      $h^0_2:[t,T]\times C([t,T];\R^L)\times \R^M \to \R^K$
      satisfy {\rm (h\ref{h1}),\ldots,(h\ref{h4})} with respect to
      $(\Omega^0,\cF^0,\P^0,(W_s^0,\cF_s^0)_{s\in [t,T]})$.
\end{enumerate}
\end{rema}
\medskip

\begin{proof}
\eqref{item:1:statement:conditons_transference_composition_new}
The existence follows by \cref{rema:definition:restriction_operator_predictable_processes}
which did yield to \cref{definition:restriction_operator_predictable_processes}.
\smallskip

\eqref{item:2:statement:conditons_transference_composition_new}
We use \cref{statement:transference_composition_new} for $h_2^0$ because
$h_2^0$ does not depend on $\omega^0$. Therefore we get, coordinate-wise,
\[ (h_2^1(s,A^1,B_s^1))_{s\in [t,T]}\in \cC_{[t,T]}^\cP((h^0_2(s,A^0,B_s^0))_{s\in [t,T]} ).\]
Now we use \cref{statement:conditons_transference_composition_new}\eqref{item:1:statement:conditons_transference_composition_new}.
\end{proof}

\bigskip
\begin{prop}
\label{statement:transference_sde}
Let $t\in [0,T)$, $N,d\in \bN$, and consider the $\R$-valued equation
\[ L_s^0 = \int_t^s b^0(u, K_u^0)\,\od u + \int_t^s \sigma^0(u, K_u^0)\,\od W_u^0
   \sptext{1}{for}{1} s\in [t,T] \quad \P^0\mbox{-a.s.} \]
with $K_u^0=(K_{1,u}^0,\ldots,K_{d,u}^0)$ and $\sigma^0(u,x)=(\sigma_1^0(u,x),\ldots, \sigma_N^0(u,x))$ such that
\medskip
\begin{enumerate}[{\rm (1)}]
\item $L^0:[t,T]\times \Omega^0\to \R$ is continuous and adapted with $L_t^0\equiv 0$,
\item $K^0:[t,T]\times \Omega^0\to \R^d$ is $\cP_{t,T}^0$-measurable, $b^0, \sigma^0_k \in \cL_0([t,T]\times \Omega^0,\cP^0_{t,T};C(\R^d))$,
      \[ K_t^0 \equiv 0, \sptext{1}{and}{1}
         b^0_t=\sigma_1^0(t)=\cdots=\sigma_N^0(t)\equiv 0, \]
\item $\E_{\P^0} \int_t^T [|b^0(u, K_u^0)|+ |\sigma^0(u, K_u^0)|^2] \od u  < \infty$,
\item $L^1\in C_{[t,T]}^{\rm c,adap}(L^0)$, where we may assume $L_t^1\equiv 0$,
\item $K^1_j\in C_{[t,T]}^{\cP}(K^0_j)$, $b^1\in C_{[t,T]}^{\cP}(b^0)$, and $\sigma^1_k\in C_{[t,T]}^{\cP}(\sigma^0_{k})$.
\end{enumerate}
\medskip
Then one has $\E_{\P^1}  \int_t^T [|b^1(u, K_u^1)|+ |\sigma^1(u, K_u^1)|^2] \od u  < \infty$ and
\[ L_s^1 = \int_t^s b^1(u, K_u^1)\,\od u + \int_t^s \sigma^1(u, K_u^1)\,\od W_u^1
   \sptext{1}{for}{1} s\in [t,T] \quad \P^1\mbox{-a.s.}\]
\end{prop}
\medskip

\begin{proof}
First we extend $L^0,K^0,b^0,\sigma^0$ to $[0,T]$ by choosing the value zero
and denote these extensions by  $\widetilde{L}^0,\widetilde{K}^0,\widetilde{b}^0,\widetilde{\sigma}^0$.
We obtain the extended equation
\[ \widetilde{L}_s^0 = \int_0^s \widetilde{b}^0(u, \widetilde{K}_u^0)\,\od u + \int_0^s \widetilde{\sigma}^0(u, \widetilde{K}_u^0)\,\od W_u^0
   \sptext{1}{for}{1} s\in [0,T] \mbox{ a.s.} \]
To this equation we apply  \cite[Theorem 3.3]{Geiss:Ylinen:21} and obtain
\[ \widetilde{L}_s^1 = \int_0^s \widetilde{b}^1(u, \widetilde{K}_u^1)\,\od u + \int_0^s \widetilde{\sigma}^1(u, \widetilde{K}_u^1)\,\od W_u^1
   \sptext{1}{for}{1} s\in [0,T] \mbox{ a.s.} \]
For $s\in [t,T]$ this implies a.s. for all $s\in [t,T]$ that
\begin{align*}
    \widetilde{L}_s^1 - \widetilde{L}_t^1
& = \int_t^s \widetilde{b}^1(u, \widetilde{K}_u^1)\,\od u + \int_t^s \widetilde{\sigma}^1(u, \widetilde{K}_u^1)\,\od W_u^1 \\
& = \int_t^s b^1(u, K_u^1)\,\od u + \int_t^s \sigma^1(u, K_u^1)\,\od W_u^1.
\end{align*}
As $(\widetilde{L}_s^1 - \widetilde{L}_t^1)_{s\in [t,T]} \in C_{[t,T]}^{\rm c,adap}(L^0)$ the proof is complete.
\end{proof}


\section{Couplings of the Wiener space and Besov spaces}


\subsection{Couplings of the Wiener space}\ \medskip
\label{sec:coupling}

The idea is to couple two Wiener spaces by replacing a part of the Gaussian structure by
an independent copy. More precisely, we couple the stochastic bases
$(\Omega^0,\cF^0,\P^0,(W^0_s,\cF^0_s)_{s\in [0,T]})$ and
$(\Omega^1,\cF^1,\P^1,(W^1_s,\cF^1_s)_{s\in [0,T]})$ on a {\it joint} stochastic basis,
where the Brownian motions
$(W^0_s)_{s\in [0,T]}$ and $(W^1_s)_{s\in [0,T]}$ relate to each other by a
coupling function  $\varphi$.
For a coupling described by $\varphi$ we will measure the impact of this coupling in the $L_p$-sense. Inspecting this impact for a specific class of couplings,
we define Besov spaces, also covering classical Besov spaces obtained by real
interpolation.
\medskip

Let us recall the coupling from \cite[Section 4.2]{Geiss:Ylinen:21}:
If $(\Omega',\cF',\P',(W'_s,\cF'_s)_{s\in [0,T]})$ is a copy of
$(\Omega ,\cF,\P,(W_s,\cF_s)_{s\in [0,T]})$,
where -- as before -- the filtrations are the augmentations of the natural filtrations of the corresponding Brownian motions
that start identically in zero and have only continuous paths,
$\cF=\cF_T$, and $\cF'=\cF'_T$, then we set
\[ \overline{\Omega}:= \Omega\times\Omega', \quad
   \overline{\P} := \P \otimes \P', \quad
   \overline{\cF}:= \overline{\cF \otimes \cF'}^{\overline{\P}}.\]
We extend the Brownian motions $W = (W_s)_{s\in [0, T]}$ and $W'=(W'_s)_{s\in [0,T]}$
canonically to $\Omega\times \Omega'$ and consider the $2N$-dimensional Brownian motion
$\overline{W}=(W,W')$ on $( \overline{\Omega},\overline{\P},\overline{\cF})$ and denote the augmented natural
filtration of $\overline{W}$ by $(\overline{\cF}_s)_{s\in [0,T]}$.
\medskip

Consider a measurable function $\varphi:[0,T] \to [0,1]$ and define the path-wise continuous
process
\[ W^\varphi_s := \int_0^s  \sqrt{1 - (\varphi(u))^2} \,\od W_u + \int_0^s \varphi(u)\,\od W'_u \]
with  $W_0^\varphi\equiv 0$ (in particular, we fix a version of
$W^\varphi$ for the following).
By L\'evy's characterization we obtain a Brownian motion on
$( \overline{\Omega},\overline{\cF},\overline{\P})$ and denote the $\overline{\P}$-augmented natural
filtration of $W^\varphi =(W^\varphi_s)_{s\in [0,T]}$
by $\F^\varphi=(\cF^\varphi_s)_{s\in [0,T]}$.
Coming back to \cref{sec:transference}, we specify the two stochastic bases considered there to be
\begin{equation}\label{eqn:phi-coupling}
   (\overline{\Omega},\cF^0_T,\overline{\P},(W^0_s,\cF^0_s)_{s\in [0,T]})
   \sptext{.5}{and}{.5}
   (\overline{\Omega},\cF^\varphi_T,\overline{\P},(W^\varphi_s,\cF^\varphi_s)_{s\in [0,T]}).
\end{equation}
By \cite[Lemma 3.1]{Geiss:Ylinen:21} we have the transference
\[ W_{s,j}^\varphi \in C_0(W_{s,j}^0)
   \sptext{1}{for}{1} s\in [0,T] \sptext{1}{and}{1}
   j=1,\ldots, N.\]
Here and in the following the superscript $0$ stands for the coupling function
$\varphi_0\equiv 0$, i.e. $W_s^0:=W_s^{\varphi_0}$ and $\cF_s^0:=\cF_s^{\varphi_0}$.
The predictable $\sigma$-algebras $\cP_{t,T}^0$ and
$\cP_{t,T}^\varphi$ are defined relative to the filtrations 
$(\cF_s^0)_{s\in [t,T]}$ and $(\cF_s^\varphi)_{s\in [t,T]}$.
\medskip

Given a random variable $\xi\in \cL_0(\Omega,\cF,\P)$ we will extend this random variable 
canonically to $\xi^0\in \cL_0(\overline{\Omega},\cF_T^0,\overline{\P})$ and obtain its
$\varphi$-coupling $\xi^\varphi\in \cL_0(\overline{\Omega},\cF_T^\varphi,\overline{\P})$. If there
is no risk of confusion we identify $\xi$ and $\xi^0$ and obtain the coupling transformation
\[ \xi \mapsto \xi^\varphi.\]
The role of the different Brownian motions is as follows: The independent Brownian motions
$W$ and $W'$ generate with $\overline{W}=(W,W')$ a joint stochastic basis and, at the same time,
the coupling $(W,W^\varphi)$ along the coupling function $\varphi$. Finally, the bases
generated by $\overline{W}$ enables us to use stochastic analysis to investigate the coupling.
\medskip

In particular, we use two types of coupling functions,  $\varphi_r:\equiv r$ for $r\in [0,1]$ 
for the \qq{isonormal decoupling} and
$\varphi=\1_{(a,c]}$ with $0\le a < c \le T$ for an \qq{cut-off decoupling}, and assume the following versions and notation:
\begin{align*}
W^\br_s         = W^{\varphi_r}_s & :\equiv \sqrt{1-r^2} W_s + r W_s',\\
W^{(a,c]}_s   = W^{\1_{(a,c]}}_s  & :\equiv
        \begin{cases}
                W_s & 0\leq s\leq a,\\
                W_a + W'_s - W'_a & a\leq s\leq c,\\
                W_a + (W'_c - W'_a) + (W_s - W_c) & c\leq s\leq T.
        \end{cases}
\end{align*}
To shorten the notation we use for these two couplings the notation
\begin{equation}\label{eqn:notation:coupling_rv}
       \xi^\br = \xi^{\varphi_r} \sptext{1}{and}{1}
 \xi^{(a,b]} = \xi^{\1_{(a,c]}}
 \sptext{1}{for}{1} \xi \in \cL_0(\Omega,\cF,\P).
\end{equation}


\subsection{Besov spaces $\B_p^\Phi$}\ \medskip
\label{sec:Besov}

We start with the pseudo-metric space $(\cD,\delta)$ given by
\[ \cD := \{\varphi \in \cL_2((0,T]) : 0 \le \varphi \le 1 \}
   \sptext{1}{with}{1}
   \delta(\varphi,\psi) := \| \varphi-\psi \|_{L_2((0,T])}.\]
From this we obtain the corresponding metric space $(\Delta,\delta)$ of equivalence classes.
Let $C^+(\Delta)$ be the space of non-negative continuous functions $F:\Delta \to [0,\infty)$. A functional
$\Phi:C^+(\Delta)\to [0,\infty]$ is called {\em admissible} provided that
\begin{enumerate}
\item $\Phi (F+G) \le \Phi(F) + \Phi(G)$,
\item $\Phi(\lambda F) = \lambda \Phi(F)$ for $\lambda>0$ and $\Phi(0)=0$,
\item $\Phi(F) \le \Phi(G)$ for $0\le F \le G$,
\item $\Phi(F) \le \limsup_n \Phi(F_n)$ for $\sup_{\varphi \in \Delta} |F_n(\varphi) - F(\varphi)| \to_n 0$.
\end{enumerate}
Next, for $p\in (0,\infty)$ we let
\begin{equation}\label{eqn:def:Phip}
   |\xi|_{\Phi,p}       := \Phi(\varphi \to \| \xi-\xi^\varphi\|_{L_p})
    \sptext{1}{and}{1}
    \| \xi \|_{\B_p^\Phi}  := \left [ \E|\xi|^p + |\xi|_{\Phi,p}^p \right ]^\frac{1}{p}.
\end{equation}
\subsubsection{The spaces $\B_p^{\Phi_\alpha}$}
\label{sec:Bpalpha}
For $p\in (0,\infty)$ and $\alpha \in [2,\infty)$ this leads to the particular Besov spaces
$\B_p^{\Phi_\alpha}$  where $\xi \in \B_p^{\Phi_\alpha}$ provided that
\[ \| \xi \|_{\B_p^{\Phi_\alpha}} := \left [ \E |\xi|^p + |\xi|_{\Phi_\alpha,p}^p \right ]^\frac{1}{p} < \infty
    \sptext{1}{with}{1}
    \Phi_\alpha(F) := \sup_{0\le a < c \le T} \frac{F(\1_{(a,c]})}{(c-a)^\frac{1}{\alpha}}
    \]
for $F\in C^+(\Delta)$.
By definition this means that
\[  |\xi|_{\Phi_\alpha,p} = \sup_{0\le a < c \le T} \frac{\| \xi - \xi^{(a,c]}\|_{L_p}}{(c-a)^\frac{1}{\alpha}}.\]
In \cite[Theorem 4.22]{Geiss:Ylinen:21} it is shown that $\B_2^{\Phi_2} \subseteq \D_{1,2}$ and that for
$p\in [2,\infty)$ and $\xi\in \D_{1,2} \cap L_p$ one has
\begin{equation}\label{eqn:characterization_Phi2p}
|\xi|_{\Phi_2,p} \sim_{c_{\eqref{eqn:characterization_Phi2p},p}} 
\sup_{0\le a < c \le T} \left \| \left ( \frac{1}{c-a} \int_a^c |D_s \xi|^2 \od s \right )^\frac{1}{2} \right \|_{L_p}, 
\end{equation}
where $c_{\eqref{eqn:characterization_Phi2p},p}\ge 1$ depends at most on $p$. For $p=2$ the RHS of 
\eqref{eqn:characterization_Phi2p} equals 
\[ \esssup_{s\in [0,T]} \big \| |D_s \xi| \big \|_{L_2}.\]


\subsubsection{\bf Characterization of real interpolation spaces}
\label{sec:real_interpolation}
Regarding real interpolation we refer to  \cite{Bennet:Sharpley:88,Bergh:Loefstroem:76} for
detailed information. Here we recall some basic definitions and properties for the convenience of the reader.
Given a pair of real Banach spaces $(E_0,E_1)$ that are continuously embedded into another Banach space $E$,
i.e. $(E_0,E_1)$ is an interpolation pair, we define the spaces
$E_0+E_1 := \{ x=x_0 + x_1 : x_i \in E_i \}$ with
\[ \| x \|_{E_0+E_1} := \inf \{ \| x_0\|_{E_0} + \| x_1\|_{E_1} : x=x_0 + x_1 \} \]
and $E_0\cap E_1$ with
\[ \| x \|_{E_0\cap E_1} := \max \{ \| x\|_{E_0}, \|x\|_{E_1}\}.\]
Given $\lambda \in (0,\infty)$ and $x\in E_0 + E_1$, the $K$-functional is given by
\[ K(x,\lambda;E_0,E_1) := \inf \{ \| x_0\|_{E_0} + \lambda \| x_1\|_{E_1} : x=x_0 + x_1 \}. \]
For $(\eta,q)\in (0,1)\times [1,\infty]$ the real interpolation space $(E_0,E_1)_{\eta,q}$ is given by
\[ \| x\|_{(E_0,E_1)_{\eta,q}} := \left \| \lambda^{-\eta}  K(x,\lambda;E_0,E_1) \right \|_{L_q\left ((0,\infty),\frac{\od \lambda}{\lambda}\right )}.\]
We have a lexicographical ordering, i.e.
\[ (E_0,E_1)_{\eta,q_0} \subseteq (E_0,E_1)_{\eta,q_1}
   \sptext{1}{if}{1} 1 \le q_0 \le q_1 \le \infty, \]
and, under the additional assumption that we have the continuous embedding $E_1 \subseteq E_0$,
\[  (E_0,E_1)_{\eta_0,q_0} \subseteq (E_0,E_1)_{\eta_1,q_1}
   \sptext{1}{if}{1} 0< \eta_1 < \eta_0 < 1 \sptext{1}{and}{1}
   q_0,q_1 \in [1,\infty]. \]
\medskip

By a special choice of the functional $\Phi$ the spaces
$\B_p^\Phi$ describe the standard Besov spaces on the Wiener space obtained by real interpolation:
We define the measure $\mu$ on $\cB([0,1])$ by
\begin{equation}\label{eqn:equivalence_mu}
\od \mu(r) := \frac{r}{\sqrt{1-r^2}(1-\sqrt{1-r^2})} \1_{(0,1)}(r) \od r
              \sim_{c_\eqref{eqn:equivalence_mu}}  \frac{1}{r \sqrt{1-r}} \1_{(0,1)}(r) \od r
\end{equation}
for some $c_\eqref{eqn:equivalence_mu}\ge 1$,
and for $\eta\in (0,1)$ the kernel $\cK_\eta:[0,1]\to \R$ by
\begin{equation}\label{eqn:equivalence_K}
\cK_\eta(r) := \frac{1}{(1-\sqrt{1-r^2})^\frac{\eta}{2}} \1_{(0,1]}(r)
        \sim_{c_{\eqref{eqn:equivalence_K},\eta}} \frac{1}{r^\eta}\1_{(0,1]}(r)
\end{equation}
for some $c_{\eqref{eqn:equivalence_K},\eta}\ge 1$.
For $(\eta,q) \in (0,1)\times [1,\infty]$ and $p\in [2,\infty)$ we define the
interpolation space
\[ \B_{p,q}^\eta := (L_p,\D_{1,p})_{\eta,q}.\]
Then it was shown in \cite[Theorem 4.16]{Geiss:Ylinen:21} that
\begin{equation}\label{eqn:real_interpolation_via_decoupling}
\| \xi \|_{\B_{p,q}^\eta}
   \sim_{c_\eqref{eqn:real_interpolation_via_decoupling}}  \|\xi\|_{\B_p^{\Phi^{\cK_\eta,\mu,q)}}}
\end{equation}
with
\[\Phi^{(\cK_\eta,\mu,q)}(F) := \left \| \cK_\eta(\cdot) F(\varphi_\cdot) \right \|_{L_q([0,1],\mu)} \]
where
$ c_\eqref{eqn:real_interpolation_via_decoupling}
 =c_\eqref{eqn:real_interpolation_via_decoupling}(p,q,\eta)
 \ge 1$ and $\varphi_r:\equiv r$.


\subsubsection{\bf Characterization of $\D_{1,2}$}
\label{sec:characterization_D12}

The characterization of $\D_{1,2}$ in \cref{statement:characterization_D12} below
completes the results from \cite{Geiss:Ylinen:21}. On the way to its proof we obtain in
\cref{statement:chaos_isotropic_decoupling} a general relation that describes the
$L_2$-impact of the isonormal decoupling 
$\xi \mapsto \xi^r$ from \eqref{eqn:notation:coupling_rv}
in terms of a multiplier in the Wiener chaos expansion.
\medskip

We recall that for $\xi \in L_2(\Omega,\cF,\P)$ there is a chaos decomposition
\[ \xi = \E \xi + \sum_{n=1}^\infty I_n(f_n) \]
with symmetric kernels $f_n:( (0,T] \times \{1,\ldots,N\})^n \to \R$
(see \cite[Example 1.1.2]{Nualart:06}).
\medskip

\begin{prop}
\label{statement:chaos_isotropic_decoupling}
For $\xi \in L_2(\Omega,\cF,\P)$ and $r\in [0,1]$ one has
\[    \E|\xi-\xi^\br|^2 = 2 \sum_{n=1}^\infty \left [ 1 -  |1-r^2|^\frac{n}{2}  \right ] \E |I_n(f_n)|^2.\]
\end{prop}

\begin{proof}
To describe the coupling between $\xi$ and $\xi^\br$
we use the chaos expansion on the extended space
$(\overline{\Omega},\overline{\cF},\overline{\P})$ with the $2N$-dimensional
Brownian motion $\overline{W}=(W,W')$ and kernels defined on
$( (0,T] \times \{1,\ldots,N,N+1,\ldots,2N\})^n$ where
$\{1,\ldots,N\}$ refers to the Brownian motion $W$ and
$\{N+1,\ldots,2N\}$  to the Brownian motion $W'$.
\smallskip

For the random variable $\xi^\br$ for $r\in [0,1]$ we get the symmetric kernels
\[ f^\br_n:( (0,T] \times \{1,\ldots,N,N+1,\ldots,2N\})^n \to \R \]
given by
\[ f_n^\br((t_1,j_1+\varepsilon_1 N),\ldots,(t_n,j_n + \varepsilon_n N))
   := q(r,\varepsilon_1) \cdots q(r,\varepsilon_n)
      f_n((t_1,j_1),\ldots,(t_n,j_n)) \]
for $j_1,\ldots,j_n\in \{1,\ldots,N\}$ and $\varepsilon_1,\ldots,\varepsilon_n\in \{0,1\}$
with $q(r,0) := \sqrt{1-r^2}$ and $q(r,1):= r$.
This can be shown by inspecting the operation $\xi\mapsto \xi^\br$ on linear
combinations of products of form
\[ (W_{c_1,j_1}-W_{a_1,j_1}) \cdots (W_{c_n,j_n}-W_{a_n,j_n}) \]
with $0\le a_k < c_k \le T$ and 
\[ ((a_k,c_k]\times \{ j_k\}) \cap ((a_l,c_l]\times \{ j_l\})=\emptyset
   \sptext{1}{if}{1} k\not = l. \]
So we have
\[ \xi^\br = \E \xi + \sum_{n=1}^\infty I_n(f_n^\br). \]
This implies that
\begin{align*}
    \E|\xi-\xi^\br|^2
& = \sum_{n=1}^\infty \E |I_n(f_n-f_n^\br)|^2 \\
& = \sum_{n=1}^\infty n! \left \|f_n-f_n^\br \right \|^2_{L_2^n} \\
& = \sum_{n=1}^\infty n! \sum_{\varepsilon_1,\ldots,\varepsilon_n=0}^1  \sum_{j_1,\ldots,j_n=1}^N \Bigg [
    \left ( q(0,\varepsilon_1) \cdots q(0,\varepsilon_n) - q(r,\varepsilon_1) \cdots q(r,\varepsilon_n) \right )^2  \\
&   \hspace{2em} \times  \int_0^T\cdots \int_0^T |f_n((t_1,j_1),\ldots,(t_n,j_n))|^2 \od t_n \cdots \od t_1 \Bigg ] \\
& = \sum_{n=1}^\infty \left [ \sum_{\varepsilon_1,\ldots,\varepsilon_n=0}^1
    \left ( q(0,\varepsilon_1) \cdots q(0,\varepsilon_n) - q(r,\varepsilon_1) \cdots q(r,\varepsilon_n) \right )^2 \right ]
    \E |I_n(f_n)|^2 \\
& = 2 \sum_{n=1}^\infty \left [ 1 -  |\sqrt{1-r^2}|^n  \right ] \E |I_n(f_n)|^2. \qedhere
\end{align*}
\end{proof}

The next statement allows to characterize $\xi\in \D_{1,2}$ by
knowing the behavior of  $\E|\xi-\xi^\br|^2$ for $r\in (0,1]$:

\begin{lemm}
\label{statement:characterization_D12_abstract}
For $(a_n)_{n=1}^\infty \subseteq [0,\infty)$ the following assertions are equivalent:
\begin{enumerate}[{\rm (1)}]
\item \label{item:1:statement:characterization_D12}
      There exists a constant $c_1>0$ such that
      \[ \sum_{n=1}^\infty \left [ 1 -  |1-r|^\frac{n}{2}  \right ] a_n \le c_1 r
         \sptext{1}{for all}{1} r\in [0,1].\]
\item \label{item:2:statement:characterization_D12}
      There exists a constant $c_2>0$ such that
      \[ \sum_{n=1}^\infty n a_n \le c_2. \]
\end{enumerate}
we can choose
$c_2 := 2 c_1$ in \eqref{item:1:statement:characterization_D12} $\Rightarrow$ \eqref{item:2:statement:characterization_D12}, and
$c_1 := c_2$          in \eqref{item:2:statement:characterization_D12} $\Rightarrow$ \eqref{item:1:statement:characterization_D12}.
\end{lemm}

\begin{proof}
\eqref{item:1:statement:characterization_D12} $\Rightarrow$ \eqref{item:2:statement:characterization_D12}
We observe that
\[   \lim_{r\downarrow 0} \frac{1 -  |1-r|^\frac{n}{2}}{r}\\
   = \lim_{r\downarrow 0} \frac{1 -  |1-r|^n}{r (1 +  |1-r|^\frac{n}{2})} \\
   = \frac{n}{2} \]
so that we can choose $c_2 := 2 c_1$.
\medskip

\eqref{item:2:statement:characterization_D12} $\Rightarrow$ \eqref{item:1:statement:characterization_D12}
We start with
\begin{align*}
     \sum_{n=1}^\infty \left [ 1 -  |1-r|^\frac{n}{2} \right ] a_n
& =  \sum_{n=1}^\infty \frac{\left [ 1 -  |1-r|^\frac{n}{2} \right ]}{n} n a_n \\
&\le \left [ \sup_{n\in \bN} \frac{\left [ 1 -  |1-r|^\frac{n}{2}  \right ]}{n}\right ]
     \sum_{n=1}^\infty n a_n.
\end{align*}
For $n\ge 2$ we check that
\[  1 -  |1-r|^\frac{n}{2}  \le \frac{1}{2}  n r.\]
For $n=2$ we have equality, let $n\ge 3$.
For $r=0$ both sides coincide. Differentiating both sides with respect to $r$  gives
\[ \frac{n}{2} (1-r)^{\frac{n}{2} -1} \le \frac{1}{2} n.\]
Finally, for $n=1$ we have that
\[  1 -  \sqrt{1-r} = \frac{r}{1 +  \sqrt{1-r}} \le r.\]
Hence we can choose $c_1 := c_2$.
\end{proof}

\cref{statement:chaos_isotropic_decoupling} and \cref{statement:characterization_D12_abstract}
immediately imply the following characterization:

\begin{coro}
\label{statement:characterization_D12}
For $\xi \in L_2$ one has $\xi \in \D_{1,2}$ if and only if
\[ \sup_{r\in (0,1]} \frac{\| \xi - \xi^\br \|_{L_2}}{r} < \infty. \]
Moreover, if $\xi \in \D_{1,2}$, then one has
\[     \frac{1}{2}  \| D\xi\|_{L_2(\Omega \times [0,T];\R^N)}
   \le \sup_{r\in (0,1]} \frac{\| \xi - \xi^\br \|_{L_2}}{r} \le \| D\xi\|_{L_2(\Omega \times [0,T];\R^N)}.\]
\end{coro}
\bigskip

Now we consider, instead of a random variable from  $L_2(\Omega,\cF,\P)$, an $\ell_2$-valued
random variable $\xi \in L_2(\Omega,\cF,\P;\ell_2)$ and interpret $\xi$ as sequence $(\xi^l)_{l=1}^\infty\subseteq L_2(\Omega,\cF,\P)$
with
\[ \| \xi\|_{L_2(\ell_2)}  = \| \xi\|_{L_2(\Omega,\cF,\P;\ell_2)} :=  \sqrt{\sum_{l=1}^\infty \| \xi^l\|_{L_2}^2}<\infty. \]
Correspondingly, we obtain a chaos decomposition with symmetric kernels $f_n:( (0,T] \times \{1,\ldots,N\})^n \to \ell_2$ that can written as
$f_n = (f_n^l)_{l=1}^\infty$. In this setting \cref{statement:chaos_isotropic_decoupling} yields to the following extension:

\begin{prop}
\label{statement:chaos_isotropic_decoupling_H_valued}
For $\xi \in L_2(\Omega,\cF,\P;\ell_2)$ and $r\in [0,1]$ one has
\[    \E\|\xi-\xi^\br\|^2_{\ell_2} = 2 \sum_{n=1}^\infty \left [ 1 -  |1-r^2|^\frac{n}{2}  \right ] \E \left \|(I_n(f^l_n))_{l=1}^\infty\right \|^2_{\ell_2}.\]
\end{prop}
Finally, in the same way as \cref{statement:characterization_D12} we obtain the following characterization:

\begin{coro}
\label{statement:characterization_D12_H_valued}
For $\xi \in L_2(\ell_2)$ one has $\xi \in \D_{1,2}(\ell_2)$ if and only if
\[ \sup_{r\in (0,1]} \frac{\| \xi - \xi^\br \|_{L_2(\ell_2)}}{r} < \infty. \]
Moreover, if $\xi \in \D_{1,2}(\ell_2)$, then one has
\[     \frac{1}{2}  \| D\xi\|_{L_2(\Omega \times [0,T];\R^N\times \ell_2)}
   \le \sup_{r\in (0,1]} \frac{\| \xi - \xi^\br \|_{L_2(\ell_2)}}{r} \le \| D\xi\|_{L_2(\Omega \times [0,T];\R^N\times \ell_2)}.\]
\end{coro}


\section{Coupling of SDEs - the Lipschitz case}
\label{sec:Lipschitz}

We begin this section with \cref{statement:Lipschitz_case_path} where we apply
the coupling technique to path-dependent SDEs driven by Lipschitz coefficients. 
In \cref{statement:Malliavin_differentiability} we show the stability in the sense 
of Malliavin differentiability, in \cref{statement:real_interpolation_SDE_Lipschitz}
we deduce the stability with respect to fractional regularity expressed in terms 
of Besov spaces obtained by real interpolation. 
\medskip

Our assumptions are as follows:

\begin{assumption}[Lipschitz assumptions]
\label{ass:bsigma_path}
We let
\begin{align*}
     b: & [t,T] \times \Oz \times C([t,T];\R^d)  \to \R^d, \\
\sigma: & [t,T] \times \Oz \times C([t,T];\R^d)  \to \R^{d\times N}
\end{align*}
such that $b(t,\cdot,\cdot)\equiv 0$ and $\sigma(t,\cdot,\cdot)\equiv 0$,
and such that $b(\cdot,\cdot,x)$ and $\sigma(\cdot,\cdot,x)$ are 
$\cP_{t,T}$-measurable for all $x\in C([t,T];\R^d)$.
Furthermore, we assume that
\begin{align}
|b(s,\omega,x) - b(s,\omega,y)|     & \leq L_b(s,\omega) \|x-y\|_{C([t,T])},\label{cd2_path}\\
        |\sigma(s,\omega,x) - \sigma(s,\omega,y)|_\hs & \leq L_{\sigma} \|x-y\|_{C([t,T])}, \label{cd1_path}\\
         |b(s,\omega,x)| & \leq K_b (1 + \|x\|_{C([t,T])})\label{cd4_path}, \\
                   |\sigma(s,\omega,x)|_{\hs} & \leq K_{\sz} (1 + \|x\|_{C([t,T])}),\label{cd3_path},\\
     b(s,\omega,x) & = b(s,\omega,x(\cdot \wedge s)), \label{cd5_path} \\
 \sigma(s,\omega,x) & = \sigma(s,\omega,x(\cdot \wedge s)) \label{cd6_path}
\end{align}
on $[t,T]\times \Omega\times C([t,T];\R^d)$,
where $L_{\sigma},K_b,K_\sigma\ge 0$ and the non-negative progressively measurable
process $L_b=(L_b(s,\cdot))_{s\in [t,T]}$ satisfies
$L_b \in \BMO(S_2)$.
\end{assumption}
\medskip

We let $p\in (0,\infty)$ and assume a solution to
\begin{equation}\label{eqn:SDE_random_coefficients_path}
X_s^{t,\xi} = \xi + \int_t^s b(u, X^{t,\xi})\od u + \int_t^s \sigma(u, X^{t,\xi})\od W_u
\sptext{1}{for}{1} s\in [t,T] \mbox{ a.s.}
\end{equation}
such that
\begin{enumerate}
\item $\xi \in L_{p \vee 2}$,
\item $(X^{t,\xi}_s)_{s\in [t,T]}$ is $(\cF_s)_{s\in [t,T]}$-adapted and path-wise continuous,
\item \label{cond:3:eqn:SDE_random_coefficients_path}
      $\E \int_t^T |\sigma(u,X^{t,\xi})|^2_\hs  \od u + \E \int_t^T |b(u,X^{t,\xi})|^2  \od u <\infty$.
\end{enumerate}
The existence and uniqueness of such equations is discussed in \cite[Sections V.2.8-12]{Rogers:Williams:2:2nd} and
\cite[Chapter IX]{Revuz:Yor:1999}. For our investigations only the existence in \eqref{eqn:SDE_random_coefficients_path} is required, not more.
However, it is standard to verify along the proof of \cref{statement:Lipschitz_case_path} that one has
uniqueness under \cref{ass:bsigma_path} of the solution to \eqref{eqn:SDE_random_coefficients_path} for a given
stochastic basis.
The above condition \eqref{cond:3:eqn:SDE_random_coefficients_path} is a result of the assumptions \eqref{cd4_path} and \eqref{cd3_path},
$\xi\in L_2$, and of \cref{statement:Xt_vs_starting_value_path}. We apply \cref{statement:transference_sde} to a solution of \eqref{eqn:SDE_random_coefficients_path}  coordinate-wise with $b^0$ and $\sigma^0$ not depending on $K^0$
and it follows
\begin{enumerate}
\item $\xi^\varphi \in C_0(\xi)\cap  L_{p \vee 2}$,
\item $((b(u,X^{t,\xi}))^\varphi)_{u\in [t,T]} \in C_{[t,T]}^\cP \left ( (b(u,X^{t,\xi}))_{u\in [t,T]} \right )$,
\item $((\sigma(u,X^{t,\xi}))^\varphi)_{u\in [t,T]} \in C_{[t,T]}^\cP \left ( (\sigma(u,X^{t,\xi}))_{u\in [t,T]} \right )$,
\item $\E \int_t^T |(\sigma(u,X^{t,\xi}))^\varphi|^2_\hs  \od u + \E \int_t^T |(b(u,X^{t,\xi}))^\varphi|^2  \od u <\infty$,
\item $(X_s^{t,\xi,\varphi})_{s\in [t,T]} \in C_{[t,T]}^{\rm c,adap}((X^{t,\xi}_s)_{s\in [t,T]})$,
\end{enumerate} 
and
\begin{equation}\label{eqn:SDE_random_coefficients_path_transformed}
X_s^{t,\xi,\varphi} = \xi^\varphi + \int_t^s (b(u, X^{t,\xi}))^\varphi \od u + \int_t^s (\sigma(u, X^{t,\xi}))^\varphi \od W_u^\varphi
\sptext{1}{for}{1} s\in [t,T] \mbox{ a.s.}
\end{equation}
We also use the following observation:

\begin{rema}
We know that $(b(u,X^{t,\xi}))_{u\in [t,T]}$ and $(\sigma(u,X^{t,\xi}))_{u\in [t,T]}$ are
$\cP_{t,T}$-measurable by \cref{statement:composition_proedictable_new}. Transforming
$X^{t,\xi}$ we get an $\overline{\cP}_{t,T}$-measurable continuous  process
$(X^{t,\xi,\varphi}_u)_{u\in [t,T]}$ so that 
$(b(u,X^{t,\xi,\varphi}))_{u\in [t,T]}$ and $(\sigma(u,X^{t,\xi,\varphi}))_{u\in [t,T]}$ become
$\overline{\cP}_{t,T}$-measurable again by \cref{statement:composition_proedictable_new}.
\end{rema}
\smallskip

The general result about coupling in the Lipschitz case is as follows:
\smallskip

\begin{theo}
\label{statement:Lipschitz_case_path}
Suppose \cref{ass:bsigma_path} with $p\in (0,\infty)$ and let
$X^{t,\xi}$ be a solution to \eqref{eqn:SDE_random_coefficients_path} and
$X^{t,\xi,\varphi}\in \cC_{[t,T]}^\cadap(X^{t,\xi})$.
Then we have
\begin{align}
&\hspace*{-1em} \left \| \sup_{s \in [t, T]}|X_s^{t,\xi,\varphi} - X_s^{t,\xi}| \right \|_{L_p}
   \leq c_\eqref{eqn:statement:Lipschitz_case_path} \Bigg [
 \left (\int_t^T |\varphi(u)|^2 \od u \right )^{\frac{1}{2}} \| 1+ |\xi| \|_{L_p} +  \| \xi-\xi^\varphi\|_{L_p} \nonumber \\
& \hspace*{9em} + \bigg\|  \int_t^T \Big|(b(u,X^{t,\xi}))^\varphi - b(u,X^{t,\xi,\varphi})\Big|\od u\bigg\|_{L_p}\nonumber\\
& \hspace*{9em} + \bigg\|  \bigg(\int_t^T  \Big| (\sigma(u,X^{t,\xi}))^\varphi - \sigma(u,X^{t,\xi,\varphi})\Big|_\hs^2 \od u
                  \bigg)^{\frac{1}{2}}\bigg\|_{L_p}\Bigg ],
   \label{eqn:statement:Lipschitz_case_path}
\end{align}
where $c_\eqref{eqn:statement:Lipschitz_case_path}=c_\eqref{eqn:statement:Lipschitz_case_path}(p,T,\| L_b\|_{\BMO(S_2)},L_\sigma,K_b,K_\sigma) > 0$.
\end{theo}
\smallskip

\begin{proof}
To shorten the notation we let $X_s:=X_s^{t,\xi}$
and $X_s^\varphi:=X_s^{t,\xi,\varphi}$.
\medskip

\underline{The case $p\in [2,\infty)$}:
For $r\in [t,T]$ and an $(\overline{\cF_s})_{s\in [0,T]}$-stopping time $\tau:\overline{\Omega} \to [t,T]$ we have,
a.s.,
\begin{align*}
&   X_{\tau\wedge r}^\varphi -X_{\tau\wedge r} \\
& = \xi^\varphi -\xi + \int_t^{\tau\wedge r} [(b( u, X))^\varphi - b(u, X)]\od u \\
&+ \int_t^{\tau\wedge r} [(\sigma(u, X))^\varphi - \sigma(u, X)]\sqrt {1-|\varphi(u)|^2}\od W_u \\
&+\int_t^{\tau\wedge r}  (\sigma(u, X))^\varphi \varphi(u) \od W'_u
  - \int_t^{\tau\wedge r} \sigma(u,X) \left  [1-\sqrt{1-|\varphi(u)|^2} \right ] \od W_u.
\end{align*}
Assume a set $A\in \overline{\cF}_t$ of positive measure and let
$\oP_A$ be the normalized restriction of $\oP$ to $A$.
By the Burkholder-Davis-Gundy inequality, $\sqrt {1-|\varphi(u)|^2}\le 1$, and
$1-\sqrt{1-|\varphi(u)|^2} \le \varphi(u)$, we continue to
\begin{align*}
& \hspace{-2em}\left  \|\sup_{s \in [t, {\tau\wedge r}]}|X_s^\varphi - X_s| \right \|_{L_p(A,\oP_A)} \\
& \leq \| \xi^\varphi -\xi\|_{L_p(A,\oP_A)} +  \left \| \int_t^{\tau\wedge r} \Big|(b(u, X))^\varphi - b(u, X)\Big|\od u\right \|_{L_p(A,\oP_A)} \\
& + \beta_p \left \| \left (  \int_t^{\tau\wedge r}  \left |(\sigma(u, X))^\varphi - \sigma(u, X)\right |^2_\hs \od u \right )^{\frac{1}{2}}
   \right \|_{L_p(A,\oP_A)} \\
&+ \beta_p  \left \| \left (\int_t^T  \Big |\sigma(u,X)\varphi(u)\Big |_\hs^2\od u\right )^{\frac{1}{2}} \right \|_{L_p(A,\oP_A)} \\
&+ \beta_p  \left \| \left (\int_t^T  \Big |(\sigma(u,X))^\varphi\varphi(u)\Big |_\hs^2\od u\right )^{\frac{1}{2}} \right \|_{L_p(A,\oP_A)}.
    \end{align*}
Now we estimate separately each term from the right-hand side.
\smallskip

\underline{Second term:} First we observe
\begin{align*}
&     \left \| \int_t^{\tau\wedge r} \Big|(b(u, X))^{\varphi} - b(u, X)\Big|\od u\right \|_{L_p(A,\oP_A)} \\
& \leq \left \| \int_t^\tau \Big|(b(u, X))^\varphi - b(u, X^\varphi) \Big|\od u\right \|_{L_p(A,\oP_A)} \\
&  \hspace{10em} + \left \| \int_t^{\tau\wedge r} \Big|b(u, X^\varphi)  - b(u, X)\Big|\od u\right \|_{L_p(A,\oP_A)}.
\end{align*}
Using the Lipschitz property of $b$ and Fefferman's inequality (\cref{statement:Fefferman_inequality}), the last term we bound from above by
\begin{align*}
&    \left \| \int_t^{\tau\wedge r} L_b(u) \left \|X^{\varphi} - X \right \|_{C\iue}\od u\right \|_{L_p(A,\oP_A)}\\
& = \left \| \int_t^r L_b(u) \1_{\{u\in [t,\tau\wedge r]\}} \left\| X^{\varphi} - X \right \|_{C\iue}\od u\right \|_{L_p(A,\oP_A)} \\
& \leq  \sqrt{2p}  \| L_b\|_{\BMO(S_2)} \left \| \left ( \int_t^r \1_{\{u\in [t,\tau\wedge r]\}} \left \|X^\varphi - X \right \|^2_{C\iue} \od u \right )^\frac{1}{2}\right \|_{L_p(A,\oP_A)}
 \\
& \leq  \sqrt{2p}  \| L_b\|_{\BMO(S_2)} T^{\frac{1}{2}-\frac{1}{p}}
  \left \| \left ( \int_t^r \sup_{v\in [t,\tau\wedge u]} |X_v^{\varphi} - X_v \Big|^p \od u \right )^\frac{1}{p}\right \|_{L_p(A,\oP_A)},
\end{align*}
where we first extend $L_b$ canonically to the stochastic basis $(\overline{\Omega},\overline{\cF},\overline{\P},(\overline{\cF}_t)_{t\in [0,T]})$
and observe that the constant in \cref{definition:BMO(S_2)} does not change, and then we restrict this extended stochastic basis
to $A\in \overline{\cF}_t$.
\smallskip

\underline{Third term:} Similarly, we start with
\begin{align*}
& \left \| \bigg(  \int_t^{\tau\wedge r} \Big |(\sigma(u, X))^\varphi - \sigma(u, X_u)\Big |_\hs^2\od u \bigg)^{\frac{1}{2}} \right \|_{L_p(A,\oP_A)}\\
& \le \left \| \bigg(  \int_t^\tau  \Big |(\sigma(u, X))^\varphi - \sigma(u, X^\varphi)\Big |_\hs^2\od u \bigg)^{\frac{1}{2}} \right \|_{L_p(A,\oP_A)} \\
&  \hspace{13em}  + \left \| \bigg(  \int_t^{\tau\wedge r}  \Big |\sigma(u, X^\varphi) - \sigma(u, X)\Big |_\hs^2\od u \bigg)^{\frac{1}{2}} \right \|_{L_p(A,\oP_A)},
\end{align*}
where the last term can be upper bounded by the Lipschitz property of $\sigma$ as follows:
\begin{multline*}
 \left \| \bigg(  \int_t^{\tau\wedge r}  L_\sigma^2 \left \| X^{\varphi} - X \right \|_{C\iue}^2 \od u \bigg)^{\frac{1}{2}} \right \|_{L_p(A,\oP_A)} \\
 \leq  L_\sigma T^{\frac{1}{2}-\frac{1}{p}} \left \| \left ( \int_t^r \sup_{v\in [t,\tau\wedge u]} |X_v^{\varphi} - X_v \Big|^p \od u\right )^\frac{1}{p}
  \right \|_{L_p(A,\oP_A)}.
\end{multline*}

\underline{Fourth and fifth term:} Using \cref{statement:Xt_vs_starting_value_path} we get
\begin{align*}
& \hspace*{-4em} \left \| \bigg(\int_t^T  \Big | \sigma(u, X)\varphi(u) \Big|_\hs^2\od u\bigg)^{\frac{1}{2}} \right \|_{L_p(A,\oP_A)} \\
& \leq K_\sigma \left \| \bigg(\int_t^T   (1 + \|X\|_{C\iue})^2 |\varphi(u)|^2 \od u\bigg)^{\frac{1}{2}} \right \|_{L_p(A,\oP_A)} \\
& \leq K_\sigma \left (\int_t^T |\varphi(u)|^2 \od u \right )^{\frac{1}{2}}\left \| \sup_{u \in [t, T]}[1+|X_u|]\right \|_{L_p(A,\oP_A)} \\
& \leq K_\sigma c_{\eqref{eqn:statement:Xt_vs_starting_value_path},p}
       \left (\int_t^T |\varphi(u)|^2 \od u \right )^{\frac{1}{2}}\left \| 1+|\xi|\right \|_{L_p(A,\oP_A)}.
\end{align*}

We know that
\[    X_s^\varphi
   = \xi^\varphi + \int_t^s (b( u, X))^\varphi \od u
                + \int_t^s (\sigma(u, X))^\varphi \od W_u^\varphi \]
and that, $\lambda\otimes \overline{\P}$ a.e.,
\[ |(b(s, X))^\varphi|     \le K_b      [1+\sup_{u\in [t,s]} |X_u^\varphi|]
   \sptext{1}{and}{1}
  |(\sigma(s, X))^\varphi| \le K_\sigma [1+\sup_{u\in [t,s]} |X_u^\varphi|]. \]
So by \cref{statement:Xt_vs_starting_value_path} we get as above that
\begin{multline*}
     \left \| \bigg(\int_t^T  \Big | (\sigma(u, X))^\varphi \varphi(u) \Big|_\hs^2\od u\bigg)^{\frac{1}{2}} \right \|_{L_p(A,\oP_A)} \\
   \leq K_\sigma c_{\eqref{eqn:statement:Xt_vs_starting_value_path},p}
       \left (\int_t^T |\varphi(u)|^2 \od u \right )^{\frac{1}{2}}\left \| 1+|\xi^\varphi|\right \|_{L_p(A,\oP_A)}.
\end{multline*}

\underline{Summarizing}, we have
\begin{align*}
& \left  \|\sup_{s \in [t, \tau\wedge r]}|X_s^\varphi - X_s| \right \|_{L_p(A,\oP_A)}
  \le \| \xi^\varphi -\xi\|_{L_p(A,\oP_A)}  \\
& + \beta_p K_\sigma c_{\eqref{eqn:statement:Xt_vs_starting_value_path},p}
      \left (\int_t^T |\varphi(u)|^2 \od u \right )^{\frac{1}{2}}\left [ \| 1+|\xi| \|_{L_p(A,\oP_A)}+\| 1+|\xi^\varphi| \|_{L_p(A,\oP_A)} \right ] \\
& \hspace*{4em}    + \left \| \int_t^\tau \Big|(b(u, X))^\varphi  - b(u, X^\varphi)\Big|\od u\right \|_{L_p(A,\oP_A)} \\
& \hspace*{4em}    + \beta_p \left \| \bigg(  \int_t^\tau  \Big |(\sigma(u, X))^\varphi - \sigma(u, X^\varphi)\Big |_\hs^2\od u \bigg)^{\frac{1}{2}}
                \right \|_{L_p(A,\oP_A)} \\
& + T^{\frac{1}{2}-\frac{1}{p}} \left [ \sqrt{2p} \| L_b\|_{\BMO(S_2)} + \beta_p L_\sigma \right ]
    \left [ \int_t^r \left \| \sup_{v\in [t,\tau\wedge u ]} |X_v^{\varphi} - X_v | \right \|_{L_p(A,\oP_A)}^p \od u \right ]^\frac{1}{p}.
\end{align*}

By Lemmas \ref{statement:Xt_vs_starting_value_path} and \ref{statement:properties_CCtTadapted} we have
$\| \sup_{v\in [t,T]} |X_v^{\varphi}|\|_{L_p}=\| \sup_{v\in [t,T]} |X_v| \|_{L_p}<\infty$.
Raising the last estimate for
$\left  \|\sup_{s \in [t, \tau\wedge r]}|X_s^\varphi - X_s| \right \|_{L_p(A,\oP_A)}$
to the $p$-th power and applying Gronwall's lemma implies
\begin{align*}
&     \left  \|\sup_{s \in [t, \tau]}|X_s^\varphi - X_s| \right \|_{L_p(A,\oP_A)} \\
& \le c \Bigg [ \| \xi^\varphi -\xi\|_{L_p(A,\oP_A)} \!  + \! \left (\int_t^T |\varphi(u)|^2 \od u \right )^{\frac{1}{2}}
      \left [ \| 1+|\xi| \|_{L_p(A,\oP_A)} + \| 1+|\xi^\varphi|\|_{L_p(A,\oP_A)} \right ] \\
& \hspace{2em}    + \left \| \int_t^\tau \Big|(b(u,X))^\varphi  - b(u, X^\varphi)\Big|\od u\right \|_{L_p(A,\oP_A)} \\
& \hspace{2em}    + \left \| \bigg(  \int_t^\tau  \Big |(\sigma(u, X))^\varphi) - \sigma(u, X^\varphi)\Big |_\hs^2\od u \bigg)^{\frac{1}{2}} \right \|_{L_p(A,\oP_A)}
    \Bigg ]
\end{align*}
with some $c=c(p,T,\| L_b\|_{\BMO(S_2)},L_\sigma,K_b,K_\sigma)>0$,
where one of the last two terms on the RHS or both might be infinite.
By choosing $A:=\overline{\Omega}$, $\tau:=T$, and using
$\| 1+|\xi| \|_{L_p} = \| 1+|\xi^\varphi|\|_{L_p}$
we obtain \cref{statement:Lipschitz_case_path} for $p\in [2,\infty)$.
\bigskip

\underline{The case $p\in (0,2)$}:
We define the processes
\begin{align*}
A_s := & |X_s^\varphi - X_s|, \\
B_s := & |\xi^\varphi -\xi|  + \left (\int_t^T |\varphi(u)|^2 \od u \right )^{\frac{1}{2}} [| 1+|\xi^\varphi|+|\xi|]\\
    &   + \int_t^s \Big|(b(u, X))^\varphi  - b(u, X^\varphi)\Big|\od u
        + \bigg(  \int_t^s  \Big |(\sigma(u, X_u))^\varphi - \sigma(u, X^\varphi)\Big |_\hs^2 \od u \bigg)^{\frac{1}{2}}.
\end{align*}
The previous case $p\in [2,\infty)$ implies that
$\E [|A_\tau|^2|\overline{\cF}_t] \le c \E [|B_\tau|^2|\overline{\cF}_t]$
with some constant $c=c(T,\| L_b\|_{\BMO(S_2)},L_\sigma,K_b,K_\sigma)$
for all stopping times $\tau:\overline{\Omega} \to [t,T]$. So we may apply Lenglart's inequality (\cref{statement:Lenglart}) and obtain
the statement for $p\in (0,2)$ as well.
\end{proof}
\medskip

\begin{defi}
\label{definition:fractional_potential}
For a coupling function $\varphi:[0,T]\to [0,1]\in \cD$
we say that $(b, \sigma)$ has a fractional $\varphi$-potential $(U,V)$
with the parameters $(\beta,\gamma;H,K)$ provided that
\begin{enumerate}
\item $\beta,\gamma \in (0, 1]$,
\item $H,K:C([t,T];\R^d)\to [0,\infty)$ are continuous,
\item $U, V:[t,T]\times \Omega\to \R$ are $\cP_{t,T}$-measurable,
\item there are
     \begin{align*}
          b^\varphi: & [t,T] \times \overline{\Oz} \times C([t,T];\R^d)  \to \R^d, \\
     \sigma^\varphi: & [t,T] \times \overline{\Oz} \times C([t,T];\R^d)  \to \R^{d\times N}
     \end{align*}
     satisfying \eqref{cd2_path}-\eqref{cd6_path} and
     \[ h^\varphi(\cdot,\cdot.x) \in \cC_{[t,T]}^\cP(h(\cdot,\cdot,x))
        \sptext{1}{for}{1} h\in  \{b,\sigma\} \sptext{.5}{and}{.5} x\in C([t,T];\R^d), \]
\item and $\lambda|_{[t,T]} \otimes \overline{\P}$ a.e. one has
      for all $x\in C([t,T];\R^d)$ that
      \begin{align*}
             |b^\varphi(s,x) - b(s,x)|
      & \leq |U^\varphi(s) -U(s)|^\beta H(x(\cdot\wedge s)), \\
             |\sigma^\varphi(s,x) - \sigma(s,x)|_\hs
      & \leq |V^\varphi(s) -V(s)|^\gamma K(x(\cdot\wedge s)).
      \end{align*}
\end{enumerate}
For $\beta=\gamma=1$
and $H=K\equiv 1$ we say that $(b, \sigma)$ has a Lipschitz $\varphi$-potential $(U,V)$.
\end{defi}
\smallskip

For example, the functions $H$ and $K$ can be chosen to be
$H(x)=K(x):=|x|$ when dealing with linear SDEs with random coefficients.
Moreover, $(U,V)$ can be chosen in dependence on the coupling function
$\varphi$: For example, if we can choose 
$h^\varphi(s,(\omega,\omega'),x) =  h(s,\omega,x)$,
then $U$ and $V$ can be chosen to be zero. The basic example behind \cref{definition:fractional_potential}
are controlled diffusions:
\smallskip

\begin{exam}[Controlled diffusions]
\label{statement:controlled_diffusions}
We assume 
\begin{align*}
B &:[t,T]\times \R\times C([t,T];\R^d)  \to \R^d,\\
\Sigma &:[t,T]\times \R\times C([t,T];\R^d)  \to \R^{d\times N},
\end{align*}
such that 
\begin{align*}
B(s,z,x) & = B(s,z,x(\cdot\wedge s)),\\  
\Sigma(s,z,x) & = \Sigma(s,z,x(\cdot\wedge s)),\\
|B(s,z_1,x_1) - B(s,z_2,x_2)|           &\le |z_1-z_2|^\beta  H(x_1) + L_b      \| x_1-x_2\|_{C([t,T])}, \\
|\Sigma(s,z_1,x_1) - \Sigma(s,z_2,x_2)| &\le |z_1-z_2|^\gamma K(x_1) + L_\sigma \| x_1-x_2\|_{C([t,T])}, \\
|B(s,z,x)| & \le K_b (1+\|x\|_{C([t,T])}),\\
|\Sigma(s,z,x)| & \le K_\sigma (1+\|x\|_{C([t,T])}), 
\end{align*}
and such that $s\mapsto H(s,z,x)$ is measurable for all $(z,x)\in \R\times C([t,T])$ and $H\in \{B,\Sigma\}$. If 
$U,V:[t,T]\times \Omega \to \R$
are  $\cP_{t,T}$-measurable and we set
\begin{align*}
     b(s,\omega,x) &:= B(s,U_s(\omega),x),\\
\sigma(s,\omega,x) &:= \Sigma(s,V_s(\omega),x),
\end{align*}
then $(U,V)$ is a fractional $\varphi$-potential with parameters 
$(\beta,\gamma;H,K)$ for all coupling functions $\varphi$
because we can choose 
$b^\varphi(s,(\omega,\omega'),x) := B(s,U_s^\varphi(\omega,\omega'),x)$ and
$\sigma^\varphi(s,(\omega,\omega'),x) := \Sigma(s,V_s^\varphi(\omega,\omega'),x)$
by \cref{statement:transference_composition_new} with $h^0=h^1$ not depending on $(\omega,x)$.
\end{exam}
\medskip

From \cref{statement:Lipschitz_case_path} we immediately get:
\medskip

\begin{coro}\label{cor:statement:Lipschitz_case_new}
Assume the conditions of \cref{statement:Lipschitz_case_path}. Provided that $(b, \sigma)$
has a fractional $\varphi$-potential $(U,V)$ with parameters $(\beta,\gamma;H,K)$, then
\begin{align*}
      \left \| \sup_{s \in [t, T]}|X_s^{t,\xi,\varphi} - X_s^{t,\xi}| \right \|_{L_p}
\leq & c_\eqref{eqn:statement:Lipschitz_case_path}
      \Bigg [  \left (\int_t^T |\varphi(u)|^2 \od u \right )^{\frac{1}{2}} \| 1+ |\xi| \|_{L_p} + \| \xi^\varphi - \xi \|_{L_p} \\
&  \hspace{1em}           + \bigg\|  \int_t^T \Big|U^{\varphi}(u) - U(u) \Big|^\beta H(X_{\cdot \wedge u}^{t,\xi,\varphi}) \od u\bigg\|_{L_p} \\
&  \hspace{1em}           +  \bigg\|  \bigg(\int_t^T  \Big| V^{\varphi}(u) - V(u)\Big|^{2 \gamma} |K(X_{\cdot \wedge u}^{t,\xi,\varphi})|^2 \od u \bigg)^{\frac{1}{2}}\bigg\|_{L_p}\Bigg ].
\end{align*}
\end{coro}
\medskip

\begin{proof}
We exploit \cref{statement:transference_composition_new} to be able to choose
$((b(u,X))^\varphi)_{u\in [t,T]}$ and
$((\sigma(u,X))^\varphi)_{u\in [t,T]}$ as
\[ (b^\varphi(u,X_u^\varphi))_{u\in [t,T]}
   \sptext{1}{and}{1}
   (\sigma^\varphi(u,X_u^\varphi))_{u\in [t,T]},\]
respectively. Therefore \cref{definition:fractional_potential} and \cref{statement:Lipschitz_case_path} give our statement.
\end{proof}
\bigskip

\begin{exam}\label{cor:statement:Lipschitz_case_bsigma_deterministic}
If $(b,\sigma)$ in \cref{statement:Lipschitz_case_path} are deterministic, then
\[
        \left \| \sup_{s \in [t, T]}|X_s^{t,\xi,\varphi} - X_s^{t,\xi}| \right \|_{L_p}
   \leq c_\eqref{eqn:statement:Lipschitz_case_path} \left [
        \left (\int_t^T |\varphi(u)|^2 \od u \right )^{\frac{1}{2}} \| 1+ |\xi| \|_{L_p} +  \| \xi-\xi^\varphi\|_{L_p} \right ].\]
Consequently, in the notation from \eqref{eqn:def:Phip} for $s\in [t,T]$ one has
\[ |X_s^{t,\xi}|_{\Phi,p}
   \leq c_\eqref{eqn:statement:Lipschitz_case_path} \left [
        \Phi \left ( \varphi \mapsto \left (\int_t^T |\varphi(u)|^2 \od u \right )^{\frac{1}{2}} \right ) \| 1+ |\xi| \|_{L_p} +  | \xi|_{\Phi,p}  \right ].\]
To verify this we choose $b^\varphi=b$, $\sigma^\varphi=\sigma$,
$U = V \equiv 0$, and $H=K\equiv 0$.
For this choice we apply \cref{cor:statement:Lipschitz_case_new}.
\end{exam}
\bigskip

An example for a pair of controls is the following:
\bigskip

\begin{exam}
We assume $d=1$, $t\le a < c \le T$, $\varphi := 1_{(a,c]}$,
and that $(b,\sigma)$ has a fractional $\varphi$-potential with parameters $(\beta,\gamma;1,1)$ of the form
\[ U_s(\omega) := \1_{\{ W_s(\omega) > K\}} 1_{(a,T]}(s)
   \sptext{1}{and}{1}
   V_s(\omega) := \1_{\{ W_s(\omega) > L\}}  1_{(a,T]}(s) \]
for some $K,L\in \R$. 
Then for $q\in [1,\infty)$ it is shown in \cite[pages 44, last line]{Geiss:Ylinen:21}
that
\[     \left \| \int_c^T \left | D_s^{(a,c]} - D_s \right |\od s  \right \|_{L_q}
   \le 8 \beta_q^2 \sqrt{(T-c)(c-a)}
   \sptext{1}{for}{1} D\in \{U,V\}, \]
where $\beta_q \ge 1$ is taken from \eqref{eqn:BDG}.
One has
$\left | D_s^{(a,c]} - D_s \right |=  \left | D_s^{(a,c]} - D_s \right |^\delta$ for all $\delta>0$
as $\left | D_s^{(a,c]} - D_s \right |\in \{0,1\}$. Therefore, for all $p\in (0,\infty)$ we get the estimates
\begin{align*}
\bigg\|  \int_t^T \Big|U^{(a,c]}(u) - U(u) \Big|^\beta \od u\bigg\|_{L_p}
&\le c \sqrt{c-a}, \\
\bigg\|  \bigg(\int_t^T  \Big| V^{(a,c]}(u) - V(u)\Big|^{2\gamma} \od u \bigg)^{\frac{1}{2}}\bigg\|_{L_p}
&\le d \sqrt[4]{c-a},
\end{align*}
where $c=c(p,T)>0$ and $d=d(p,T)>0$.
\end{exam}

First we connect the above considerations to Malliavin differentiability. Here we contribute further to
\cite[Theorem 3.2]{Imkeller:Reis:Salkeld:2019} by using less structural assumptions and a BMO-condition
on the drift. By \cref{cor:statement:Lipschitz_case_bsigma_deterministic} the \cref{statement:Malliavin_differentiability}
below works in the fully path-dependent case as well when $(b,\sigma)$ are deterministic. Compared  
to \cite[Theorem 3.2]{Imkeller:Reis:Salkeld:2019} we do not compute the Malliavin derivative, we are only
interested in its existence and quantitative estimates.
\medskip

\begin{defi}
\label{definition:UV-Besov}
Assume $p\in [2,\infty)$, $(\beta,\gamma)\in (0,1]^2$,
and $\cP_{t,T}$-measurable processes $U, V:[t,T]\times \Omega\to \R$.
Then we let  
\begin{align*} 
     |(U,V)|_{\B_p^{\Phi_2}} 
& := \sup_{0\le a < c \le T} \frac{1}{\sqrt{c-a}} 
     \Bigg [    \left \|  \int_t^T \left |U^{(a,c]}(u) - U(u) \right | \od u\right \|_{L_p} \\
&    \hspace{10em}         +  \left \|  \left(\int_t^T  \left| V^{(a,c]}(u) - V(u)\right |^2 \od u \right )^{\frac{1}{2}}\right \|_{L_p} 
     \Bigg ],\\
|(U,V)|_{\B_{p,q}^{\eta;(\beta,\gamma)}} &:=  \begin{cases} 
   \inf \| \cK_\eta \kappa \|_{L_q([0,1],\mu)} &: (\eta,q)\in (0,1)\times [1,\infty)\\
   \sup_{r\in (0,1]} \frac{P(r)}{r^\eta}       &: (\eta,q)\in (0,1]\times \{\infty\}
   \end{cases}, 
\end{align*}
where 
\[ P(r):= \bigg\|  \int_t^T \Big|U^\br(u) - U(u) \Big|^\beta \od u \bigg\|_{L_p}
                   + \bigg\| \left (\int_t^T \Big|V^\br(u) - V(u) \Big|^{2\gamma} \od u \right )^\frac{1}{2}\bigg\|_{L_p} \]
and the infimum extends over  all measurable $\kappa:[0,1]\to [0,\infty)$ with
$P \le \kappa$.
\medskip

Similarly, for a continuous adapted process $A=(A_s)_{s\in [t,T]}$ we let
\begin{align*}
     |A|_{\B_p^{\Phi_2,*}} 
& := \sup \left \{ 
     \frac{ \left \|  \sup_{u\in [t,T]}\left |A^{(a,c]}_u - A_u \right | \right \|_{L_p}}{\sqrt{c-a}}:
     0\le a < c \le T \right \}, \\
     |A|_{\B_{p,q}^{\eta,*}} 
& := 
   \begin{cases}
   \inf \| \cK_\eta \kappa \|_{L_q([0,1],\mu)}: &: (\eta,q) \in (0,1)\times [1,\infty)\\
               \sup_{r\in (0,1]} \frac{Q(r)}{r^\eta}      &: (\eta,q) \in (0,1]\times \{\infty\}
   \end{cases},
\end{align*}
where $Q(r):= \left \|  \sup_{u\in [t,T]} |A_u^\br - A_u | \right \|_{L_p}$
and the infimum extends over  all measurable $\kappa:[0,1]\to [0,\infty)$ with
$Q\le \kappa$.
\end{defi}
\medskip

We use the measurable majorants $\kappa$ to avoid to discuss
the measurability of the functions $r\mapsto P(r)$ and $r\mapsto Q(r)$. In this article this 
issue is without relevance, but certainly will be clarified within further investigations along \cite[Lemma 4.7]{Geiss:Ylinen:21}. 

\begin{rema}
In particular we will use the condition
$|(U,V)|_{\B_{2,\infty}^{1;(1,1)}} < \infty$, i.e.
\[ \sup_{r\in (0,1]} \frac{1}{r} \left [
         \bigg\| \int_t^T \Big|U^\br(u) - U(u) \Big| \od u \bigg\|_{L_2}
       + \bigg\| \left (\int_t^T \Big|V^\br(u) - V(u) \Big|^2 \od u \right )^\frac{1}{2}\bigg\|_{L_2} \right ] < \infty. \]
Assume that $V\in L_2([t,T]\times \Omega)$, which we interpret as
$V\in L_2(\Omega,\cF,\P;\ell_2)$. Then by \cref{statement:characterization_D12_H_valued}
the condition for $V$ is equivalent to $V \in \D_{1,2}(\ell_2)$.
As one has
\[ \left \|  \int_t^T |U^\br(u) - U(u) | \od u \right \|_{L_2}
   \le \sqrt{T}  \bigg\| \left (\int_t^T \Big|U^\br(u) - U(u) \Big|^2 \od u \right )^\frac{1}{2}\bigg\|_{L_2}, \]
the condition on $V$ in $|(U,V)|_{\B_{2,\infty}^{1;(1,1)}} < \infty$  is stronger than on $U$.
\end{rema}
\medskip

\begin{coro}[Malliavin differentiability I]
\label{statement:Malliavin_differentiability}
Assume the conditions of \cref{statement:Lipschitz_case_path} with $d=1$, $p\in [2,\infty)$,
that $(U, V)$ is a Lipschitz potential for $(b, \sigma)$ for all $1_{(a,c]}$ with $0\le a < c \le T$.
If $\xi\in \B_p^{\Phi_2}$ is $\cF_t$-measurable, then
$X_s^{t,\xi}\in \D_{1,2}$ for all $s\in [t,T]$ with 
\begin{align*}
      \sup_{0\le a < c \le T \atop s\in [t,T]}\left \| \left ( \frac{1}{c-a} \int_a^c |D_u X_s^{t,\xi}|^2 \od u \right )^\frac{1}{2} \right \|_{L_p}
&\le c_{\eqref{eqn:characterization_Phi2p},p} |X|_{\B_p^{\Phi_2,*}} \\
&\le c_{\eqref{eqn:characterization_Phi2p},p} 2 c_\eqref{eqn:statement:Lipschitz_case_path} \left [ 1+ \| \xi\|_{\B^{\Phi_2}_p} + |(U,V)|_{\B_p^{\Phi_2}} \right ].
\end{align*}
\end{coro}
\smallskip

\begin{proof}
For $0\le a < c \le T$ and $s\in [t,T]$ we get from \cref{cor:statement:Lipschitz_case_new} that
\begin{align*}
      \left \| \sup_{s\in [t,T]} |X_s^{t,\xi,\1_{(a,c]}} - X_s^{t,\xi}| \right \|_{L_p}
\leq & c_\eqref{eqn:statement:Lipschitz_case_path}
      \Bigg [  \sqrt{c-a} \| 1+ |\xi| \|_{L_p} +  \| \xi-\xi^{(a,c]}\|_{L_p} \\
&  \hspace{5em}           + \bigg\|  \int_t^T \Big|U^{(a,c]}(u) - U(u) \Big| \od u\bigg\|_{L_p} \\
&  \hspace{5em}           +  \bigg\|  \bigg(\int_t^T  \Big| V^{(a,c]}(u) - V(u)\Big|^2 \od u \bigg)^{\frac{1}{2}}\bigg\|_{L_p}\Bigg ] \\
\leq & c_\eqref{eqn:statement:Lipschitz_case_path}
     \sqrt{c-a} \left [ \| 1+ |\xi| \|_{L_p} +  | \xi|_{\Phi_2,p} +|(U,V)|_{\B_p^{\Phi_2}} \right ] \\
\leq & c_\eqref{eqn:statement:Lipschitz_case_path}
     \sqrt{c-a} \left [ 1+ 2^{1-\frac{1}{p}} \| \xi\|_{\B^{\Phi_2}_p} + |(U,V)|_{\B_p^{\Phi_2}} \right ].
\end{align*}
This implies the second inequality. The first one is \eqref{eqn:characterization_Phi2p}.
\end{proof}
\smallskip

\begin{coro}[Malliavin differentiability II]
\label{statement:Malliavin_differentiability_II}
Assume the conditions of \cref{statement:Lipschitz_case_path} with $d=1$, $p=2$, and
that $(U, V)$ is a Lipschitz potential for $(b, \sigma)$ for all $\varphi_r\equiv r$ with $r \in (0,1]$.
If $\xi\in \D_{1,2}$ is $\cF_t$-measurable, then
$X_s^{t,\xi}\in \D_{1,2}$ for all $s\in [t,T]$ and
\begin{multline*}
      \sup_{s\in [t,T]}   \left\| D X_s^{t,\xi} \right \|_{L_2(\Omega\times [0,T];\R^N)} 
  \le 2 |X^{t,\xi}|_{\B_{1,\infty}^{1,*}} \\
  \le 2  c_\eqref{eqn:statement:Lipschitz_case_path} \left [ \sqrt{T} \| 1+|\xi|\|_{L_2} +
      \left\| D \xi \right \|_{L_2(\Omega\times [0,T];\R^N)} + |(U,V)|_{\B_{2,\infty}^{1;(1,1)}} \right ].
\end{multline*}
\end{coro}
\smallskip

\begin{proof}
Similarly as in the proof of \cref{statement:Malliavin_differentiability} we get,
for $r\in (0,1]$, from \cref{cor:statement:Lipschitz_case_new} and
\cref{statement:characterization_D12} that
\begin{align*}
     &  \left \| \sup_{s\in [t,T]} |X_s^{t,\xi,\br} - X_s^{t,\xi}| \right \|_{L_2} \\
\leq & c_\eqref{eqn:statement:Lipschitz_case_path}
      \left [  r \sqrt{T} \| 1+ |\xi| \|_{L_2} +  \| \xi-\xi^r \|_{L_2} + |(U,V)|_{\B_{2,\infty}^{1;(1,1)}}  r \right ] \\
\leq & c_\eqref{eqn:statement:Lipschitz_case_path}
       r \left [  \sqrt{T} \| 1+ |\xi| \|_{L_2} +  \left\| D \xi \right \|_{L_2(\Omega\times [0,T];\R^N)} + |(U,V)|_{\B_{2,\infty}^{1;(1,1)}}  \right ].
\end{align*}
Applying \cref{statement:characterization_D12} once more, this time to the LHS, the assertion
follows.
\end{proof}
\medskip

Our main result about real interpolation is the following statement:
\smallskip

\begin{coro}[Real interpolation]
\label{statement:real_interpolation_SDE_Lipschitz}
Assume the conditions of \cref{statement:Lipschitz_case_path} for $d=1$, and
$p \in [2,\infty)$, and
that $(U,V)$ is a fractional potential for $(b, \sigma)$
with parameters $(\beta,\gamma;1,1)$ for fixed
$(\beta,\gamma)\in (0,1]^2$, but for all $\varphi_r\equiv r \in [0,1]$.
Assume further that
\begin{enumerate}[{\rm (1)}]
\item $(\eta,q)\in (0,1)\times [1,\infty]$,
\item $\xi \in \B_{p,q}^\eta$ is $\cF_t$-measurable.
\end{enumerate}
Then one has $X_s^{t,\xi} \in \B_{p,q}^\eta$ for all $s\in [t,T]$ and
\begin{equation}\label{eqn:statement:real_interpolation_SDE_Lipschitz}
    \left | X^{t,\xi} \right |_{\B_{p,q}^{\eta,*}}
\le c_\eqref{eqn:statement:real_interpolation_SDE_Lipschitz}
    \left [  1+   \| \xi\|_{\B_{p,q}^\eta} + |(U,V)|_{\B_{p,q}^{\eta;(\beta,\gamma)}} \right ],
\end{equation}
where
$c_\eqref{eqn:statement:real_interpolation_SDE_Lipschitz}=
c_\eqref{eqn:statement:real_interpolation_SDE_Lipschitz}(p,T,\| L_b\|_{\BMO(S_2)},L_\sigma,K_b,K_\sigma,\eta,q)>0$.
\end{coro}

\begin{proof}
For $r\in [0,1]$ \cref{cor:statement:Lipschitz_case_new} gives that
\begin{align*}
      \Big\| \sup_{s\in [t,T]}|X_s^{t,\xi,\varphi_r} - X_s^{t,\xi}| \Big\|_{L_p}
\leq & c_\eqref{eqn:statement:Lipschitz_case_path}
      \Bigg [ r \sqrt{T} \| 1+ |\xi| \|_{L_p}  +  \|\xi^\br-\xi\|_{L_p} \\
&   \hspace{3em}          + \bigg\|  \int_t^T \Big|U^\br(u) - U(u) \Big|^\beta \od u\bigg\|_{L_p} \\
&   \hspace{3em}          +  \bigg\|  \bigg(\int_t^T  \Big| V^\br(u) - V(u)\Big|^{2\gamma} \od u \bigg)^{\frac{1}{2}}\bigg\|_{L_p}\Bigg ].
\end{align*}
Let $q<\infty$.
By relation \eqref{eqn:real_interpolation_via_decoupling} 
and an admissible $\kappa$ for \cref{definition:UV-Besov}
this implies 
\begin{align*}
      \left | X^{t,\xi} \right |_{\B_{p,q}^{\eta,*}}
& \le  c_\eqref{eqn:statement:Lipschitz_case_path} 
      \Big [   \| r \mapsto \cK_\eta(r) r  \|_{L_q([0,1],\mu)} 
                \sqrt{T} \| 1+ |\xi| \|_{L_p} \\ 
& \hspace*{12em}
              + c_\eqref{eqn:real_interpolation_via_decoupling}  \| \xi\|_{\B_{p,q}^\eta}
              + \| \cK_\eta \kappa \|_{L_q([0,1],\mu)} \Big ],
\end{align*}
where we use $\mu$ and $\cK_\theta$ as defined in 
\eqref{eqn:equivalence_mu} and \eqref{eqn:equivalence_K}.
To continue to estimate the RHS we observe that
\begin{enumerate}[(a)]
\item $\| 1+ |\xi| \|_{L_p}\le 1 +  \| \xi\|_{\B_{p,q}^\eta}$ by the definition of the Besov space,
\item $\| r \mapsto \cK_\eta(r) r  \|_{L_q([0,1],\mu)} < \infty$ by \eqref{eqn:equivalence_K}
\end{enumerate}
to conclude the proof. For $q=\infty$ the proof is the same,
even more direct as the function $\kappa$ is not needed.
\end{proof}


\section{A counterexample regarding H\"older continuous coefficients}
\label{sec:counterexample}

In \cite[Corollary 4.2]{Alos:Ewald:2008} it is shown that the square-root diffusion for the Heston stochastic
volatility model belongs to $\D_{1,2}$ although the diffusion coefficient for this diffusion is of form $c \sqrt{X_s}$.
In this section we construct an example where a diffusion with a H\"older continuous  diffusion coefficient does not belong
to $\D_{1,2}$.
\bigskip

For this section we fix $\theta \in (0,1)$. In \cref{statement:hoelder_and_not_D12} below we show that under the assumption that $\sigma$ is uniformly
$\theta$-H\"older continuous and bounded, and even when the drift $b$ vanishes, one cannot expect
that $X_T\in \D_{1,2}$. To construct our example we use
the methodology behind the  characterization of H\"older  functions due to
Ciesielski \cite{Ciesielski:60b}. To start with, for any $n\in\bN$ and $y\in\R^+$ we define
\[ r_n(y) := \sum^{\fz}_{i = 1}
             \left[\chf_{(\frac{2i - 2}{2^{n + 1}},\frac{2i - 1}{2^{n + 1}}]}(y) - \chf_{(\frac{2i - 1}{2^{n + 1}},\frac{2i}{2^{n + 1}}]}(y)\right].\]
For any $x \in \R$, let
        \begin{equation}\label{ce4}
                S_n^\theta(x) := 2^{-(n+1)(\theta - 1)} \int^{|x|}_{0} r_n(y)\,\od y.
        \end{equation}
We have
\begin{enumerate}
\item $\| S_n^\theta \|_{C(\R)} = 2^{-(n+1)\theta}\to 0$ as $n\to \infty$,
\item $| S_n^\theta |_\theta =1$,
\item $| S_n^\theta |_1 = 2^{(n+1)(1-\theta)}$.
\end{enumerate}
We fix $T=1$, choose a net $0=t_0<t_1<t_2 \uparrow 1$, define
$I_\ell:= (t_{\ell-1},t_\ell]$, and let
\[ \sigma^\theta(s,x) := 1 +  \sum_{\ell=1}^\infty \chf_{I_\ell}(s) S_{n_\ell}^\theta(x)
   \sptext{1}{for}{1} (s,x)\in [0,1)\times \R \]
and a fixed sequence $0 \le n_1 < n_2 < \cdots$.
With this construction we get that
\begin{enumerate}
\item $1 \le \sigma^\theta < 2$,
\item $\lim_{s\uparrow 1} \left [ \sup_{x\in \R} |\sigma^\theta(s,x) - 1| \right ]=0$ because $\lim_{\ell \to \infty} n_\ell = \infty$,
\item $\sigma^\theta:[0,1]\times \R\to \R$ is $(\cB([0,1])\otimes \cB(\R),\cB(\R))$-measurable if $\sigma^\theta(1,x):=1$,
\item $|\sigma^\theta(s,\cdot)|_\theta\le 1$ for all $s\in [0,1]$,
\item $|\sigma^\theta(s,\cdot)|_1 \le 2^{(n_\ell+1)(1-\theta)}$ for all $s\in [0,t_\ell]$ and $\ell\in \bN$.
\end{enumerate}
\bigskip

\begin{theo}
\label{statement:hoelder_and_not_D12}
There are sets of parameters $(t_\ell)_{\ell \in \bN}$ and $(n_\ell)_{\ell \in \bN}$
such that for all $x\in \R$ and  the unique strong solution $(X_s)_{s\in [0,1]}$ to the SDE
\[ X_s = x + \int_0^s \sigma^\theta(u,X_u) d W_u
   \sptext{1}{for}{1} s\in [0,1]\]
one has $X_1\not \in \D_{1,2}$.
\end{theo}
\smallskip

\begin{proof}
We do not yet fix the parameters $(t_\ell)_{\ell \in \bN}$ and $(n_\ell)_{\ell \in \bN}$.
\smallskip

(a) On each interval $[0,t_\ell]$, $\ell \in \bN$, the SDE has a unique strong solution. As
$|\sigma(s,x)| \le 2$ we can uniquely extend this solution to a continuous solution on $[0,1]$ with a possible
change of the original solution on a null-set.
\medskip

(b) From \cite[Theorem 2.3.1]{Nualart:06} we infer that
for $t\in (0,1]$ the random variable $X_t$ has a law that is absolutely
continuous with respect to the Lebesgue measure.
\medskip

(c) From now on we assume that $X_1\in \D_{1,2}$. For $s\in [0,1)$ the relations
$X_s = \E [X_1|\cF_s]$ a.s. and \eqref{eqn:expected_value_D12} imply
$\| D X_s \|_{L_2(\Omega\times [0,1])} \le \| D X_1 \|_{L_2(\Omega\times [0,1])}$ with
\[ D \sigma^\theta(s,X_s) = (DX_s) G(s,X_s) \]
where we take as generalized derivative of $\partial \sigma^\theta/\partial x$ the expression
\[ G(s,x) := 2^{(n+1)(1-\theta)} r_n(|x|) \sign_0(x) \]
if $s\in I_\ell$ (see \cite[Proposition 1.2.4 and the remark following it]{Nualart:06}).
\medskip

(d) For $0\le a < b <1$ we have
\begin{align}
     \| D (X_b - X_a) \|_{L_2(\Omega\times [0,1])}
&\le \| D X_b \|_{L_2(\Omega\times [0,1])} + \| D X_a \|_{L_2(\Omega\times [0,1])} \notag \\
&\le 2 \| D X_1 \|_{L_2(\Omega\times [0,1])}. \label{eqn:DXb-DXa}
\end{align}
Let $t_{\ell-1} < a < b < t_\ell$ for some $\ell\in \bN$. Then by
\cite[Lemma 1.3.4]{Nualart:06} we have
\begin{align}
     \int_a^b \E|D_r(X_b-X_a)|^2 \od r
& =  \int_a^b \E |\sigma^\theta(s,X_s)|^2 \od s + \int_a^b \int_a^s \E |D_r\sigma^\theta (s,X_s)|^2 \od r \od s \notag \\
&\ge (b-a) + 2^{2(n_\ell+1)(1-\theta)} \int_a^b \int_a^s \E |D_r (X_s-X_a)|^2 \od r \od s. \label{eqn:lower_bound_D12_Gronwall}
\end{align}
With
\[ h(s) := \int_a^s \E |D_r (X_s-X_a)|^2 \od r
   \sptext{1}{for}{1}
   s\in [a,b]\]
we get from \eqref{eqn:lower_bound_D12_Gronwall} that
\begin{align*}
h(b) & \ge (b-a) + 2^{2(n_\ell+1)(1-\theta)} \int_a^b h(s) \od s \\
     & \ge (b-a) + 2^{2(n_\ell+1)(1-\theta)} \int_a^b (s-a) \od s \\
     &  = (b-a) + 2^{2(n_\ell+1)(1-\theta)} \frac{(b-a)^2}{2}.
\end{align*}
Choosing $(b-a) = |I_\ell|/2$, we get
\[ \int_a^b \E |D_r (X_b-X_a)|^2 \od r \ge  2^{2(n_\ell+1)(1-\theta)} \frac{|I_\ell|^2}{8}.\]
For any chosen sequence $(t_\ell)_{\ell\in \bN}$ we find $(n_\ell)_{\ell\in \bN}$ such that
\[ \lim_{\ell \to \infty} 2^{2(n_\ell+1)(1-\theta)} \frac{|I_\ell|^2}{8}=\infty.\]
But this leads to a contradiction to \eqref{eqn:DXb-DXa}.
\end{proof}


\section{Coupling of SDEs - H\"older dispersion in dimension one}
\label{sec:Hoelder}

In this section we turn to the coupling of SDEs where the diffusion coefficient is only H\"older continuous.
By \cref{statement:hoelder_and_not_D12} we cannot expect to obtain the same results as in the setting of 
\cref{sec:Lipschitz}. Even more, the situation changes in a way that we have to restrict ourselves 
to coupling functions $\varphi=\1_{(a,c]}$. The main result of this section is
\cref{statement:holder:main}. This result will be applied to \cref{statement:BSDE_variation}
and finally to \cref{statement:variation_bsde_CIR} where we consider the CIR-process as forward
process in a decoupled forward backward stochastic differential equation (FBSDE).
\medskip 

Let us start with the conditions on the coefficients of the SDE for this
section:

\begin{assumption}
\label{ass:hoelder}
We assume $\theta\in \left [\frac{1}{2},1 \right )$ and
\begin{align*}
     b: & [t,T] \times \Oz \times \R  \to \R, \\
\sigma: & [t,T]            \times \R  \to \R,
\end{align*}
to be
$(\cP_{t,T} \otimes \cB(\R), \cB(\R))$- and
$(\cB([t,T]) \otimes \cB(\R), \cB(\R))$-measurable, such that
$b(t,\cdot,\cdot)\equiv 0$, $\sigma(t,\cdot)\equiv 0$, and
\begin{align}
 |b(s,\omega,x) - b(s,\omega,y)|     & \leq L_b |x-y|,\label{ass:1:ass:hoelder}\\
        |\sigma(s,x) - \sigma(s,y)| & \leq L_{\sigma} |x-y|^\theta, \label{ass:2:ass:hoelder} \\
          |b(s,\omega,x)| & \leq K_b (1 + |x|)\label{ass:3:ass:hoelder}, \\
                    |\sigma(s,x)| & \leq K_{\sz} (1 + |x|),\label{ass:4:ass:hoelder}
\end{align}
on $[t,T]\times \Omega\times \R$,  where $L_b,L_{\sigma},K_b,K_\sigma\ge 0$ are constants.
\end{assumption}
\medskip

Again for $p\in (0,\infty)$  we assume and fix a solution to
\begin{equation}\label{eqn:SDE_hoelder}
        X_s^{t,\xi}
        = \xi + \int_t^s b(u, X_u^{t,\xi})\,\od u
        + \int_t^s \sigma(u, X_u^{t,\xi})\,\od W_u
\end{equation}
such that
\begin{enumerate}
\item $\xi \in L_{p \vee 2}(\Omega,\cF_t,\P)$,
\item $(X^{t,\xi}_s)_{s\in [t,T]}$ is $(\cF_s)_{s\in [t,T]}$-adapted and path-wise continuous,
\item $\E \int_t^T |\sigma(u,X_u^{t,\xi})|^2  \od u + \E \int_t^T |b(u,X_u^{t,\xi})|^2  \od u <\infty$.
\end{enumerate}
\medskip
We use the coupling function $\varphi_{a,c} := \chf_{(a,c]}$ for $t\le a < c \le T$ and the notation
\[  X_s^{t,\xi,(a,c]} := X_s^{t,\xi,\1_{a,c}} \]
where the RHS was introduced in \cref {sec:Lipschitz}.

\begin{rema}
If $(b,\sigma)$ are non-random, bounded, and continuous, and if the law of $\xi$ has a compact support,
then \cite[Theorem IV.2.2]{Ikeda:Watanabe:2nd} yields to a weak solution
to \eqref{eqn:SDE_hoelder} (for weaker conditions see
\cite[Remark IV.2.1]{Ikeda:Watanabe:2nd}). By  \cite[Theorem IX.3.5(ii)]{Revuz:Yor:1999} we have path-wise uniqueness.
Combining this with \cite[Corollary 5.3.23]{Karatzas:Shreve:98} we get a strong solution to \eqref{eqn:SDE_hoelder}.
\end{rema}

The main theorem of this section is:
\medskip

\begin{theo}
\label{statement:holder:main_general}
Let $c\in [t,T]$ and  $\varphi:[0,T]\to [0,1]$ be a coupling function such that
\[ \varphi(s)=0 \sptext{1}{for}{1} s\in (c,T] \]
and where we assume that $(W^\varphi_s)_{s\in [0,T]}$ is a version such that $W_s^\varphi-W_c^\varphi=W_s-W_c$ for $s\in [c,T]$.
Suppose that $(X^{t,\xi}_s)_{s\in [t,T]}$ is a solution to SDE \eqref{eqn:SDE_hoelder}
with the coefficients satisfying \cref{ass:hoelder} and that $(X^{t,\xi,\varphi}_s)_{s\in [t,T]}$ is its $\varphi$-coupled
version. Let $p,q\in (0,\infty)$,
\begin{align*}
\xi & \in L_{p\vee 2}(\Omega,\cF_t,\P),\\
 \Lambda & := \int_c^T \left | b\left (s,X^{t,\xi,\varphi}_s\right ) - b^\varphi\left (s,X^{t,\xi,\varphi}_s\right )\right|\od s,\\
\Delta & := \sup_{s\in [t,c]} \left |X_s^{t,\xi} - X_s^{t,\xi,\varphi}\right |,  \\
p_\theta & := \min \left \{ 1,\frac{p}{3-2\theta} \right \}.
\end{align*}
Then the following holds:
\medskip
\begin{enumerate}[{\rm\bf (1)}]
\item \label{item:(0,1):statement:holder:sup:r:de_general}
      If \underline{$p\in (0,1)$}, then one has
      \[    \left\|\sup_{s\in [t,T]}\left|X^{t,\xi}_s - X^{t,\xi,\varphi}_s\right|\right\|_{L_p}
          \leq  \frac{e^{L_bT}}{p \sqrt[p]{1 - p}} \left \| \Delta + \Lambda  \right \|_{L_p}. \]
      \bigskip

\item \label{item:1:statement:holder:de_general}
      $\left\|X^{t,\xi}_T - X^{t,\xi,\varphi}_T\right\|_{L_1} \leq e^{L_bT} \| \Delta + \Lambda \|_{L_1}$.
      \bigskip

 \item \label{item:[1,2]:statement:holder:sup:r:decoupling_general}
       For \underline{$p\in [1,\infty)$} there is a
       $c_\eqref{eqn:item:[1,2]:statement:holder:main:ran_general}=c_\eqref{eqn:item:[1,2]:statement:holder:main:ran_general}(T,L_b,L_\sigma,\theta,p)>0$
       such that one has
       \begin{equation}\label{eqn:item:[1,2]:statement:holder:main:ran_general}
            \left\|\sup_{s\in [t,T]}\left|X^{t,\xi}_s - X^{t,\xi,\varphi}_s\right|\right\|_{L_p}
       \le  c_\eqref{eqn:item:[1,2]:statement:holder:main:ran_general} \left [  \left \|\Delta + \Lambda  \right \|_{L_p}
            + (\E |\Delta + \Lambda |^{p_\theta})^\frac{1}{p} \right ].
       \end{equation}
       \bigskip

\item \label{item:[1,2]:statement:holder:main:deter_general}
       If \underline{$p\in [1,3 - 2\tz)$} and $b: [0,T] \times \R  \to \R$, then
       one has
       \begin{equation}\label{eqn:item:[1,2]:statement:holder:main:deter_general}
              \left\|\sup_{s\in [t,T]}\left|X^{t,\xi}_s - X^{t,\xi,\varphi}_s\right|\right \|_{L_{p,\infty}}
         \le c_\eqref{eqn:item:[1,2]:statement:holder:main:deter_general} \left [ \| \Delta \|_{L_p} +  (\E \Delta)^{\frac{1}{p}} \right ],
        \end{equation}
       where
       $c_\eqref{eqn:item:[1,2]:statement:holder:main:deter_general}=c_\eqref{eqn:item:[1,2]:statement:holder:main:deter_general}(T,L_b,L_\sigma,\theta,p)>0$.

\end{enumerate}
\end{theo}
\bigskip

\subsection{Proof of \cref{statement:holder:main_general}}\ \medskip

We start with some lemmas. Inspecting the proof of \cite[Lemma 3.2]{Gyongy:Rasonyi:11} the following estimates were proven:

\begin{lemm}
\label{statement:GR}
For $c\in [t,T)$, a measurable function $D:[c,T]\to \R$, and $s\in [c,T]$ one has:
\begin{enumerate}[{\rm(1)}]
\item If $1\le \rho \le q \le p<\infty$, then
      \[ \left ( \int_c^s |D_u|^\rho \od u\right )^\frac{p}{q}
         \le (s-c)^{\frac{p}{q}-1}\int_c^s [|D_u|+ |D_u|^p]\od u.\]
\item If $1<p<q<\infty$ and $q-p+1 \le \rho \le q$, then
      \[ \left ( \int_c^s |D_u|^\rho \od u\right )^\frac{p}{q}
     \le \left ( \frac{1}{K} \right )^\alpha |D_s^*|^p + K^\beta \int_c^s [[D_u|+ |D_u|^p]\od u\]
     for $K>0$, $\alpha:= \frac{q}{q-p}$, $1=\frac{1}{\alpha}+\frac{1}{\beta}$, and with
     $D_s^* := \sup_{u\in [c ,s]} |D_u|$.
\end{enumerate}
\end{lemm}
\bigskip

The next statement is a relative of \cite[page 390, 4th line from below]{Revuz:Yor:1999}:

\begin{lemm}
\label{statement:basic_equality_modulus_D}
For $c\in [t,T)$ assume $\cP_{c,T}$-measurable processes $(\sigma_s)_{s\in [c,T]}$ and
$(b_s)_{s\in [c,T]}$ such that
\begin{enumerate}[{\rm (1)}]
\item $\E \int_c^T [ |b_s|+|\sigma_s|^2] \od s<\infty$,
\item $A\in L_2(\Omega,\cF_c,\P)$,
\item a path-wise continuous adapted process $(D_s)_{t\in [c,T]}$ such that
      \[ D_s = A +\int_c^s \sigma_u \od W_u + \int_c^s b_u \od u
         \sptext{1}{for}{1} s\in [c,T] \mbox{ a.s.},\]
\item $|\sigma| \le L_\sigma |D|^\theta$  on $[c,T]\times \Omega$ for some
      $L_\sigma\ge 0$ {\rm (}recall that $\theta\in [1/2,1)${\rm)}.
\end{enumerate}
Then one has
\[ |D_s| = |A| + \int_c^s \sign_0(D_u) \od D_u  \sptext{1}{for}{1} s\in [c,T] \mbox{ a.s.}\]
\end{lemm}

\begin{proof}
For $n\ge 2$ we find continuous functions $v_n:\R\to [0,\infty)$ such that
$\{x\in \R: v_n(x)\not = 0\} \subseteq [\frac{1}{n^2},\frac{1}{n}]$, $|v_n(z)|\le \frac{2}{|z| \ln n}$ for
$z\in [\frac{1}{n^2},\frac{1}{n}]$, and such that $\int_\R v(z) \od z=1$
(note that $\int_{1/n^2}^{1/n} \frac{\od z}{z \ln n} =1$), which is the Yamada-Watanabe method as presented
in \cite[proof of Proposition 2.2]{Gyongy:Rasonyi:11}.
We let $\Phi_n(x) := \int_0^{|x|} \int_0^y v_n(z) \od z \od y$ so that
$\Phi_n \in C_b^2(\R)$, $\lim_{n\to\infty}\Phi_n(x)=|x|$, $\lim_{n\to \infty} \Phi'_n=\sign_0(x)$, and
$|\Phi_n'(x)| \le 1$. By It\^o's formula,
\[ \Phi_n(D_s) =\Phi_n(D_c) +\int_c^s \Phi_n'(D_u) \od D_u + \frac{1}{2} \int_c^s \Phi_n''(D_u) \sigma_u^2 \od u \mbox{ a.s.}\]
The first three terms converge in $L_2$ to $|D_s]$, $|D_c|$, and $\int_c^s \sign_0(D_u) \od D_u$. For the last term we observe
$\Phi_n''(z)=v_n(|z|)$ and
\begin{multline*}
      \int _c^s \Phi_n''(D_u) \sigma_u^2 \od u
     =  \int _c^s v_n(|D_u|) \sigma_u^2 \od u\\
    \le \frac{2}{\ln n} \int_c^s \chf_{\left \{ |D_u|\in \left [\frac{1}{n^2},\frac{1}{n}\right ] \right \}} |D_u|^{-1} L^2_\sigma |D_u|^{2 \theta} \od u
    \le \frac{2 L^2_\sigma}{\ln n} (s-c) \to 0 \sptext{1}{as}{1} n\to\infty.
\end{multline*}
This proves our statement.
\end{proof}

\begin{lemm}\label{statement:basic_inequality_modulus_D_as}
Assume the conditions and notation of \cref{statement:basic_equality_modulus_D}, and further that
\[ |b| \le L_b |D| + \beta \sptext{1}{on}{1} [c,T]\times \Omega,\]
where $\beta=(\beta_s)_{s\in [c,T|}$ is $\cP_{c,T}$-measurable and non-negative. For $s\in [c,T]$ let
\[ \Gamma_s := |A| + \int_c^s |\sign_0(D_u)| \beta_u \od u
   \sptext{1}{and}{1}
   M_s := \int_c^s \sign_0(D_u) \sigma_u  \od W_u,\]
where we assume a path-wise continuous version of $(M_s)_{s\in [c,T]}$.
Then, for any stopping time $\tau:\Omega \to [c,T]$, one has
\[ D^*_\tau \le e^{L_b(T-c)} [ \Gamma_\tau + M^*_\tau] \mbox{ a.s.} \]
with $D_s^* := \sup_{u\in [c,s]} |D_u|$ and $M_s^* := \sup_{u\in [c ,s]} |M_u|$.
\end{lemm}

\begin{proof}
From  \cref{statement:basic_equality_modulus_D} we know that there is an $\Omega_0\in \cF$
with $\P(\Omega_0)=1$ such that we have for all $c \le v \le s \le T$ and $\omega\in \Omega_0$ that
\begin{align*}
     |D_{\tau\wedge v}(\omega)|
&\le \Gamma_{\tau\wedge v}(\omega) + L_b \int_c^{\tau(\omega)\wedge v} |D_u(\omega)| \od u + M^*_{\tau\wedge v}(\omega) \\
&\le \Gamma_{\tau\wedge v}(\omega) + L_b \int_c^v |D_{\tau(\omega)\wedge u}(\omega)| \od u + M^*_{\tau\wedge v}(\omega) \\
&\le \Gamma_{\tau\wedge s}(\omega) + L_b \int_c^v |D_{\tau(\omega)\wedge u}(\omega)| \od u + M^*_{\tau\wedge s}(\omega).
\end{align*}
As the function
$[c,T]\ni u \mapsto |D_{\tau(\omega)\wedge u}(\omega)|$ is continuous we can apply Gronwall's lemma to get
\[ |D_{\tau\wedge s}(\omega)| \le e^{L(s-c)} \left [  \Gamma_{\tau\wedge s}(\omega) + M^*_{\tau\wedge s}(\omega) \right ] \]
regardless whether $\Gamma_{\tau\wedge s}(\omega)$ is finite or not.
As the RHS is non-decreasing in $s$, we also have
\[ D^*_{\tau\wedge s}(\omega) \le e^{L(s-c)} \left [  \Gamma_{\tau\wedge s}(\omega) + M^*_{\tau\wedge s}(\omega) \right ].\]
Setting $s=T$ gives our statement.
\end{proof}
\bigskip

\begin{lemm}
\label{statement:pre_lemma_hoelder}
Let $p\in (0,\infty)$ and $p_\theta := \min \left \{ 1,\frac{p}{3-2\theta} \right \}$,
assume the  assumptions and notation from \cref{statement:basic_equality_modulus_D} and \cref{statement:basic_inequality_modulus_D_as},
$A\in L_{p\vee 2}(\Omega,\cF_c,\P)$, and
\[ \left ( \int_c^T |\sigma_s|^2 \od s \right )^\frac{1}{2} + \int_c^T |b_s|\od s  +\Gamma_T\in L_{p\vee 2}(\Omega,\cF,\P).\]
Then, for any stopping time $\tau:\Omega\to [c,T]$, one has the following:
\begin{enumerate}[{\rm (1)}]
\item \label{item:(0,1)_sup:statement:pre_lemma_hoelder}
    If $p\in (0,1)$, then
    $\E [ |D_\tau^*|^p |\cF_c ] \le e^{L_b(T-c)p} \frac{p^{-p}}{1-p} \E [|\Gamma_\tau|^p|\cF_c]$  a.s.
\item \label{item:1_without_sup:statement:pre_lemma_hoelder}
    If $p=1$, then
    $\E [ |D_\tau| |\cF_c ] \le e^{L_b(T-c)}  \E [\Gamma_\tau|\cF_c]$ a.s.

\item \label{item:[1,infty)_sup:statement:pre_lemma_hoelder}
    If $p\in [1,\infty)$, then
    \begin{equation}\label{eqn:item:[1,infty)_sup:statement:pre_lemma_hoelder}
    \E [ |D_\tau^*|^p |\cF_c ] \le c_{\eqref{eqn:item:[1,infty)_sup:statement:pre_lemma_hoelder},p}^p  \E [\Gamma_\tau^p + \Gamma_\tau^{p_\theta}|\cF_c] \mbox{ a.s.}
    \end{equation}
    where $c_{\eqref{eqn:item:[1,infty)_sup:statement:pre_lemma_hoelder},p}=c_{\eqref{eqn:item:[1,infty)_sup:statement:pre_lemma_hoelder},p}(T,L_b,L_\sigma,\theta,p)>0$.
\end{enumerate}
\end{lemm}

\begin{proof}
Throughout the proof we will exploit that $[c,T]\ni s\mapsto \Gamma_s\in [0,\infty)$ is non-decreasing.
\smallskip

\eqref{item:1_without_sup:statement:pre_lemma_hoelder}
By \cref{statement:basic_equality_modulus_D} we deduce for $s\in [c,T]$ and $B\in \cF_c$ of positive measure that
\begin{align*}
    \E_{\P_B} |D_{\tau\wedge s}|
& = \E_{\P_B} |A| + \E_{\P_B} \int_c^{\tau\wedge s} \sign_0(D_u) b_u \od u \\
&\le \E_{\P_B} |A| + L_b \int_c^s \E_{\P_B} |D_{\tau\wedge u}| \od u + \E_{\P_B} \int_c^\tau |\sign_0(D_u)| \beta_u \od u \\
& = \E_{\P_B} \Gamma_\tau  + L_b \int_c^s \E_{\P_B} |D_{\tau\wedge u}| \od u
\end{align*}
for $s\in [c,T]$. Gronwall's lemma implies
\[   \E_{\P_B} |D_{\tau\wedge s}| \le e^{L_b(T-c)} \E_{\P_B} \Gamma_\tau \]
which is assertion \eqref{item:1_without_sup:statement:pre_lemma_hoelder}.
\smallskip

\eqref{item:(0,1)_sup:statement:pre_lemma_hoelder}
follows from \eqref{item:1_without_sup:statement:pre_lemma_hoelder}
Lenglart's inequality (\cref{statement:Lenglart}).
\smallskip

\eqref{item:[1,infty)_sup:statement:pre_lemma_hoelder} for $p\in [3-2\theta,\infty)$: Again we assume $B\in \cF_c$ of positive measure.
By \cref{statement:basic_inequality_modulus_D_as} and \eqref{eqn:BDG} we obtain for
$s\in [c,T]$ that
\begin{align*}
     \E_{\P_B} |D^*_{\tau\wedge s}|^p
&\le 2^{p-1} e^{L_b(T-c)p} \left [ \E_{\P_B} |\Gamma_{\tau\wedge s}|^p + \E_{\P_B} |M^*_{\tau\wedge s}|^p \right ] \\
&\le 2^{p-1} e^{L_b(T-c)p} \left [ \E_{\P_B} |\Gamma_{\tau\wedge s}|^p + \beta_p^p L^p_\sigma \E_{\P_B} \left ( \int_c^{\tau\wedge s}|D_u|^{2\theta} \od u \right )^\frac{p}{2}
    \right ].
\end{align*}
Let $\rho:= 2\theta$ and $q:=2$, so that $1\le \rho < q$. Now we have either $p\ge q=2$ or both $p<q=2$ and $q-p+1\le \rho$.
In both cases we use \cref{statement:GR} to get an estimate
\begin{equation}\label{eqn:from_statement:GR}
2^{p-1} e^{L_b(T-c)p}  \beta_p^p L^p_\sigma \left ( \int_c^{\tau\wedge c} |D_u|^{2\theta} \od u \right )^\frac{p}{2}
   \le \frac{1}{2} D_{\tau\wedge s}^* + c_\eqref{eqn:from_statement:GR} \int_c^{\tau\wedge s} [|D_u| + |D_u|^p] \od u
\end{equation}
for some $c_\eqref{eqn:from_statement:GR}=c_\eqref{eqn:from_statement:GR}(T,L_b,L_\sigma,p)>0$.
Hence
\begin{align*}
     \E_{\P_B} |D_{\tau\wedge s}^*|^p
&\le  2^p e^{L_b(T-c)p} \E_{\P_B} \Gamma_{\tau\wedge s}^p + 2 c_\eqref{eqn:from_statement:GR} \E_{\P_B} \int_c^{\tau\wedge s} [|D_u|+ |D_u|^p] \od u \\
&\le  2^p e^{L_b(T-c)p} \E_{\P_B} \Gamma_{\tau\wedge s}^p + 2 c_\eqref{eqn:from_statement:GR} \E_{\P_B} \int_c^s [|D_{\tau\wedge u}|+ |D_{\tau\wedge u}|^p] \od u.
\end{align*}
Using Gronwall's inequality we get
\begin{align*}
     \E_{\P_B} |D_{\tau\wedge s}^*|^p
&\le   e^{2c_\eqref{eqn:from_statement:GR} (T-c)}  \left [ 2^p e^{L_b(T-c)p} \E_{\P_B} \Gamma_{\tau\wedge s}^p + 2 c_\eqref{eqn:from_statement:GR}  \E_{\P_B} \int_c^s |D_{\tau\wedge u}| \od u \right ] \\
&\le   e^{2c_\eqref{eqn:from_statement:GR} (T-c)}  \left [ 2^p e^{L_b(T-c)p} \E_{\P_B} \Gamma_{\tau\wedge s}^p + 2 c_\eqref{eqn:from_statement:GR} e^{L_b(T-c)} (s-c) \E_{\P_B} \Gamma_{\tau\wedge s} \right ]
\end{align*}
where we did use \eqref{item:1_without_sup:statement:pre_lemma_hoelder} in the last inequality. Setting $s=T$ the claim follows.
\smallskip

\eqref{item:[1,infty)_sup:statement:pre_lemma_hoelder} for $p\in [1,3-2\theta)=:[1,\overline{p}_\theta)$: For $\overline{p}_\theta\in (1,2]$
we know from the previous case that
\[ \E [ |D_\tau^*|^{\overline{p}_\theta} |\cF_c ] \le c_{\eqref{eqn:item:[1,infty)_sup:statement:pre_lemma_hoelder},\overline{p}_\theta}^{\overline{p}_\theta}
       \E [\Gamma_\tau^{\overline{p}_\theta} + \Gamma_\tau|\cF_c] \mbox{ a.s.}\]
Using again  Lenglart's inequality (\cref{statement:Lenglart}) implies for
$p\in [1,\overline{p}_\theta)$ that
\[      \E [ |D_\tau^*|^p |\cF_c ]
    \le \frac{q^{-q}}{1-q} c_{\eqref{eqn:item:[1,infty)_sup:statement:pre_lemma_hoelder},\overline{p}_\theta}^p
        \E [[\Gamma_\tau^{\overline{p}_\theta} + \Gamma_\tau]^\frac{p}{\overline{p}_\theta}|\cF_c] \mbox{ a.s.}\]
with $q:= p/\overline{p}_\theta\in (0,1)$. Note that
$p/\overline{p}_\theta = p/(3-2\theta)= p_\theta$.
\end{proof}
\bigskip

\begin{lemm}
\label{statement:statement:pre_lemma_hoelder_b=0}
If we assume that $b$ is deterministic, i.e. $b:[t,T]\times \R\to \R$, and $p\in [1,3-2\theta)$ in \cref{statement:pre_lemma_hoelder}, then
one has for all $B\in \cF_c$ of positive measure that
\begin{equation}\label{eqn:statement:statement:pre_lemma_hoelder_b=0}
    \| D^*_T \|_{L_{p,\infty}(B,\P_B)}
\le c_\eqref{eqn:statement:statement:pre_lemma_hoelder_b=0}
    \| |A| + |A|^\frac{1}{p} \|_{L_p(B,\P_B)}
\end{equation}
with $c_\eqref{eqn:statement:statement:pre_lemma_hoelder_b=0}=c_\eqref{eqn:statement:statement:pre_lemma_hoelder_b=0}(T,L_b,L_\sigma,\theta,p)>0$.
\end{lemm}

\begin{proof}
First we assume $p\in (1,3-2\theta)=:(p_0,p_1)$. By
\cref{statement:pre_lemma_hoelder} we have for any stopping time
$\tau:\Omega\to [c,T]$ that
\begin{align*}
\E[|D_\tau|^{p_0} |\cF_c] & \le e^{L_b(T-c)p_0} |A|^{p_0} \mbox{ a.s.},\\
\E[|D_\tau|^{p_1} |\cF_c] & \le c_{\eqref{eqn:item:[1,infty)_sup:statement:pre_lemma_hoelder},p_1}^{p_1}  [|A|^{p_1}+|A|] \mbox{ a.s.}
\end{align*}
We define $\eta\in (0,1)$ by $\frac{1}{p}=\frac{1-\eta}{p_0}+ \frac{\eta}{p_1}$. The conditional
H\"older inequality implies a.s. that
\begin{align*}
     \left ( \E[|D_\tau|^p | \cF_c]\right )^\frac{1}{p}
&\le \left ( \E[|D_\tau|^{p_0} | \cF_c]\right )^\frac{1-\eta}{p_0} \left ( \E[|D_\tau|^{p_1} | \cF_c]\right )^\frac{\eta}{p_1}\\
&\le e^{L_b(T-c)(1-\eta)} c_{\eqref{eqn:item:[1,infty)_sup:statement:pre_lemma_hoelder},p_1}^{\eta}
     |A|^{1-\eta}  [|A|^{p_1}+|A|]^\frac{\eta}{p_1} \\
&\le e^{L_b(T-c)(1-\eta)} c_{\eqref{eqn:item:[1,infty)_sup:statement:pre_lemma_hoelder},p_1}^{\eta}
     [|A|+|A|^\frac{1}{p}].
\end{align*}
So for $p\in [1,3-2\theta)$  there is a constant $c_{\eqref{eqn:proof:statement:statement:pre_lemma_hoelder_b=0},p}=c_{\eqref{eqn:proof:statement:statement:pre_lemma_hoelder_b=0},p}(T,L_b,L_\sigma,\theta,p)>0$
such that
\begin{equation}\label{eqn:proof:statement:statement:pre_lemma_hoelder_b=0}
     \left ( \E[|D_\tau|^p | \cF_c]\right )^\frac{1}{p}
 \le c_{\eqref{eqn:proof:statement:statement:pre_lemma_hoelder_b=0},p}  [|A|+|A|^\frac{1}{p}] \mbox{ a.s.}
\end{equation}
Let $\lambda>0$ and $B\in \cF_c$ of positive measure. Define the stopping time
\[ \tau_\lambda := \inf \{ s\in [c,T] : |D_s| > \lambda\} \wedge T.\]
Then
\begin{align*}
     \lambda^p \P_B \left ( \sup_{s\in [c,T]} |D_s| >\lambda  \right )
& =  \lambda^p \int_B \chf_{\{\tau_\lambda<T\}} \od \P_B \\
&\le \int_B \chf_{\{\tau_\lambda<T\}} |D_{\tau_\lambda}|^p \od \P_B \\
&\le \int_B |D_{\tau_\lambda}|^p \od \P_B \\
&\le c_{\eqref{eqn:proof:statement:statement:pre_lemma_hoelder_b=0},p}^p \int_B  [|A|+|A|^\frac{1}{p}]^p \od \P_B. \qedhere
\end{align*}
\end{proof}
\bigskip

\begin{proof}[Proof of \cref{statement:holder:main_general}]
We use that by \cref{statement:transference_composition_new} 
and \cref{statement:conditons_transference_composition_new}\eqref{item:1:statement:conditons_transference_composition_new}
we have the relations
\begin{align*}
     (b^\varphi(u,X_u^{t,\xi,\varphi}))_{u\in [t,T]}
&\in C_{[t,T]}^\cP((b(u,X_u^{t,\xi}))_{u\in [t,T]}),\\
     (\sigma(u,X_u^{t,\xi,\varphi}))_{u\in [t,T]}
&\in C_{[t,T]}^\cP((\sigma(u,X_u^{t,\xi}))_{u\in [t,T]}),
\end{align*}
so that for the distributions on $[t,T]\times \overline{\Omega}$ we get
\begin{align}
                  (b^\varphi(u,X_u^{t,\xi,\varphi}))_{u\in [t,T]}
& \stackrel{d}{=} (b(u,X_u^{t,\xi}))_{u\in [t,T]}, \label{eqn:distribution_b} \\
                  (\sigma(u,X_u^{t,\xi,\varphi}))_{u\in [t,T]}
& \stackrel{d}{=} (\sigma(u,X_u^{t,\xi}))_{u\in [t,T]} \label{eqn:distribution_sigma}
\end{align}
by \cref{statement:CPtT_is_isometry}.
For $p\in (0,\infty)$ we have
\[    \left\|\sup_{s\in [t,T]}\left|X^{t,\xi}_s - X^{t,\xi,\varphi}_s\right|\right\|_{L_p} \\
   \le 2^{\left ( \frac{1}{p} -1 \right )^+} \left [ \left\|\Delta \right\|_{L_p}
       + \left\|\sup_{s\in [c,T]}\left|X^{t,\xi}_s - X^{t,\xi,\varphi}_s\right|\right\|_{L_p} \right ]. \]
For $s\in [c,T]$ we have by \cref{statement:transference_sde} that
\begin{align*}
    X^{t,\xi,\varphi}_s
& = X^{t,\xi,\varphi}_c + \int^s_c b^{\varphi}(u,X^{t,\xi,\varphi}_u)\,\od u + \int^s_c \sz(u,X^{t,\xi,\varphi}_u)\,\od W_u, \\
    X^{t,\xi}_s
& =  X^{t,\xi}_c + \int^s_c b(u,X^{t,\xi}_u)\,\od u + \int^s_c \sz(u,X^{t,\xi}_u)\,\od W_u
\end{align*}
where we exploit that $W_s^\varphi -W_c^\varphi=W_s-W_c$ for $s\in [c,T]$.
Now we use the following notation:
\begin{align*}
D_s &:= X^{t,\xi}_s -  X^{t,\xi,\varphi}_s,\\
A   &:= X^{t,\xi}_c -  X^{t,\xi,\varphi}_c,\\
b_u & := b(u,X^{t,\xi}_u) - b^\varphi(u,X^{t,\xi,\varphi}_u),\\
\sigma_u &:= \sigma(u,X^{t,\xi}_u) - \sigma(u,X^{t,\xi,\varphi}_u),\\
\beta_u & := |b(u,X^{t,\xi,\varphi}_u) - b^\varphi(u,X^{t,\xi,\varphi}_u)|,\\
\Gamma_T &:= |A|+ \int_c^T |\sign_0(D_u)| \beta_u \od u \le \Delta + \Lambda.
\end{align*}
Using this notation we get that
\[ D_s = A + \int_c^s b_u \od u + \int_c^s \sigma_u \od W_u \mbox{ a.s.}\]
with
\[ |b_u|      \le L_b |D_u| + \beta_u
   \sptext{1}{and}{1}
   |\sigma_u| \le L_\sigma |D_u|^\theta. \]
Furthermore, we get
\[     \left ( \int_c^T |\sigma_s|^2 \od s \right )^\frac{1}{2} + \int_c^T |b_s|\od s  +\Gamma_T
   \in L_{p\vee 2}(\overline{\Omega},\overline{\cF},\overline{\P}).\]
Indeed, we use \eqref{eqn:distribution_b}, \eqref{eqn:distribution_sigma},
\eqref{ass:3:ass:hoelder}, \eqref{ass:4:ass:hoelder}, and \cref{statement:Xt_vs_starting_value_path} to see
$\left ( \int_c^T |\sigma_s|^2 \od s \right )^\frac{1}{2} + \int_c^T |b_s|\od s \in L_{p\vee 2}$
and to estimate the second term of $\beta_u$. To estimate the first term of $\beta_u$ we first use
\eqref{ass:3:ass:hoelder}, then that the distributions of $(X_u^{t,\xi,\varphi})_{u\in [t,T]}$ and $(X_u^{t,\xi})_{u\in [t,T]}$
coincide, and conclude with \cref{statement:Xt_vs_starting_value_path}.
\bigskip

So we can apply \cref{statement:pre_lemma_hoelder} and  \cref{statement:statement:pre_lemma_hoelder_b=0}, where we note that we apply these statements
under the stochastic basis $(\overline{\Omega},\overline{\cF},\overline{\P},(\overline{\cF}_s)_{s\in [0,T]})$ and the canonical extended Brownian motion $(W_s)_{s\in [0,T]}$.
\bigskip

\eqref{item:(0,1):statement:holder:sup:r:de}
If $p\in (0,1)$, then by \cref{statement:pre_lemma_hoelder}\eqref{item:(0,1)_sup:statement:pre_lemma_hoelder}
we have
\[       \left\|\sup_{s\in [c,T]}\left|X^{t,\xi}_s - X^{t,\xi,\varphi}_s\right | \right\|_{L_p}
   \le \frac{e^{L_bT}}{p \sqrt[p]{1 - p}} \| \Gamma_T  \|_{L_p}
   \le \frac{e^{L_bT}}{p \sqrt[p]{1 - p}} \left\|\Delta + \Lambda  \right \|_{L_p}. \]

\eqref{item:1:statement:holder:de}
Similarly we apply \cref{statement:pre_lemma_hoelder}\eqref{item:1_without_sup:statement:pre_lemma_hoelder} to get
\[             \left\|X^{t,\xi}_T - X^{t,\xi,\varphi}_T \right\|_{L_1}
  \le e^{L_bT} \| \Gamma_T  \|_{L_1}
  \le e^{L_bT} \left\|\Delta + \Lambda  \right \|_{L_1}.
\]

\eqref{item:[1,2]:statement:holder:sup:r:decoupling}
From \cref{statement:pre_lemma_hoelder}\eqref{item:[1,infty)_sup:statement:pre_lemma_hoelder} we get
\begin{align*}
       \left\|\sup_{s\in [c,T]}\left|X^{t,\xi}_s - X^{t,\xi,\varphi}_s\right | \right\|_{L_p}
&\le c_{\eqref{eqn:item:[1,infty)_sup:statement:pre_lemma_hoelder},p}  \left ( \E [\Gamma_T^p + \Gamma_T^{p_\theta} ] \right )^\frac{1}{p} \\
&\le c_{\eqref{eqn:item:[1,infty)_sup:statement:pre_lemma_hoelder},p}
     \left [ \| \Gamma_T \|_{L_p} + \|\Gamma_T\|_{L_{p_\theta}}^{\frac{p_\theta}{p}} \right ].
\end{align*}
\medskip

\eqref{item:[1,2]:statement:holder:main:deter}
follows from \cref{statement:statement:pre_lemma_hoelder_b=0} as in \eqref{item:[1,2]:statement:holder:sup:r:decoupling}.
\end{proof}


\subsection{An application of \cref{statement:holder:main_general}}\ \medskip

We start by two lemmas. The first lemma is intuitive, but requires a formal proof:

\begin{lemm}
\label{statement:ab-transformation_is_local}
Assume $0\le a < c \le T$ and an $\cF_a^0$-measurable random variable $A:\overline{\Omega}\to \R$.
Then one has $A=A^{(a,c]}$ $\overline{\P}$-a.s.
\end{lemm}

\begin{proof}
If $a=0$, then $A$ is almost surely a constant, so $A^{(a,c]}$ equals almost surely the same constant as
$A$ and $A^{(a,c]}$ have the same law.
\smallskip

Assume $a>0$. By our construction we have $W_s^0\equiv W_s^{(a,c]}$ for $s\in [0,a]$.
We find a sequence of non-negative integers $n_1,n_2,\ldots$ and bounded Borel functions
$f_l:\R^{2^{n_l}}\to \R$ such that for
$A_l:=f_l(W_{ \frac{a}{2^{n_l}}},\ldots,W_{\frac{2^{n_l}a}{2^{n_l}}})$ one has $A_l\to A$ in probability with 
respect  to $\overline{\P}$. By \cite[Proposition 2.5(4)]{Geiss:Ylinen:21} we have 
$A_l^{(a,c]}=A_l$ $\overline{\P}$-a.s. and by \cite[Proposition 2.5(2)]{Geiss:Ylinen:21} that the convergence in probability 
is preserved. So we get $A_l^{(a,c]}\to A^{(a,c]}$ in probability with respect  to $\overline{\P}$ and therefore
$A=A^{(a,c]}$ $\overline{\P}$-a.s.
\end{proof}
\medskip

The next lemma uses only the size conditions on $b$ and $\sigma$:

\begin{lemm}
\label{statement:decoupling_small_t}
Assume conditions \eqref{ass:3:ass:hoelder} and \eqref{ass:4:ass:hoelder}, $p\in (0,\infty)$, $0\leq t\leq a < c\leq T$, and
$\xi \in L_{p\vee 2}$.
Then, for some $c_{\eqref{eqn:statement:decoupling_small_t},p}=c_{\eqref{eqn:statement:decoupling_small_t},p}(T,K_b,K_\sigma,p)>0$, one has
\begin{equation}\label{eqn:statement:decoupling_small_t}
      \left \| \sup_{u\in [a,c]}\left|X_u^{t,\xi}- X_u^{t,\xi,(a,c]}\right| \right \|_{L_p}
 \leq c_{\eqref{eqn:statement:decoupling_small_t},p} (c - a)^{\frac{1}{2}}
      [ 1+ \|\xi\|_{L_p}].
\end{equation}
\end{lemm}

\begin{proof}
By \cref{statement:ab-transformation_is_local} for any $u\in (a,c]$, we have
\begin{multline*}
X_u^{t,\xi}- X_u^{t,\xi,(a,c]} = \int^u_a b(s,X_s^{t,\xi}) - b^{(a,c]}(s,X_s^{t,\xi,(a,c]})\,\od s \\
 + \int^u_a \sigma(s,X_s^{t,\xi}))\,\od W_s - \int^s_a \sigma(s,X_s^{t,\xi,(a,c]})\,\od W'_s \mbox{ a.s.}
\end{multline*}
As we have \eqref{eqn:distribution_b} and \eqref{eqn:distribution_sigma},
we use \eqref{eqn:BDG} to deduce that
\begin{align*}
&  \hspace*{-2em}    \E \left[\sup_{u\in [a,c]}\left| X_u^{t,\xi} - X_u^{t,\xi,(a,c]} \right|^p\right] \\
\leq & 4^{(p-1)^+} \Bigg \{
            \E\left[ \int^c_a |b        (u,X_u^{t,\xi})|\od u\right]^p
          + \E\left[ \int^c_a |b^{(a,c]}(u,X_u^{t,\xi,(a,c]})|\od u\right]^p \\
&     + \beta_p^p  \E\left[ \int^c_a |\sigma        (u,X_u^{t,\xi})|^2\od s \right ]^\frac{p}{2}
      + \beta_p^p  \E\left[ \int^c_a |\sigma(u,X_u^{t,\xi,(a,c]})|^2\od s \right ]^\frac{p}{2} \Bigg \} \\
= & 2\,\, 4^{(p-1)^+} \Bigg \{
            \E\left[ \int^c_a |b        (u,X_u^{t,\xi})|\od u\right]^p
            + \beta_p^p  \E\left[ \int^c_a |\sigma        (u,X_u^{t,\xi})|^2\od s \right ]^\frac{p}{2} \Bigg \}  \\
\leq & 2 \,\, 4^{(p-1)^+} \Bigg \{
            K_b^p \E\left[ \int^c_a [ 1+|X_u^{t,\xi}| ] \od u\right]^p
            + \beta_p^p  K_{\sigma}^p \E\left[ \int^c_a [1+|X_u^{t,\xi}|        ]^2\od u \right ]^\frac{p}{2} \Bigg \} \\
\leq & 2 \,\, 4^{(p-1)^+} \left [  K_b^p (c-a)^p +  \beta_p^p  K_{\sigma}^p  (c-a)^\frac{p}{2} \right ] \,
        \E \left [1+ \sup_{u\in [a,c]} |X_u^{t,\xi}| \right ]^p.
\end{align*}
By \cref{statement:Xt_vs_starting_value_path} we have
\[    \E \left [1+ \sup_{u\in [a,c]} |X_u^{t,\xi}| \right ]^p
   \le c^p_{\eqref{eqn:statement:Xt_vs_starting_value_path},p} \E [ 1+|\xi|]^p] \]
so that the proof is complete.
\end{proof}

Combining \cref{statement:holder:main_general} with \cref{statement:decoupling_small_t} we get:
\smallskip

\begin{coro}
\label{statement:holder:main}
Let $0\leq t \leq a < c\leq T$ and $(X^{t,\xi}_s)_{s\in [t,T]}$ be a solution to SDE \eqref{eqn:SDE_hoelder}
with the coefficients satisfying \cref{ass:hoelder}, and $(X^{t,\xi,(a,c]}_s)_{s\in [t,T]}$ be its decoupled
process.
For $p,q\in (0,\infty)$ let $\xi \in L_{p\vee 2}(\Omega,\cF_t,\P)$,
\[ \Lambda := \int_c^T \left | b\left (s,X^{t,\xi,(a,b]}_s\right ) - b^{(a,c]}\left (s,X^{t,\xi,(a,b]}_s\right )\right|\od s
   \sptext{1}{and}{1}
   p_\theta := \min \left \{ 1,\frac{p}{3-2\theta} \right \}. \]
Then the following holds:
\medskip
\begin{enumerate}[{\rm\bf (1)}]
\item \label{item:(0,1):statement:holder:sup:r:de}
      If \underline{$p\in (0,1)$}, then one has
      \begin{equation}\label{eqn:item:(0,1):statement:holder:main:ran:de}
      \left\|\sup_{s\in [t,T]}\left|X^{t,\xi}_s - X^{t,\xi,(a,c]}_s\right|\right\|_{L_p}
      \leq c_\eqref{eqn:item:(0,1):statement:holder:main:ran:de}  \left[(c - a)^\frac{1}{2} (1+\|\xi\|_{L_p}) + \|\Lambda \|_{L_p} \right],
      \end{equation}
      where $c_\eqref{eqn:item:(0,1):statement:holder:main:ran:de}=c_\eqref{eqn:item:(0,1):statement:holder:main:ran:de}(T,L_b,K_b,K_\sigma,p)>0$.
      \bigskip

\item \label{item:1:statement:holder:de}
      One has
      \begin{equation}\label{eqn:item:1:statement:holder:de}
      \left\|X^{t,\xi}_T - X^{t,\xi,(a,c]}_T\right\|_{L_1}
      \leq c_\eqref{eqn:item:1:statement:holder:de} \left[(c - a)^\frac{1}{2} (1+\|\xi\|_{L_1}) + \|\Lambda \|_{L_1} \right],
      \end{equation}
      where $c_\eqref{eqn:item:1:statement:holder:de}=c_\eqref{eqn:item:1:statement:holder:de}(T,L_b,K_b,K_\sigma)>0$.
      \bigskip

 \item \label{item:[1,2]:statement:holder:sup:r:decoupling}
       If \underline{$p\in [1,\infty)$}, then one has
       \begin{multline}\label{eqn:item:[1,2]:statement:holder:main:ran}
       \frac{1}{c_\eqref{eqn:item:[1,2]:statement:holder:main:ran}} \left\|\sup_{s\in [t,T]}\left|X^{t,\xi}_s - X^{t,\xi,(a,c]}_s\right|\right\|_{L_p} \\
       \leq (c-a)^\frac{1}{2} (1+\|\xi\|_{L_p}) + \left [ (c-a)^\frac{1}{2} (1+\|\xi\|_{L_{p_\theta}}) \right ]^{\frac{p_\theta}{p}}
            + \|\Lambda \|_{L_p} + (\E |\Lambda|^{p_\theta})^\frac{1}{p},
       \end{multline}
       where $c_\eqref{eqn:item:[1,2]:statement:holder:main:ran}=c_\eqref{eqn:item:[1,2]:statement:holder:main:ran}(T,L_b,L_\sigma,K_b,K_\sigma,\theta,p)>0$.
       \bigskip

\item \label{item:[1,2]:statement:holder:main:deter}
       If \underline{$p\in [1,3 - 2\tz)$} and $b: [0,T] \times \R  \to \R$, then
       one has
       \begin{multline}\label{eqn:item:[1,2]:statement:holder:main:deter}
              \frac{1}{c_\eqref{eqn:item:[1,2]:statement:holder:main:deter}} \left\|\sup_{s\in [t,T]}\left|X^{t,\xi}_s - X^{t,\xi,(a,c]}_s\right|\right \|_{L_{p,\infty}} \\
         \leq (c-a)^\frac{1}{2} (1+\|\xi\|_{L_p})
                       + \left [ (c-a)^\frac{1}{2} (1+\|\xi\|_{L_1}) \right ]^{\frac{1}{p}},
        \end{multline}
       where
       $c_\eqref{eqn:item:[1,2]:statement:holder:main:deter}=c_\eqref{eqn:item:[1,2]:statement:holder:main:deter}(T,L_b,L_\sigma,K_b,K_\sigma,\theta,p)>0$.

\end{enumerate}
\end{coro}
\bigskip


\subsection{A lower bound for \cref{statement:holder:main}} \ \medskip

In this section we address the optimality of \cref{statement:holder:main}\eqref{item:[1,2]:statement:holder:sup:r:decoupling}.
For this we assume
\[ \theta\in \left [ \frac{1}{2},1 \right ), \quad 0=t=a<c\le 1 < T=2, \sptext{1}{and}{1} b\equiv 0 \]
for the drift and use that for $\eta \in L_q(\Omega,\cF_c,\P)$
with $q\in [2,\infty)$ the SDE
\begin{equation}\label{eqn:sde_eta}
 X_s = \eta + \int^s_c |X_u|^{\tz}\,\od W_u
   \sptext{1}{for}{1}
   s\in [c,T]
\end{equation}
has a unique strong solution with $\sup_{s\in [c,T]} |X_s| \in L_q$ by
\cite[Corollary 5.5.10]{Karatzas:Shreve:98} and \cref{statement:Xt_vs_starting_value_path}.
\bigskip

The next \cref{statement:sharpness_2p} shows the following: If we consider in
\cref{statement:holder:main}\eqref{item:[1,2]:statement:holder:sup:r:decoupling} the
case $p\in [3-2\theta,\infty)$, then $p_\theta=1$ and
$[(c-a)^\frac{1}{2}]^\frac{p_\theta}{p} = (c-a)^\frac{1}{2p}$. This exponent $\frac{1}{2p}$ is
optimal for $(c-a)\downarrow 0$ when we require that the constant $c_\eqref{eqn:item:[1,2]:statement:holder:main:ran}>0$
may depend only on $(T,L_b,L_\sigma,K_b,K_\sigma,\theta,p)$.
\bigskip

\begin{prop}
\label{statement:sharpness_2p}
Let $\xi:= 0$,
$p\in [3-2\theta,\infty)$, $T=2$, $b\equiv 0$, and
\[ \sigma_c(t,x) = \sigma(t,x) := \begin{cases}
                            1 &: t\in [0,c] \\
                   |x|^\theta &: t\in (c,2]
                   \end{cases}.\]
Then \cref{ass:hoelder} is satisfied with $L_b=K_b=0$ and $L_\sigma=K_\sigma =1$, and
for the solution to \eqref{eqn:SDE_hoelder} one has
\begin{equation}\label{eqn:statement:sharpness_2p}
 \left \|X_T^{0,0} - X_T^{0,0,(a,c]} \right \|_{L_p}
\ge c_{\eqref{eqn:statement:sharpness_2p},p} (c-a)^\frac{1}{2p}
\end{equation}
with
\[ c_{\eqref{eqn:statement:sharpness_2p},p}:= \left [\frac{1}{4} \E|W_1|  \left[\frac{(3 - 2\tz)(2 - 2\tz)}{2}\right]^{\frac{p - 1}{2 - 2\tz}}\right ]^{\frac{1}{p}}.\]

\end{prop}
\medskip

It follows from the remarks in the beginning of this subsection that the SDE,
which depends on $c$, considered in \cref{statement:sharpness_2p},
\[ X_s^{0,0} = \int^s_0 \sigma_c(u,X_u^{0,0}) \,\od W_u
   \sptext{1}{for}{1}
   s\in [0,2], \]
has a unique strong solution with $\sup_{s\in [0,2]} |X_s^{0,0}| \in L_q$ for all $q\in (0,\infty)$.

For the proof of \cref{statement:sharpness_2p} we exploit the following lemma:
\medskip

\begin{lemm}\label{statement:holder:sup:lowbd} 
Let $p\in [3-2\theta,\infty)$,
$\eta\in L_{p\vee 2}(\Omega,\cF_c,\P)$, and let $(X_t)_{t\in [c,T]}$ be the solution to
\eqref{eqn:sde_eta}. Then one has
\begin{equation}\label{eqn:statement:holder:sup:lowbd}
     \E |X_s|^p
\geq \left[\frac{(3 - 2\tz)(2 - 2\tz)}{2}(s - c)\right]^{\frac{p - 1}{2 - 2\tz}} \E|\eta|
     \sptext{1}{for}{1} s\in [c,T].
\end{equation}
\end{lemm}

\begin{proof}
(a) We observe that for $p\in [3-2\theta,\infty)$ one has $p>1$ and
\begin{equation}\label{eqn:Lp_moment_Xs}
\E |X_s|^p = \E |\eta|^p + \frac{p(p-1)}{2} \int_c^s \E |X_u|^{2\theta+p-2} \od u.
\end{equation}
To check \eqref{eqn:Lp_moment_Xs}, for $n\ge 1$ we define the continuous and bounded functions $v_n:[0,\infty)\to [0,\infty)$  by
$v_n(x) := p(p-1) (x^{p-2}\wedge n)$ with $0^0:=1$ and
$\Phi_n:\R\to [0,\infty) \in C_b^2(\R)$ by $\Phi_n(x) := \int_0^{|x|} \int_0^y v_n(z)\od z \od y$.
Then $\Phi_n(x)\uparrow |x|^p$ for $x\in \R$ and $\Phi''_n(x) \uparrow p(p-1) |x|^{p-2}$ for $x\not = 0$.
By It\^o's formula,
\begin{equation}\label{eqn:Lp_moment_Xs_Phi}
\E \Phi_n(X_s) = \E \Phi_n(\eta) + \E \int_c^s \Phi_n'(X_u) |X_u|^\theta \od W_u
                 + \frac{1}{2} \int_c^s \E \Phi_n''(X_u) |X_u|^{2\theta} \od u.
\end{equation}
Note that $\E \sup_{u\in [c,T]}|X_u|^2 < \infty$ by \cref{statement:Xt_vs_starting_value_path} and that
$|\Phi_n'(x)| \le \int_0^{|x|} v_n(z) \od z \le p(p-1) n |x|$, which implies
\[ |\Phi_n'(X_u)||X_u|^\theta \le p(p-1) n |X_u|^{1+\theta}. \]
Hence
\begin{align*}
     \E \left ( \int_c^T [ \Phi_n'(X_u) |X_u|^\theta ]^2 \od u \right )^\frac{1}{2}
&\le p(p-1)n (T-c)^\frac{1}{2} \E \sup_{u\in [c,T]} |X_u|^{1+\theta} < \infty.
\end{align*}
Using the BDG-inequality \eqref{eqn:BDG} with the parameter 1 implies that the stochastic integral  term is a true martingale
and  therefore its expected value vanishes. Letting $n\to \infty$ in
\eqref{eqn:Lp_moment_Xs_Phi}, while observing
$0\le \Phi_n''(x) |x|^{2\theta} \uparrow p(p-1) |x|^{2\theta+p-2}$ for $x\in \R$ we get \eqref{eqn:Lp_moment_Xs}.
\medskip

(b) When $p = 3 - 2\tz$ (so that $2\theta +p-2=1$) we have for any $s\in [c,T]$ that
        \begin{align}
            \E|X_s|^{3 - 2\tz}
            & = \E|\eta|^{3 - 2\tz} + \frac{(3 - 2\tz)(2 - 2\tz)}{2} \int_c^s \E|X_u|\,\od u\notag\\
            & \geq \E|\eta|^{3 - 2\tz} + \frac{(3 - 2\tz)(2 - 2\tz)(s - c)}{2} \E |\eta| \notag \\
            & \geq \frac{(3 - 2\tz)(2 - 2\tz)(s - c)}{2} \E |\eta| \label{lowbd:re1_new}
        \end{align}
which gives our statement.
\medskip

(c) Assume $p\in (3-2\theta,\infty)$ and let $q_0:=1$, $q_1:= p$, and
$q:= 3-2\theta$, so that $1=q_0 < q < q_1 =p$. Then, by H\"older's inequality,
\[ \|X_s \|_{L_q} \le \|X_s\|_{L_{q_0}}^{1-\delta} \|X_s\|_{L_{q_1}}^\delta
   \sptext{1}{for}{1}
   \frac{1}{q} = \frac{1-\delta}{q_0} + \frac{\delta}{q_1}
   \sptext{1}{with}{1}
   \delta:= \frac{1-\frac{1}{3-2\theta}}{1-\frac{1}{p}}.\]
From \cref{statement:basic_equality_modulus_D} we know that
$\|X_s\|_{L_{q_0}} = \|X_s\|_{L_1} = \|\eta\|_{L_1}$. Using \eqref{lowbd:re1_new} we get
\[  \left (  \frac{(3 - 2\tz)(2 - 2\tz)}{2} (s-c) \|\eta\|_{L_1} \right )^\frac{1}{3-2\theta}
    \le \|\eta\|_{L_1}^{1-\delta} \| X_s\|_{L_p}^\delta.\]
If $\|\eta\|_{L_1}>0$ (otherwise the statement of the lemma holds trivially), then rearranging the terms the assertion
follows.
\end{proof}
\smallskip

\begin{proof}[Proof of \cref{statement:sharpness_2p}]
We write $X_s:=X_s^{0,0}$ and
$X_s^{(a,c]}:=X_s^{0,0,(a,c]}$. We get
\[ \E |X_T-X_T^{(a,c]}|^p = \E |X_T^{c,W_c}-X_T^{c,W_c'}|^p
                         \ge  \E \chf_B |X_T^{c,W_c}-X_T^{c,W_c'}|^p \]
with
\[ B:= \{(\omega,\omega')\in \overline{\Omega}: W_c(\omega) \le 0 \le W_c'(\omega')\}.\]
By the relation $\chf_B(\omega,\omega')\sigma(s,x) = \sigma(s,\chf_B(\omega,\omega') x)$
we get the SDEs
\begin{align*}
\chf_B  X_s^{c,W_c}  & = \chf_B  W_c + \int_c^s \sigma(u,\chf_B X_u^{c,W_c}) \od W_u, \\
\chf_B  X_s^{c,W'_c} & = \chf_B  W'_c + \int_c^s \sigma(u,\chf_B X_u^{c,W'_c}) \od W_u,
\end{align*}
both for $s\in [c,T]$ a.s.
Using the comparison theorem \cite[Proposition 5.2.18]{Karatzas:Shreve:98} on $[c,T]$
(where we redefine $\sigma(c,x)$ so that $\sigma:[c,T]\times \R\to \R$ is continuous and
use the filtration $(\overline{\cF_s})_{s\in [c,T]}$)
we get
\[ \chf_B X_T^{c,W_c} \le 0 \le \chf_B X_T^{c,W'_c} \mbox{ a.s.}
   \sptext{1}{as}{1} \chf_B W_c \le 0 \le \chf_B W'_c.
\]
By  \cref{statement:holder:sup:lowbd} this implies that
\begin{align*}
      \E |X_T-X_T^{(a,c]}|^p
&\ge \E \left |X_T^{c,\chf_B W_c} \right |^p \\
&\ge \left [ \frac{(3-2\theta)(2-2\theta)}{2} (T-c) \right ]^\frac{p-1}{2-2\theta} \E |\chf_B W_c| \\
&\ge \left [ \frac{(3-2\theta)(2-2\theta)}{2} \right ]^\frac{p-1}{2-2\theta} \frac{1}{4} \sqrt{c} \E |W_1|. \qedhere
\end{align*}
\end{proof}


\section{The Zvonkin and Lamperti transform and coupling}
\label{sec:Zvonkin}

For $t=0$ we describe how the Zvonkin and Lamperti transform can be used 
in our context.
\smallskip

The coupling method for SDEs and the Zvonkin method to remove an irregular drift
for SDEs work very well together.
Regarding the original Zvonkin method and Zvonkin methods in the more general sense
the reader is referred, for example, to
\cite{Zvonkin:1974,
      Veretennikov:1981,
      Krylov:Roeckner:2004,
      XZhang:2016,
      deRaynal:2017,
      Xia:Xie:XZhang:Zhao:2020,
      SQZhang:Yuan:2021,
      Cheng:Hao:Roeckner:2024}.
We address two ways to apply the Zvonkin method in our context.


\subsection{Zvonkin method and quasi-isometries}\ \medskip

The initial article of Zvonkin \cite{Zvonkin:1974} and later works use
quasi-isometries in the following way: There is a function $\Phi:[0,T]\times \R^d \to \R^d$
and $M>1$ such that
\[ 0<\frac{1}{M} \le \frac{|\Phi(s,y_1)-\Phi(s,y_2)|}{|y_1-y_2|} \le M < \infty\]
for $s\in [0,T]$ and $y_1,y_2\in \R^d$,
and with the following property: If for some - in a sense - more appropriate coefficients $(\widetilde{b},\widetilde{\sigma})$
 one has a solution
\[ Y_s = \zeta + \int_0^s \widetilde{b}(u,Y_u) \od u + \int_0^s \widetilde{\sigma}(u,Y_u) \od W_u, \]
then $X_s:= \Phi(s,Y_s)$ solves
\[ X_s = \Phi(0,\zeta) + \int_0^s b(u,X_u) \od s + \int_0^s \sigma(u,X_u) \od W_u. \]
If we apply a $\varphi$-coupling to $X$ and $Y$ and use that $x\mapsto \Phi(s,x)$ is continuous, hence measurable,
then we get that
\[
  \E |X_s-X_s^\varphi|^p
= \E |\Phi(s,Y_s)- (\Phi(s,Y_s))^\varphi|^p
= \E |\Phi(s,Y_s)- \Phi(s,Y_s^\varphi)|^p
\sim_M \E |Y_s- Y_s^\varphi|^p. \]
This means that results for couplings of the process $Y$ transfer directly to couplings for the process $X$.


\subsection{Zvonkin method and the dependence on the initial value}\ \medskip

For example, it is shown in
\cite[page 5205]{Xia:Xie:XZhang:Zhao:2020}
that one has a quasi-isometry $\Phi$ as above
and in equation (4.10) on page 5207 that for the $Y$ process
one has
\begin{equation}\label{eqn:initial_value_Zvonkin}
\E \sup_{s\in [0,T]} |Y_s^{0,y_1} - Y_s^{0,y_2}|^p \le \kappa^p |y_1-y_2|^p
\end{equation}
for some $p\in [1,\infty)$ and $\kappa>0$; for $\sigma\equiv 1$ cf. \cite[Corollary 14]{Mohammed:etal:2015}.
By the quasi-isometry $\Phi$ this property transfers to the process $X$ (with
a different constant $\kappa$). 
The above suggests to replace \cref{statement:pre_lemma_hoelder} by a statement of type \eqref{eqn:initial_value_Zvonkin}
and then to continue as in the proof of \cref{statement:holder:main}. This idea is captured by 
the following statement, which is somehow an alternative to the known path to deduce Malliavin differentiability from the 
sensitivity of the starting value for example taken in \cite[Lemma 2.3, Theorem 1.1]{XZhang:2016}:

\medskip

\begin{prop}
Assume the setting and conditions of \cref{statement:decoupling_small_t} for
$p\in [2,\infty)$, $x\in \R$, and some $\kappa\ge 0$ such that for all $0\le a < c \le T$ one has
\[ \E | X_T^{0,x} - X_T^{0,x,(a,c]}|^p \le \kappa^p \E | X_c^{0,x} - X_c^{0,x,(a,c]} |^p.\]
Then one has $X_T^{0,x} \in \B_p^{\Phi_2}\subseteq \D_{1,2}$ with
\[ \sup_{0\le a < c \le T} \left \| \left ( \frac{1}{c-a} \int_a^c |D_s (X_T^{0,x})|^2 \od s \right )^\frac{1}{2} \right \|_{L_p}
   \le \kappa c_{\eqref{eqn:characterization_Phi2p},p} c_{\eqref{eqn:statement:decoupling_small_t},p} [1+ |x|].\]
\end{prop}

\begin{proof}
By assumption and relation \eqref{eqn:statement:decoupling_small_t} we observe
\begin{align*}
    \| | X_T^{0,x} - X_T^{0,x,(a,c]}\|_{L_p}
&\le \kappa \| X_c^{0,x} - X_c^{0,x,(a,c]} \|_{L_p} \\
&\le \kappa c_{\eqref{eqn:statement:decoupling_small_t},p} (c - a)^{\frac{1}{2}}
      [ 1+ |x|].
\end{align*}
Then we use relation \eqref{eqn:characterization_Phi2p}.
\end{proof} 


\subsection{Lamperti transform and the dependence on the initial value}\ \medskip

There are several studies on the Cox-Ingersoll-Ross (CIR) process
$\od X_s = (A - B  X_s) \od s + \sigma \sqrt{X_s} \od W_s$ 
with $A,B,\sigma>0$. Let $p\in [2,\infty)$. Under the condition $4 A > \sigma^2$  
\cite[Proposition 39]{Desmettre_etal:2023} provides the estimate 
$\E \sup_{s\in [0,T]} |X_s^{0,x_1} - X_s^{0,x_2}|^p \le \kappa^p \left [ |x_1-x_2|^p +  |x_1-x_2|^\frac{p}{2} \right ]$
for $x_1,x_2>0$. For $p>2$ the exponent $\frac{p}{2}>1$ is better then the corresponding exponent in \cref{statement:pre_lemma_hoelder}
as $p_\theta =1$ in this case, however the structure of the CIR-process is used together
with $4 A > \sigma^2$. Moreover, by \cref{statement:holder:sup:lowbd} we show for 
$A=B=0$ and $\sigma=1$ that in general we cannot have an estimate of form
\[ \E |X_T^{0,x_1} - X_T^{0,x_2}|^p \le \kappa^p \left [ |x_1-x_2|^p +  |x_1-x_2|^\frac{p}{2} \right ] \]
for $x_1,x_2\in \R$ when $p>2$ (take $x_1=0$ and consider small $x_2$). 


\section{Application to BSDEs}
\label{sec:application:BSDE}

Regularity properties of the solution $(Y,Z)$ to a BSDE are of theoretical interest, but also
 - for example - used for the design of numerical schemes (see for instance
\cite[Lemma 4.11]{Lionnet:dosReis:Lukasz:2015} for the locally Lipschitz and Markovian context).
Recent results about Malliavin differentiability
were obtained in \cite{Geiss:Ylinen:21}[Corollary 1.6] and  in \cite[Theorem 4.2]{Imkeller_etal:2024},
and about general Besov regularity in \cite{Geiss:Ylinen:21}[Corollary 6.22].
\smallskip

Here we present further results based on the coupling approach 
in a fully path-dependent setting under minimal structural assumptions.
The setting is as follows:
\medskip

For $\alpha \in (0,1]$, $\delta\in [0,1]$, and functions
\[ g:C([t,T])\to \R
   \sptext{1}{and}{1}
   f:[0,T]\times C([t,T]) \times \R \times \R \to \R \]
we assume  and $|g|_{\alpha}, |f|_{\alpha;x}, L_Y,L_Z > 0$
\begin{align*}
u       \mapsto  f(u,x,y,z) & \sptext{1}{is measurable if}{.5} (x,y,z)\in  C([t,T]) \times \R \times \R,\\
(x,y,z) \mapsto  f(u,x,y,z) & \sptext{1}{is continuous if}{4} u\in [0,T], \\
     f(s,x,y,z)
& =  f(s,x(\cdot \wedge s) ,y,z) \sptext{2.0}{if}{.4} s\in (t,T],\\
     f(s,x,y,z)
& =  0 \sptext{5.8}{if}{.4} s\in [0,t],\\
     |g(x_1)-g(x_2)|
&\le |g|_{\alpha} \|x_1-x_2\|_{C[t,T]}^\alpha, \\
     \int_t^T |f(u,x_1,y,z)-f(u,x_2,y,z)|\od u
&\le |f|_{\alpha;x} \|x_1-x_2\|_{C[t,T]}^\alpha, \\
     |f(s,x,y_1,z_1) - f(s,x,y_2,z_2)|
& \le L_Y |y_1 - y_2| +  L_Z [1+|z_1|+|z_2|]^\delta |z_1-z_2|.
\end{align*}
As we have a backward problem and go backward from $T$ only up to time $t$, we put the generator to zero
on $[0,t]$. The following \cref{ass:BSDE} is in full accordance with \cite{Geiss:Ylinen:21}:

\begin{assumption}
\label{ass:BSDE}
For $(\eta,\mu)\in (0,1]\times (0,\infty)$ and a random variable  $F:\Omega \to \R$  we let
\[ |F|_{\rm cEXP(\eta,\mu)} := \sup_{s\in [0,T)} (T-s)^{\frac{1}{\eta}-1} \left \| \E[e^{\mu |F|}|\cF_s] \right \|_{L_\infty}.\]
For $p\in [2,\infty)$ and a continuous adapted process
$X=(X_s)_{s\in [t,T]}$ we assume one of the following cases in accordance
with \cite[Table 1, page 73]{Geiss:Ylinen:21}:
\smallskip

\begin{tabular}{|l|l|l|} \hline
 type & terminal condition $F=g(X)$ & generator \\ \hline \hline
$\delta=0$ &  $F\in L_p$ & $\int_0^T |f(s,X,0,0)| \od s \in L_p$ \\ \hline
$\delta\in (0,1)$ & $|F|_{\rm cEXP(\eta,\mu)}<\infty$ for  & $\sup\limits_{s\in [0,T] \atop \omega\in \Omega}|f(s,X(\omega),0,0)|<\infty$ \\
                  &  some $(\eta,\mu)\in (0,1)\times (0,\infty)$ & \\ \hline
$\delta=1$ & $|F|_{\rm cEXP(\eta,\mu)}<\infty$ for  & $\sup\limits_{s\in [0,T] \atop\omega\in \Omega}|f(s,X(\omega),0,0)|<\infty$ \\
           &  some $\eta \in  (0,1]$ and $\mu > 4 L_Z e^{L_Y T}$ & \\ \hline
\end{tabular}
\end{assumption}
\smallskip
The above assumption guarantees the existence of a solution to the BSDE
\[ Y_s=F + \int_s^T f(u,X,Y_u,Z_u) \od u - \int_s^T Z_u \od W_u,
   \quad s\in [t,T],\]
with
\begin{equation}\label{eqn:Lp_boundedness_YZ}
 \int_0^T |f(s,X_s,Y_s,Z_s)| \od s + \sup_{s\in [0,T]} |Y_s| \in L_p 
\end{equation}
(see \cite[Lemma 6.2]{Geiss:Ylinen:21}) that we will fix from now on.
The following Theorems \ref{statement:BSDE_decoupling} and \ref{statement:BSDE_variation}
are direct consequences of \cite{Geiss:Ylinen:21} and hold for general continuous adapted processes $X$:

\begin{theo}
\label{statement:BSDE_decoupling}
Assume $\alpha \in (0,1]$ and \cref{ass:BSDE}. Then one has
\begin{multline}\label{eqn:statement:BSDE_decoupling}
      \left \| \sup_{s\in [t,T]} |Y_s^\varphi - Y_s| \right \|_{L_p} + \left \| \left ( \int_t^T |Z_s^\varphi- Z_s|^2 \od s \right )^\frac{1}{2} \right \|_{L_p} \\
    \le c_\eqref{eqn:statement:BSDE_decoupling}
     \left [ |g|_\alpha +  |f|_{\alpha;x} \right ] \left \| \sup_{s\in [t,T]} |X^\varphi_s-X_s| \right \|_{L_{\alpha p}}^\alpha
\end{multline}
with $c_\eqref{eqn:statement:BSDE_decoupling}=c_\eqref{eqn:statement:BSDE_decoupling}(T,L_Y,L_Z,(Z_s)_{s\in [0,T]},\delta,p)>0$.
\end{theo}
\medskip

\begin{proof}
Using \cite[Theorem 6.3]{Geiss:Ylinen:21} and
\cref{statement:transference_composition_new}\eqref{item:1:statement:transference_composition_new}
(we use this for non-random $h^0$) we get
\begin{align*}
&    \left \| \sup_{s\in [t,T]} |Y_s^\varphi - Y_s| \right \|_{L_p} + \left \| \left ( \int_t^T |Z_s^\varphi- Z_s|^2 \od s \right )^\frac{1}{2} \right \|_{L_p} \\
&\le c \left [ \| g(X^\varphi)- g(X)\|_{L_p} + \left \| \int_t^T |f(s,X^\varphi,Y_s,Z_s) - f(s,X,Y_s,Z_s)| \od s \right \|_{L_p} \right ] \\
&\le c \left [ |g|_\alpha +  |f|_{\alpha;x} \right ] \left \| \sup_{s\in [t,T]} |X^\varphi_s-X_s| \right \|_{L_{\alpha p}}^\alpha,
\end{align*}
where $c=c(T,L_Y,L_Z,(Z_s)_{s\in [0,T]},\delta,p)>0$.
\end{proof}
\medskip

Now we apply \cref{statement:BSDE_decoupling} to the diffusion process $X=(X_s)_{s\in [t,T]}$ investigated in \cref{sec:Lipschitz}.
For $0<\eta<\alpha \le 1$ and $q\in [1,\infty]$ we get that

\begin{align*}
  | Y |_{\B_p^{\Phi_2,*}}
      + | Z |_{\B_p^{\Phi_2}}
&\le 2  c_{\eqref{eqn:statement:BSDE_decoupling},p}
     \left [ |g|_1 +  |f|_{1;x} \right ]
     | X |_{\B_p^{\Phi_2,*}} \sptext{1.4}{for}{1} p\in [2,\infty),\\
  | Y |_{\B_{2,\infty}^{1,*}}
      + | Z |_{\B_{2,\infty}^1}
&\le 2  c_{\eqref{eqn:statement:BSDE_decoupling},2}
     \left [ |g|_1 +  |f|_{1;x} \right ]
     | X |_{\B_{2,\infty}^{1*}},\\
        | Y |_{\B_{p,q}^{\eta,*}}
      + | Z |_{\B_{p,q}^\eta}
&\le 2  c_{\eqref{eqn:statement:BSDE_decoupling},p}
     \left [ |g|_\alpha +  |f|_{\alpha;x} \right ]
     | X |_{\B_{\alpha p,\alpha q}^{\frac{\eta}{\alpha},*}}^\alpha  \sptext{1}{for}{1} p\in \left [\frac{2}{\alpha},\infty \right ),
\end{align*}
provided the corresponding RHS is finite and the quantities to measure the $Z$-process are defined in full
accordance to those for the $Y$ process, i.e. we let
\smallskip

\begin{align*}
     | Z |_{\B_p^{\Phi_2}} 
& := \sup_{0\le a < c \le T} \frac{1}{\sqrt{c-a}} \left \| \left ( \int_t^T |Z^{(a,c]}_s - Z_s|^2 \od s \right )^\frac{1}{2}
     \right \|_{L_p},\\
     | Z |_{\B_{p,q}^\theta}
&:= \begin{cases}
    \inf \| \cK_\theta \kappa \|_{L_q([0,1],\mu)} &: (\theta,q)\in (0,1)\times [1,\infty)\\ 
    \sup_{r\in (0,1]} \frac{1}{r^\theta}\| ( \int_t^T |Z^\br_s - Z_s|^2 \od s )^\frac{1}{2} \|_{L_p}
                                                  &: (\theta,q)\in (0,1]\times \{\infty\}
    \end{cases},
\end{align*}
\smallskip

where the infimum is taken over all measurable $\kappa:[0,1]\to \infty$ such that one has
 $\| ( \int_t^T |Z^\br_s - Z_s|^2 \od s )^\frac{1}{2} \|_{L_p} \le \kappa(r)$. The factor 2 in the inequalities above
originates from the fact that the two terms on the LHS are estimated separately.
\medskip

The table below
recalls the meaning of the different Besov spaces and the
respective statements from \cref{sec:Lipschitz}:
\bigskip

\begin{tabular}{|l|l|l|} \hline
space               & interpretation                     & statement to estimate the RHS \\ \hline \hline
$\B_p^{\Phi_2,*}$   & $\sup_{0\le a < c \le T}\left \| \left ( \frac{1}{c-a} \int_a^c |D_u A|^2 \od u \right )^\frac{1}{2} \right \|_{L_p}$
                    & \cref{statement:Malliavin_differentiability} for $| X |_{\B_p^{\Phi_2,*}}$ \\
                    & for the Malliavin derivative $D A$&\\  \hline
$\B_{2,\infty}^{1*}$& Malliavin derivative in $L_2(\Omega\times [0,T])$
                    & \cref{statement:Malliavin_differentiability_II} for $|X|_{\B_{1,\infty}^{1,*}}$ \\ \hline
$\B_{p,q}^{\eta,*}$ & real Malliavin interpolation space
                    & \cref{statement:real_interpolation_SDE_Lipschitz} for $ | X |_{\B_{\alpha p,\alpha q}^{\frac{\eta}{\alpha},*}}^\alpha$ \\ \hline
\end{tabular}
\bigskip
\bigskip

Next we consider the $L_p$-variation of the BSDE. The general statement is as follows:
\medskip

\begin{theo}
\label{statement:BSDE_variation}
Assume $\alpha \in (0,1]$, \cref{ass:BSDE}, and $t \le a < c \le T$. Then we get
\begin{align}
&     \hspace{-4em} \left \| \sup_{s\in [a,c]} |Y_s - Y_a| \right \|_{L_p} + \left \| \left ( \int_a^c |Z_s|^2 \od s \right )^\frac{1}{2} \right \|_{L_p} \notag \\
& \le \left \| \int_a^c |f(s,X,0,0)| \od s \right \|_{L_p} + L_Y (c-a) \sup_{s\in [0,T]} \| Y_s\|_{L_p} \label{eqn:statement:BSDE_variation} \\
&     \hspace{1em} + c_\eqref{eqn:statement:BSDE_variation} \left [ |g|_\alpha + |f|_{\alpha;x}  \right ]  
       \left \| \sup_{s\in [t,T]} |X_s^{(a,c]}-X_s| \right \|_{L_{\alpha p}}^\alpha. \notag
\end{align}
with $c_\eqref{eqn:statement:BSDE_variation}=c_\eqref{eqn:statement:BSDE_variation}(T,L_Y,L_Z,(Z_s)_{s\in [0,T]},\delta,p)>0$.
\end{theo}
\medskip

\begin{proof}
From \cite[Theorem 6.24]{Geiss:Ylinen:21} and \cref{statement:transference_composition_new} we get that
\begin{align*}
&     \left \| \sup_{s\in [a,c]} |Y_s - Y_a| \right \|_{L_p} + \left \| \left ( \int_a^c |Z_s|^2 \od s \right )^\frac{1}{2} \right \|_{L_p} \\
& \le \left \| \int_a^c f(s,X,0,0) \od s \right \|_{L_p}
      + L_Y (c-a) \sup_{s\in [0,T]} \| Y_s\|_{L_p}
      + c \Bigg [ \left \| g(X) -  g(X^{(a,c]}) \right \|_{L_p} \\
&     \hspace*{12em}    + \left \| \int_a^T \left |f(s,X,Y_s,Z_s) -  f(s,X^{(a,c]},Y_s,Z_s)
                      \right | \od s \right \|_{L_p} \Bigg ] \\
& \le \left \| \int_a^c |f(s,X,0,0)| \od s \right \|_{L_p}
      + L_Y (c-a) \sup_{s\in [0,T]} \| Y_s\|_{L_p} \\
& \hspace{16em} + c \left [ |g|_\alpha + |f|_{\alpha;x}  \right ] 
             \left \| \sup_{s\in [t,T]} \left |X_s -  X_s^{(a,c]}\right | \right \|_{L_{\alpha p}}^\alpha.
\end{align*}

\end{proof}
\bigskip

\begin{coro}
\label{statement:variation_bsde_CIR}
Let $d=1$. We consider the CIR-process
\[ X_s = x_0 + \int_t^s (A-B X_u) \od u + \sigma \int_t^s \sqrt{|X_u|} \od B_s \]
for some $x_0,A,B\in \R$ and $\sigma>0$. Let $p\in [2,\infty)$, assume \cref{ass:BSDE}, and  let $\alpha \in (0,1]$
and $\varepsilon \in \left ( 0, \frac{1}{2p} \right )$. Then one gets
\begin{align}
&    \hspace{-4em}
     \left \| \sup_{r\in [a,c]} |Y_r - Y_a| \right \|_{L_p} + \left \| \left ( \int_a^c |Z_r|^2 \od r \right )^\frac{1}{2} \right \|_{L_p} \notag \\
&\le \left \| \int_a^c |f(s,X,0,0)| \od s \right \|_{L_p} + L_Y (c-a) \sup_{s\in [0,T]} \| Y_s\|_{L_p}  \notag \\
&    \hspace{2em} + c_\eqref{eqn:statement:BSDE_variation} c_\eqref{eqn:statement:variation_bsde_CIR}
     \left [ |g|_\alpha +  |f|_{\alpha;x} \right ]
      \begin{cases}
          (c-a)^\frac{\alpha}{2} & p\in \left (0,\frac{1}{\alpha} \right ) \\
          (c-a)^{\frac{1}{2p} -\varepsilon} & p\in \left [ \frac{1}{\alpha},2 \right ) \\
          (c-a)^\frac{1}{2p} & p\in \left [ \frac{2}{\alpha},\infty \right ) \label{eqn:statement:variation_bsde_CIR}\\
          \end{cases}
\end{align}
where
$c_\eqref{eqn:statement:variation_bsde_CIR}=
 c_\eqref{eqn:statement:variation_bsde_CIR}(T,L_b,L_\sigma,K_b,K_\sigma,\alpha,p,x_0,\varepsilon)>0$.
\end{coro}

\begin{proof}
First we apply \cref{statement:BSDE_variation} and then 
\cref{statement:holder:main} for $\theta=\frac{1}{2}$ and $\alpha p$ instead of $p$.
\end{proof}
 
\begin{rema}
For \cref{statement:variation_bsde_CIR} we note that $\sup_{s\in [0,T]} \| Y_s\|_{L_p} <\infty$ due to 
\eqref{eqn:Lp_boundedness_YZ}. Moreover, one can use  
\[ \int_a^c |f(s,X,0,0)| \od s \le (c-a) \sup_{s\in [0,T],x\in \R} |f(s,x,0,0)| \]
in case the generator is uniformly bounded, or
\[ \left \| \int_a^c |f(s,X,0,0)| \od s \right \|_{L_p}
   \le \sqrt{c-a} \left \| \sqrt{\int_0^T |f(s,X,0,0)|^2 \od s} \right \|_{L_p} \]
if $\sqrt{\int_0^T |f(s,X,0,0)|^2 \od s} \in L_p$. In both cases,
the dominating term on the RHS in \cref{statement:variation_bsde_CIR} is the last term which comes
from \cref{statement:holder:main}.
\end{rema}


\bibliographystyle{amsplain}

\section*{Acknowledgement} 
We thank Hannah Geiss for the careful reading of the manuscript and her comments that improved the manuscript.

\end{document}